\documentclass[10pt,oneside,a4paper,draft]{article} 				

\usepackage[fleqn]{amsmath} 										
\usepackage{amsthm,amsfonts,latexsym,amssymb,amscd} 				
\usepackage[mathscr]{eucal}   										
\usepackage[all,2cell]{xy} \UseAllTwocells 							
\usepackage{framed}													
\usepackage{color}													
\usepackage{txfonts}												

\definecolor{cmyk}{cmyk}{.0,.4,.9,.5}					

\usepackage[draft=false]{hyperref} 						

\renewcommand{\H}{\mathcal{H}} 										
 
\renewcommand{\O}{\mathcal{O}}										
\renewcommand{\P}{\mathcal{P}}										
 
\newcommand{\Rg}{\mathfrak{R}}

\newcommand{\CC}{{\mathbb{C}}}

\newcommand{\KK}{{\mathbb{K}}}
\newcommand{\MM}{{\mathbb{M}}}
\newcommand{\NN}{{\mathbb{N}}}

\newcommand{\RR}{{\mathbb{R}}} 

\newcommand{\ZZ}{{\mathbb{Z}}}

\newcommand{\As}{{\mathscr{A}}}\newcommand{\Bs}{{\mathscr{B}}}\newcommand{\Cs}{{\mathscr{C}}}
\newcommand{\Ds}{{\mathscr{D}}}
\newcommand{\Es}{{\mathscr{E}}}\newcommand{\Fs}{{\mathscr{F}}}

\newcommand{\Is}{{\mathscr{I}}}\newcommand{\Ks}{{\mathscr{K}}} 
\newcommand{\Ms}{{\mathscr{M}}}\newcommand{\Ns}{{\mathscr{N}}}\newcommand{\Os}{{\mathscr{O}}}
\newcommand{\Ps}{{\mathscr{P}}}										
\newcommand{\Qs}{{\mathscr{Q}}}\newcommand{\Rs}{{\mathscr{R}}}
\newcommand{\Ts}{\mathscr{T}} 										

\newcommand{\Xs}{{\mathscr{X}}}
\newcommand{\Ys}{{\mathscr{Y}}}

\DeclareFontFamily{U}{rsfs}{\skewchar\font127 }
\DeclareFontShape{U}{rsfs}{m}{n}{%
   <5> <6> rsfs5
   <7> rsfs7
   <8> <9> <10> <10.95> <12> <14.4> <17.28> <20.74> <24.88> rsfs10
}{}
\DeclareSymbolFont{rsfs}{U}{rsfs}{m}{n} 
\DeclareSymbolFontAlphabet{\scr}{rsfs}								

\newcommand{\Af}{\scr{A}}\newcommand{\Cf}{\scr{C}}
\newcommand{\Hf}{\scr{H}}

\newcommand{\Mf}{\scr{M}}\newcommand{\Pf}{\scr{P}}
\newcommand{\Sf}{\scr{S}}
\newcommand{\Xf}{\scr{X}}

\DeclareMathOperator{\Hom}{Hom}

\DeclareMathOperator{\hatcirc}{\hat{\circ}}
\DeclareMathOperator{\hatplus}{\hat{+}}
\DeclareMathOperator{\hatcdot}{\hat{\cdot}}

\renewcommand{\emph}{\textbf} 										
\newcommand{\cj}[1]{\overline{#1}}									
\newcommand{\ip}[2]{\langle #1\mid #2\rangle}						
\renewcommand{\iff}{\Leftrightarrow}								
\newcommand{\imp}{\Rightarrow}										

\newcommand{\hlink}[2]{\href{#1}{\texttt{#2}}} 						

\newcommand{\xqedhere}[2]{%
  \rlap{\hbox to#1{\hfil\llap{\ensuremath{#2}}}}}

\newcommand{\xqed}[1]{%
  \leavevmode\unskip\penalty9999 \hbox{}\nobreak\hfill
  \quad\hbox{\ensuremath{#1}}}

\theoremstyle{plain}
\newtheorem{theorem}{Theorem}[section]								

\newtheorem{proposition}[theorem]{Proposition}

\newtheorem{definition}[theorem]{Definition}

\theoremstyle{definition} 
\newtheorem{remark}[theorem]{Remark}
\newtheorem{example}[theorem]{Example}  

\numberwithin{equation}{section}  									
\title{\textbf{On Strict Higher C*-categories}}

\author{\normalsize  
Paolo Bertozzini $^a$\footnote{Independent researcher - Bangkok (since July 2024).}  \  \  
Roberto Conti $^b$  \ 
Wicharn Lewkeeratiyutkul $^c$ \ 
Noppakhun Suthichitranont $^d$
\\  
\normalsize $^a$ \textit{Department of Mathematics and Statistics, Faculty of Science and Technology,}
\\
\normalsize \textit{Thammasat University, Pathumthani 12121, Thailand}
\\
\normalsize e-mail: \texttt{paolo.th@gmail.com} 
\\ 
\normalsize $^b$ \textit{Dipartimento di Scienze di Base e Applicate per l'Ingegneria,} 
\\
\normalsize \textit{Sapienza Universit\`a di Roma, Via A. Scarpa 16, I-00161 Roma, Italy}
\\
\normalsize e-mail: \texttt{roberto.conti@sbai.uniroma1.it} 
\\
\normalsize $^c$ \textit{Department of Mathematics and Computer Science, Faculty of Science,}
\\
\normalsize \textit{Chulalongkorn University, Bangkok 10330, Thailand} 
\\ 
\normalsize e-mail: \texttt{Wicharn.L@chula.ac.th} 
\\
\normalsize $^d$ \textit{393 Putthamonthon 3, soi 18/4, Sala Thammasop, Taweewattana, Bangkok 10170, Thailand}
\\
\normalsize e-mail: \texttt{noppakhuns@hotmail.com}
}

\date{\normalsize{
		first version: 28 July 2014 \quad revised: 12 August 2016 / 14 June 2017 \quad finalized: 27 September 2017}  corrected typos: 05 October 2017 \quad submitted: 07 October 2017 \quad referee report: 04 February 2019 \\ second revision: 09 June 2019 \quad  resubmitted: 14 June 2019 \quad third revision: 01 July 2019 
		\quad Important notes in footnotes \ref{foo: norms} and \ref{foo: norms2} added: 29 October 2025}

\begin{document}

\maketitle

\begin{center}
\textit{Dedicated to the memory of John E.Roberts who gave us C*-categories.} 
\end{center} 

\begin{abstract} \noindent 
We provide definitions for strict involutive higher categories (a vertical categorification of dagger categories), strict 
higher C*-categories and higher Fell bundles (over arbitrary involutive higher topological categories).  
We put forward a proposal for a relaxed form of the exchange property for higher \hbox{(C*)-categories} that avoids the Eckmann-Hilton collapse and hence allows the construction of explicit non-trivial ``non-commutative'' examples arising from the study of hypermatrices and hyper-C*-algebras, here defined. 
Alternatives to the usual globular and cubical settings for strict higher categories are also explored. 
Applications of these non-commutative higher C*-categories are envisaged in the study of morphisms in non-commutative geometry and in the algebraic formulation of relational quantum theory.  

\medskip


\smallskip

\noindent
\emph{MSC-2010:} 					
					18D05,			
					46M15, 			
					46M99,			
 					16D90.			


\end{abstract}

\tableofcontents


\section{Introduction} 

The usage of categorical methods in functional analysis is probably going back to A.Grothendieck and F.Linton~\cite{Li}, but category theory began to be applied to the theory of operator algebras in the seventies, with the pioneering work of J.E.Roberts~\cite{GLR} that introduced the definition of C*-categories mainly in view of applications to algebraic quantum field theory. Since then C*-categories have been extensively used by J.E.Roberts and S.Doplicher~\cite{DR} in the theory of superselection sectors in algebraic quantum field theory (see R.Haag's monograph~\cite{H} and H.Halvorson-M.M\"uger's review~\cite{HM}). 
Operator categories and \hbox{C*-(tensor) categories} have been further significantly studied by S.Yamagami~\cite{Y0,Y1,Y}, P.Mitchener~\cite{M}, T.Kajiwara-C.Pinzari-Y.Watatani~\cite{KPW}, M.M\"uger~\cite{Mu,Mu2}, and more recently A.Henriquez-D.Penneys~\cite{HP} and C.Jones-D.Penneys~\cite{JP} to name just a few. 
The closely related and more general notion of a Fell bundle over a topological group was first defined by J.Fell~\cite{FD} (under the name of a Banach $*$-algebraic bundle) and later extended respectively to: topological groupoids, by S.Yamagami and then A.Kumjian~\cite{K}; topological inverse semigroups, by N.Seiben (see R.Exel~\cite{Ex}); and topological involutive inverse categories in~\cite{BCL2,BCL4}. 

\medskip 

The study of higher $n$-categories, at least in their strict versions, can be traced back to the work of C.Ehresmann~\cite{E} on structured categories (the notion of $\infty$-groupoid predating actually to a paper of D.Kan on simplicial complexes~\cite{Ka}). Strict $\omega$-categories were proposed by J.E.Roberts~\cite{R} (in his work on local cohomology in algebraic quantum field theory) and were independently developed by R.Brown-P.Higgins~\cite{BH} and A.Grothendieck~\cite{G} with strong motivations from homotopy theory in algebraic topology.  
Weak higher category theory, starting from the notion of bicategory of J.Benabou~\cite{Be}, subsequently formalized by R.Street~\cite{St}, recently developed into a wide area of extremely active research (see for example as references T.Leinster~\cite{Lsur, L}, E.Cheng-A.Lauda~\cite{CL} and the wiki-resources at \hlink{http://ncatlab.org/nlab}{http://ncatlab.org/nlab}).  

\medskip 

Surprisingly, despite their quite close initial historical developments and the current widespread usage of categorical methods/techniques in recent research,\footnote{
As can be seen for example in the works by J.E.Roberts, G.Ruzzi, E.Vasselli~\cite{V1,RRV,RRV2} in algebraic quantum field theory; by S.Abramski, B.Coecke, C.Heunen, M.L.Reyes and their collaborators~\cite{AC,AC2,CP,Coe,AH,HR} in categorical quantum mechanics; and by A.Buss-R.Meyer-C.Zhu~\cite{BMZ1,BMZ2}, S.Mahanta~\cite{Ma}, among many others, in non-commutative topology and geometry.} 
a satisfactory interplay between higher categories and operator algebra theory has never been achieved and higher category theory has more recently evolved along lines, much closer to classical higher homotopy, that in our opinion further prevent a direct interaction between the two subjects. 

\medskip 

Although monoidal C*-categories (2-C*-categories with one object) have been systematically used since the inception of the theory of superselection sectors~\cite{DR,H}, a first notion of 2-C*-category appears only in the paper by R.Longo-J.E.Roberts~\cite{LR} and the topic has been later reconsidered by P.Zito~\cite{Z}. 
Following the studies on the generalization of V.Jones' index theory of subfactors via Q-systems, as in R.Longo~\cite{Lo1,Lo2}, and its relations with low dimensional superselection theory~\cite{KLM} (see for example the expositions in Y.Kawahigashi~\cite{Kaw}, K.-H.Rehren~\cite{Re} and the references in~\cite{BLKR}), an enormous body of research in conformal field theory has been systematically using several variants of C*-tensor categories and 2-C*-categories as can be seen in the recent works by A.Bartels-C.L.Douglas-A.Henriques-C.Jones-D.Penneys-J.Tener~\cite{He0,He1,He2}~\cite{BDH}~\cite{HP,HPT1,HPT2,JP} among others.  
Bicategories (weak 2-categories) of von Neumann algebras have also been investigated by N.Landsman~\cite{Lan},  R.M.Brouwer~\cite{Br} and Y.Sawada-S.Yamagami~\cite{Y,Y2}.  
Anyway, no hint of operator algebraic structures capable of climbing up, in a non-trivial way, the ladder of \hbox{$n$-C*-categories} for $n$ bigger than 2 has been produced,\footnote{With the possible exception of the 3-categories recently introduced by A.Bartel-C.L.Douglas-A.Henriques~\cite{BDH}.} so that higher C*-categories, and with them the full development of a comprehensive theory of ``higher functional analysis'', have remained elusive.  
We announced a tentative definition of strict (globular) $n$-C*-category (still based on the usual axioms for strict higher categories) in~\cite[section~4.2.2]{BCL1} with details in~\cite[section~3.3]{BCL3} and in the following paper we propose a much wider notion of strict $n$-C*-category, able to encompass several quite interesting non-trivial and natural examples of non-commutative operator theoretic constructs.

\medskip 

The well-known term \emph{categorification}, introduced in L.Crane-D.Yetter~\cite{CY}, is informally used to denote any ``mathematical process'' in which set-theoretic structures get replaced by ``categorical versions'' (see for example the discussion in J.Baez-J.Dolan~\cite{BD}). 
In this work we make systematic use of a more specific terminology, that we introduced for example in~\cite[section~4.2]{BCL1}, distinguishing between a \emph{horizontal categorification} (also called ``oidification'' or ``many-objectification'') that is the process consisting in replacing one-object-structures with their many-object versions (adding objects and morphisms between them) as for example in: 
\begin{align*}
& 
[\text{monoid}]\mapsto[\text{category}], 
& &
[\text{group}]\mapsto[\text{groupoid}], 
& &
[\text{ring}]\mapsto[\text{ringoid}], 
\\ 
& 
[\text{algebra}]\mapsto[\text{algebroid}], 
& &
[\text{C*-algebra}]\mapsto[\text{C*-category}], 
& & 
\end{align*}
and a \emph{vertical categorification} process, that properly consists in adding further higher-level morphisms as in the following examples:
\begin{align*}
& 
[\text{set}]\mapsto[\text{category}]\mapsto\cdots\mapsto[\text{$n$-category}], 
& & 
[\text{set}]\mapsto[\text{groupoid}]\mapsto\cdots\mapsto[\text{$n$-groupoid}], 
\\
& 
[\text{algebroid}]\mapsto\cdots\mapsto[\text{$n$-algebroid}], 
& & 
[\text{C*-category}]\mapsto\cdots\mapsto[\text{$n$-C*category}]. 
\end{align*}

\medskip 

In very general terms, the efforts presented here can be seen as a first attempt for the development of a full \textit{vertical categorification} of functional analysis and operator algebra, in the same way as the transition from C*-algebras to \hbox{C*-categories} can be a \textit{horizontal categorification} of functional analysis. 

\medskip 

We stress, as a disclaimer, that the main inspiration for our proposed set of C*-categorical axioms stems from the attempt to vertically categorify Gel'fand-Na\u\i mark dualities; in particular it is not our intention here to define ``higher C*-categorical settings'' for a discussion of Tannaka-Krein dualities. 

\medskip 

Before entering into the description of the actual content of the paper, we would like to devote a few more lines of motivation for the reader that might feel quite uncomfortable with the prospect of having to slog through the long series of definitions and preparatory material here below, without some clear indications that this might be worth and justified. 
\begin{itemize}
\item 
Since the relevance of 1 and 2-categorical C*-structures in the formalism of quantum field theory is now indisputable, it is quite natural to investigate if there is any additional role for higher \hbox{C*-categories}. Unfortunately, in the current literature, for now, even basic definitions of these (strict) higher structures are missing and hence we decided to make a first effort in this direction filling the gap (at least for strict C*-categories) and propose tentative definitions. 
A specific motivation is the desire to place the classical works on C*-categories by S.Doplicher-J.E.Roberts into their wider higher-category-theoretic context and to have a more systematic study of the role, and generalization, of the (strict) involutions over objects in 2-C*-categories that already appear in R.Longo-J.E.Roberts~\cite{LR}, P.Zito~\cite{Z}, and more recently in A.Henriques-D.Penneys~\cite{HP}, C.Jones-D.Penneys~\cite{JP}, L.Giorgetti-R.Longo~\cite{GL}, and that also play a role in the theory of representations of quantum groups in the works of C.Pinzari-J.E.Roberts~\cite{PR0}. 
\item 
Suitable (weak) tensor-2-C*-categories are systematically used in conformal field theory (see for example the review by M.Bischoff-R.Longo-Y.Kawahigashi-K.-H.Rehren~\cite{BLKR}) and certain \hbox{3-cat}\-egories have been very recently introduced by A.Bartels-C.L.Douglas-A.Henriquez~\cite{BDH}. At the very minimum, one might want to see if involutions over 1-arrows and over objects play any role there and in which sense higher C*-categorical axioms hold. 
\item 
From the strictly mathematical point of view, it is reasonable to ask (and the answer is far from obvious) if the $n$-categorical structures widely used in the literature (with motivations typically related to topological higher homotopy theory) are compatible and in which way with the further requirements imposed by the operator algebraic world. 
This is very far from being just an academic exercise: the study of C*-algebras is (in a widely accepted way among the practitioners of operator algebras) considered as the study of ``non-commutative topology'' and hence it would only be natural to ask (at least from the point of view of a ``non-commutative topologist'') if and how higher homotopy theory can be formulated for non-commutative spaces and so, dually, in C*-categorical language.\footnote{
The main reasons for the complications encountered are mostly due to the extra ``geometric rigidity'' implicit in C*-algebras, that secretly are (non-commutative) uniform spaces rather than just topological spaces (such input comes directly from recent study of non-commutative extensions of Gel'fand-Na\u\i mark duality for unital C*-algebras). 
}
\item 
The study of dynamical systems, defined as action of groups on topological spaces (and dually on C*-algebras), has been also extended to actions of 2-groups (crossed modules) and \hbox{2-cat}\-egories on C*-algebras, for example in the works of A.Buss-R.Meyer-C.Zhu~\cite{BMZ1,BMZ2}. 
It is unlikely that (fully) involutive 2-categories (as here defined in section~\ref{sec: *}) will not play a significant role in the study of representations of 2-categories on C*-algebras. 

Since $n$-groupoids, for $n>2$ are a quite well-known, in the spirit of vertical categorification, one might explore if such higher groupoids (and more generally fully involutive $n$-categories) have associated ``higher C*-dynamical systems'' and we guess that higher-C*-categories provide an adequate minimal environment for such theories: in particular one might study, in the same spirit of the previous works by R.Buss-R.Meyer-C.Zhu, the role of higher actions of involutive $n$-categories on hyper-C*-algebras. (work is ongoing on these topics). 
\item 
In previous works~\cite{BCL2,BCL4}, we have already examined a horizontal categorification (oidification) of Gel'fand-Na\u\i mark duality, where commutative C*-algebras are replaced by commutative full C*-categories. It is perfectly justified to ask if higher categories might add further information and if such duality survives (in which form) a vertical categorification.\footnote{Although in this paper, we will not enter into the discussion of spectral theory for strict higher C*-categories, under similar commutativity and fullness conditions, one can perform a vertical categorification of Gel'fand-Na\u\i mark duality using ``higher spaceoids'' that are just special one-dimensional examples of the $n$-Fell bundles here defined.} 
\item 
Currently there are several lines of development in algebraic quantum field theory where the \hbox{``$n$-categorical language''} starts to appear explicitly in attempts to formulate local gauge theories, either via ``net cohomology'' in the many works by J.E.Roberts-G.Ruzzi-E.Vasselli~\cite{R,R2,RR,RRV,RRV2}; via ``operads'' in the recent studies by M.Benini-A.Schenkel-L.Woike~\cite{BSc,BSW1}. It is a safe bet that higher involutions and higher C*-categories (possibly in their future ``weak-versions'') will provide useful technical ingredient in such investigations. 

In section~\ref{sec: *} of present work, we propose a full vertical categorification of the notion of dagger category; these (strict) involutive higher categories, as a close generalization of strict higher groupoids, should provide a wider playground for higher gauge theory (see for example J.Baez-J.Huerta~\cite{BHu}, U.Schreiber~\cite{Sc4}) and ``(higher) transports'' (see U.Schreiber-K.Waldorf~\cite{SW}). 
\end{itemize} 

A much more speculative, but quite deep motivation, for the study of fully involutive higher \hbox{C*-cat}\-e\-gories (very likely in their future weak incarnations, see~\cite{BB}) should come from ``(extended) functorial quantum field theories'' and ``homotopy/homology theoretic approaches'' to quantum field theory, especially if we desire to formalize a theory based on ``non-commutative geometrical entities''. 

\medskip 

The present functorial approaches to quantum field theory, either via topological quantum field theory by M.Atiyah~\cite{At} (see also the review by F.Quinn~\cite{Q}) or via conformal field theory by G.Segal~\cite{Seg} (see also A.Henriques~\cite{He0}), including their extended higher-categorical versions (see for example D.Freed~\cite{Fr}, J.Baez-J.Dolan~\cite{BD1}, S.Stolz-P.Teichner~\cite{ST} and J.Lurie~\cite{Lu}) and their possible utilization in quantum gravity (as suggested by L.Crane~\cite{Cra2}, J.Baez~\cite{Ba}, J.Morton~\cite{Mo}), are all based on the existence of suitable (higher) \textit{quantization} functors, from \textit{classical geometric} categories of (higher) cobordisms of certain manifolds, to \textit{quantum algebraic} categories of morphisms of (higher) Hilbert spaces. 

\medskip 

In a rather similar way, all the recent categorical reformulations/generalizations of algebraic quantum field theory, as in~R.Brunetti-K.Fredenhagen-R.Verch~\cite{BFV} (see U.Schreiber~\cite{Sc0} for relations between the algebraic and the functorial approach) are considering functors defined on categories of \textit{geometric spaces} (usually isometries of globally hyperbolic Lorentzian manifolds)\footnote{Some categories of non-commutative geometric spaces have actually been considered by M.Paschke-R.Verch~\cite{PV}.} with values in categories of (usually unital $*$-homomorphisms of) \textit{operator algebras}. 

\medskip 

\begin{itemize}
\item 
In both cases, the study of the intrinsic features of higher (involutive) operator algebraic categories should be of crucial relevance in order to produce suitable target categories for the above (higher) functors,\footnote{Extended functorial quantizations will have as target higher categories of representations of $n$-C*-categories.} avoiding the ``extra commutativity assumptions'', implicit in the specification of classical geometrical (higher) categories of spaces and cobordisms, that might impose too restrictive requirements on the algebraic/quantum target-side of such functors. The latter is a very concrete danger in any rigorous approach to quantum field theory, and especially quantum gravity, if the nature of space-time must be non-commutative (as often suggested for the purpose to eliminate singularities and divergences). 

\medskip 

In particular, current ``stabilization hypotheses'' in functorial quantization, based on Eckmann-Hilton's argument, as suggested in J.Baez-J.Dolan~\cite[section~5]{BD1}, might turn out to be too strong (the non-commutative exchange property, here proposed in section~\ref{sec: hc}, should make it easier to consider ``non-commutative cobordisms'').  
\item 
Starting from the pioneering work of J.Baez-J.Dolan~\cite{BD1}, there are strong indications that (higher) involutions play an essential role in characterizing the \textit{quantum} higher target-categories of an extended quantization functor and, 
the efforts here presented in section~\ref{sec: *} (at least in the case of strict involutions, instead of dualities) might constitute a first modest contribution in the study and understanding of such higher involutions, as explicitly invoked several years ago by J.Baez-M.Stay~\cite[section~2.7]{BS},~\cite{Ba}. 
\item 
As already suggested in ``modular algebraic quantum gravity''~\cite[section~6.3, page~32]{BCL1}, contrary to the usual assumptions of functorial and algebraic quantum field theories, we are not looking for a (higher) functor from categories of \textit{commutative geometries} to categories of \textit{quantum observables}, but rather, in the reverse direction, we are trying recover geometries (likely non-commutative) from (higher) categories of quantum operational data and we need to define a \textit{``spectral'' functorial quantum field theory}  from ``C*-operator categories'' to ``geometry''. 
In this case, the Eckmann-Hilton collapse might essentially suppress important classes of higher operator categories and all non-trivial higher morphisms of non-commutative spaces. 
\end{itemize}

As further stressed in the outlook section, the full ``justification'' of this work (and of the specific axioms proposed for ``quantum'' strict higher C*-categories) is not only coming from the existence of some non-trivial examples (although that would probably be already sufficient); but from more substantial ideological inter-related requests concerning the investigation of the nature of ``morphisms of non-commutative spaces'', the formalization of an operational ``relational quantum systems theory'' and from the above-mentioned ongoing attempts in the direction of ``modular algebraic quantum gravity''. 

\begin{itemize}
\item 
Most of the current notions of morphism of non-commutative spaces (usually consisting of suitable bimodules, possibly equipped with additional structures) are extremely ``rigid'', and hence somehow not completely satisfactory (especially when one compares with the extreme ``morphic-freedom'' available in the case of cobordisms between usual manifolds).  
To address this issue, we put forward (see section~\ref{sec: app}), some conjectural ideas (motivated by recent spectral results on Gel'fand-Na\u\i mark duality for non-commutative C*-algebras) that support the claim that a treatment of sufficiently general classes of morphisms of non-commutative spaces, might require the introduction of ``higher-bimodules'' i.e.~``representation spaces'' of higher-C*-categories and hyper-C*-algebras, similar to those here introduced. 
\end{itemize}

``Categorical quantum theory'' has been quite successfully developed, in a long series of works by S.Abramsky-B.Coecke-C.Heunen~\cite{AC,AC2,AH}, P.Selinger~\cite{S} and their many collaborators, using the framework of compact symmetric monoidal dagger categories (or, as in J.Vicary~\cite{Vi}, certain symmetric monoidal \hbox{2-categories} motivated by J.Baez's 2-Hilbert spaces~\cite{Ba1}). 
In this approach, quantum processes (channels) and their compositional 2-categorical structure are primary. 

\begin{itemize}
\item 
One of us has proposed the still rather speculative hypothesis that a mathematical formalization of relational quantum theory, along the lines suggested by C.Rovelli~\cite{Rov}, might be achieved via higher C*-categories~\cite{B}.

The usage of C*-categories in foundations of quantum theory and operational quantum theory might be considered suspicious (since a C*-algebra, via Gel'fand-Na\u\i mark representation theorem, implicitly contains all the mathematical ingredients that allow to reconstruct the usual Hilbert space picture of quantum mechanics, and the formalism of C*-algebras has never been completely justified on clear operational grounds~\cite{Str}). Anyway, on the basis of current work on non-commutative Gel'fand-Na\u\i mark duality, we have good reasons now to assert (see the \textit{spectral conjecture} put forward in the outlook section~\ref{sec: app}) that non-commutative C*-algebras can be operationally motivated, via ``convolutions algebras of certain relations'' between spectra of observables, following original suggestions by W.Heisenberg and J.Schwinger, as further elaborated by A.Connes~\cite[chapter~1, section~1]{C} (see also the very recent work by  F.M.Ciaglia-A.Ibort-G.Marmo~\cite{CIM}). Higher categorical levels might become relevant as soon as ``nested chains of observers'' are allowed and (as in relational quantum theory) states are observer-dependent. 
\end{itemize}

Natural classes of non-trivial examples usually help to support the introduction of new axioms and, in the case of quantum higher C*-categories with their convolution hyper-C*-algebras, we can easily provide them:  

\begin{itemize}
\item 
Whenever a Hilbert space factorizes $\H=\bigotimes_{\lambda\in\Lambda}\H_\lambda$, the topological algebra $\Ks(\H)$ of compact operators factorizes as well $\Ks(\H)=\bigotimes_{\lambda\in\Lambda}\Ks(\H_\lambda)$ and naturally becomes equipped with several mutually compatible multiplications, involutions and norms. Although such Hilbert factorizations are quite rare in algebraic quantum field theory (split inclusions), they are quite ubiquitous in quantum information theory (where $\H$ is often finite dimensional) and produce a huge family of examples of ``hyper-C*-algebras'' and ``higher C*-categories'' of hypermatrices that, in section~\ref{sec: hC*}, will be introduced and studied in detail.\footnote{These natural examples of strict higher C*-categories have been considered in the very early stages of this research~\cite{Su} and they were erroneously discarded, apart from the trivial commutative cases, exactly because they generally failed to satisfy the familiar exchange property for higher categories. 
} 
\item 
Standard examples, from ``representation theory'', are obtained considering any one-dimensional (fully involutive) higher C*-category in place of the usual field of complex numbers and defining higher Hilbert spaces as higher \hbox{C*-modules} over them (exactly as usual Hilbert spaces are \hbox{C*-modules} over the C*-algebra of complex numbers) and finally studying the higher \hbox{C*-cat}\-egories of ``endomorphisms'' of such Hilbert higher C*-modules and the vertical categorification of Gel'fand-Na\u\i mark representation theorem. Work in this direction is under development. 
\end{itemize}

\medskip 

We proceed now to describe in some detail the content of the paper. 

\medskip 

In section~\ref{sec: C*}, we briefly recall the main C*-algebraic definitions and results that constitute the background for our work. Here the reader who is not already familiar with operator algebras will find a detailed definition of C*-algebras, their horizontal categorification (C*-categories) and their ``bundlified'' generalizations (Fell bundles also on general \textit{involutive categories}) as well as previously available definitions of monoidal (tensor) C*-category (Doplicher-Roberts) and 2-C*-category (Longo-Roberts).  

\medskip 

In section~\ref{sec: hc} the standard notions of strict globular higher $n$-category are introduced making use of ``partial $n$-monoids'', an equivalent definition in term of properties of $n$ composition operations ($\circ_0,\dots,\circ_{n-1}$) partially defined on a family of $n$-cells. The Eckmann-Hilton collapse argument is presented in detail, explaining how it prevents any inclusion of non-commutative ``diagonal home-sets'' at depth higher than~$1$.  
In order to avoid this fatal degeneration (that is ultimately responsible for the lack of reasonable examples of higher C*-categories that exhibit non-commutative features), we propose here to relax the exchange property and substitute it with a weaker condition of left/right $\circ_p$-functoriality of the compositions by $\circ_q$-identities, for $0\leq q<p < n$. We are fully aware of the fact that this modified ``non-commutative'' exchange property is not fitting with most of the current developments in higher category theory, but we stress that ultimately its relevance in higher category theory will be vindicated by the abundance of quite natural examples available. In this same section, for later use, we also discuss examples of strict $n$-categories (mainly Cartesian products of 1-categories) whose $n$-cells naturally admit compositions that do not fit with the usual globular or cubical picture of strict higher $n$-categories now available: relaxing the exchange property not only allows more non-commutativity for the compositions, but also more freedom in the ``composability'' of cubical $n$-cells. 

\medskip  

In section~\ref{sec: *}, we describe a full vertical categorification of P.Selinger's dagger categories, via strict involutions defined as endo-functors that can be covariant or contravariant with respect to any of the partial compositions of a strict globular $n$-category. This is not the only way to introduce notions of ``duality'' for $n$-cells, but it is in perfect agreement with the tradition of J.E.Roberts' $*$-categories, where involutions are treated on the same footing as compositions. 
The resulting notion of a (partially/fully) involutive higher category should be interesting on its own. 
A much more detailed study of higher involutions for globular and cubical $n$-categories appears in our companion paper~\cite{BCM}.  

Although the introduction of involutions with mixed covariance properties might seem to invalidate the non-commutativity gained via the relaxed exchange property (see remark~\ref{rem: obs}), its effects still allow the existence of non-trivial non-commutative examples as long as the ``diagonal home-sets'' are equipped with ``more products/involutions'' as will be described in section~\ref{sec: hC*} (see proposition~\ref{prop: obs} and theorem~\ref{th: obs}). 

\medskip 

The definition of strict higher (globular) C*-categories rests on several additional pieces of structure that are considered in section~\ref{sec: hC*}.
As the first step, we define higher $*$-algebroids (of minimal depth) introducing complex linear structures on each family of globular $n$-arrows with a common \hbox{$(n-1)$-sources/targets} and imposing conditions of bilinearity for compositions and conjugate-linearity for involutions. 
This is just the easiest form of vertical categorification of the usual notion of $*$-category used by J.E.Roberts (and later reconsidered by P.Mitchener): in principle (as already suggested in the axioms presented in~\cite{BCL3}) one might provide, for all $0\leq p<n$, completely different linear structures on the sets of $n$-arrows with common source/targets at depth-$p$; for simplicity we decided to avoid here this further generalization, that will be discussed in more details elsewhere. 
Next, if a Banach norm is placed on each of the previous linear spaces one can impose suitable axioms of submultiplicativity for the compositions $\circ_0,\dots,\circ_{n-1}$ and \hbox{C*-norm} and positivity conditions for some or for all pairs $(\circ_p,*_p)$ of depth-$p$ composition/involution. 
In this way one obtains a vertical categorification of the notion of C*-category (of minimal linear depth) that can in principle be fully involutive and that   also generalizes Longo-Roberts 2-C*-categories, where only the $*_{n-1}$ involution is present. A definition of $n$-Fell bundle is easily obtained, whenever the pair $n$-groupoid $\Cs/\Cs$ of linear spaces in the strict globular fully involutive $n$-C*-category $\Cs$ is replaced by a more general fully involutive $n$-category. 

In order to provide many non-trivial examples of fully involutive strict higher C*-categories, always in 
section~\ref{sec: hC*}, we look at the usual algebra $\MM_{N\times N}(\CC)$ of square complex matrices of order $N$ as an algebra of sections of a Fell line-bundle over the pair groupoid of a set of $N$ points and we simply substitute $\CC$ with an arbitrary (possibly non-commutative) unital C*-algebra $\As$ and $N\times N$ with an arbitrary finite discrete (fully) involutive $n$-category $\Xs$. 
The Cartesian bundle $\Xs\times\As$ is a natural example of a strict globular $n$-category (see theorem~\ref{th: convo}) and the non-commutative exchange property is absolutely necessary to give ``citizenship rights'' to the structure in the case of non-commutative algebras $\As$. 
Whenever $\As$ is a commutative C*-algebra, $\Xs\times\As$ becomes an $n$-Fell bundle (a fully involutive $n$-C*-category, when $\Xs$ is a $n$-groupoid); unfortunately (as explained in remark~\ref{rem: obs} and proposition~\ref{prop: obs}) this result cannot hold for the case of non-commutative C*-algebras $\As$, but it can be recovered (see~theorem~\ref{th: obs}) with a more complex system of non-commutative coefficients in place of the C*-algebra $\As$.    

The resulting family of sections $\MM_\Xs(\As)$, the ``enveloping $n$-convolution algebra'' of the $n$-Fell bundle 
$\Xs\times \As$, is the first example of what we call a hyper-C*-algebra: a complete topolinear space equipped with $N$ different C*-algebraic structures $(\circ_p,*_p,\|\cdot\|_p)$, whose norms are equivalent. 
We finally provide further natural examples of such hyper-C*-algebras, via nested hypermatrices and we also show how these 
hyper-C*-algebras can be seen as higher-convolution algebras \dots\ as long as we allow cubical sets (in place of globular sets) and we consider, over the set of $n$-cells, $2^n$ different pairs $(\circ_p,*_p)$, $p=0,\dots,n-1$, of composition/involution. This is quite a strong hint for the relevance of higher categorical constructs that do not find place in present-day axiomatizations of $n$-categories and where the relevance of our non-commutative exchange property is even more evident. 

\medskip 

In section~\ref{sec: app} we informally discuss some wilder speculations on the possible applications of the formalism of non-commutative  higher C*-categories to non-commutative geometry and quantum theory. 
A quite strong motivation for the consideration of higher C*-categories comes from the need to formulate general categorical environments for non-commutative geometry. Morphisms between usual ``commutative'' spaces are given by families of 1-arrows (a relation or more generally a 1-quiver) connecting points of the spaces, so that ``dually'' a morphism corresponds to a bimodule over the commutative algebra of functions over the graph of a relation. In that ``classical'' context, as suggested in~\cite{BJ}, there is no problem at all in performing a vertical categorification. On the contrary,  vertical categorifications of morphisms between non-commutative spaces (dually described by bimodules over non-commutative algebras) are quite difficult to achieve, since the usual exchange property imposes strong commutativity conditions. Taking inspiration from our previous work on the spectral theorem for commutative full C*-categories~\cite{BCL2}, we are led to think of the spectrum of a non-commutative algebra as a ``family'' of Fell line-bundles  (spaceoids), so that morphisms of non-commutative spaces appear to be naturally described by 2-quivers with a cubical structure, and hence dually, by suitable higher bimodules. 

\medskip 

Since non-commutative spaces (in the language of A.Connes' spectral triples) are essentially very specific quantum dynamical systems, it does not come as a surprise that higher operator category theory becomes relevant in the description of ``quantum channels'' and ``correlations'' between quantum systems (at least when these are described in the language of algebraic quantum theory as C*-algebras). 
Actually, since the very beginning of this investigation in higher C*-category theory, the mathematical formalization of relational quantum theory has been one of the basic goals of our research in view of its potential impact on our ongoing efforts in modular algebraic quantum 
gravity~\cite{modular,B}. 

\medskip 

We finally collect in section~\ref{sec: o} some further indication on possible extensions of this work, also in directions that we plan to explore in the future. 

\section{C*-algebras and C*-categories}\label{sec: C*}

The theory of operator algebras (see for example B.Blackadar~\cite{Bl} for an overview of the subject and furher references) is a quite developed area of functional analysis with extremely important applications to the mathematical approaches to quantum theory (see for example the books by F.Strocchi~\cite{Str}, R.Haag~\cite{H}, G.Emch~\cite{Em}, O.Bratteli-D.Robinson~\cite{BR} and J.E.Roberts' lectures~\cite{R2,R3}). Since our main purpose is to examine some possible routes for a vertical categorification of such a theory (with some non-trivial examples), we start here with a brief review, recalling the basic notion of C*-algebra, its horizontal categorified and ``bundlified'' versions (C*-categories and Fell bundles), as well as the few instances of already available axioms for monoidal and 2-C*-categories. 

The readers that are not already familiar with the notions of category theory mentioned here, will find all the references and required definitions in detail in the following section~\ref{sec: hc}. 

\subsection{C*-algebras, C*-categories, Fell Bundles, Spaceoids}  

C*-algebras, originally defined by I.Gel'fand-M.Na\u\i mark~\cite{GN}, are the most basic gadget in the theory of operator algebras and non-commutative geometry~\cite{Co}, where they play the role of non-commutative topological spaces and it is natural to start from them in any attempt to categorify functional analysis. 

A C*-algebra is a rigid blend of algebraic and topological structures: an associative algebra over $\CC$, equipped with an antimultiplicative conjugate-linear involution, that is at the same time a Banach space with a norm that is submultiplicative and satisfies the so called C*-property. 
\begin{definition}
A \emph{complex unital C*-algebra} $(\Cs,\circ,*,+,\cdot, \|\cdot \|)$ is given by the following data: 
\begin{itemize}
\item 
a complex associative unital involutive algebra i.e.~a complex vector space $(\Cs,+,\cdot)$ over $\CC$, equipped with an associative unital bilinear multiplication $\circ:\Cs\times\Cs\to\Cs$ and conjugate-linear antimutiplicative involution $*:\Cs\to\Cs$, 
\item 
a norm $\|\cdot \|:\Cs\to\RR$ such that the following properties are satified: 
\begin{itemize}
\item
completeness: $(\Cs,+,\cdot,\| \cdot \|)$ is a Banach space, 
\item
submultiplicativity of the norm: $\|x\circ y\| \leq \|x\| \cdot \|y\|$ for all $x,y\in \Cs$, 
\item
C*-property: $\|x^*\circ x\|=\|x\|^2$ for all $x\in \Cs$.  
\end{itemize} 
\end{itemize}
\end{definition}
Basic non-commutative examples are the families $\Bs(\H)$ of linear continuous maps on a Hilbert space $\H$ (and all the norm-closed unital involutive subalgebras of them); commutative examples are essentially algebras $C(X;\CC)$ of complex-valued continuous functions on a compact Hausdorff topological space $X$. 

\medskip 

Horizontal categorifications of C*-algebras have been developed a long time ago by J.E.Roberts and used in the theory of superselection sectors in algebraic quantum field theory. The formal definition first appeared in P.Ghez-R.Lima-J.E.Roberts~\cite{GLR} and  it has been revisited more recently in greater details in P.Mitchener~\cite{M}: 
\begin{definition}
A \emph{C*-category} $(\Cs,\circ,*,+,\cdot,\| \cdot \|)$ is given by the following data: 
\begin{itemize}
\item 
an involutive algebroid $(\Cs,\circ, *,+,\cdot)$ over $\CC$:
\begin{itemize}
\item
a category $(\Cs,\circ)$, with objects (partial identities) $\Cs^0\subset \Cs$, 
\item  
a contravariant functor $*:\Cs\to\Cs$ acting trivially on $\Cs^0$, 
\item 
for all pairs of objects $A,B\in \Cs^0$, a complex vector space structure $(\Cs_{AB},+,\cdot)$ on the home-sets $\Cs_{AB}:=\Hom_\Cs(B,A)$, 
on which the composition $\circ:\Cs_{BC}\times\Cs_{AB}\to\Cs_{AC}$, $(y,x)\mapsto x\circ y$ is bilinear and the involution $*:\Cs_{AB}\to\Cs_{BA}$, $x\mapsto x^*$ is conjugate-linear,   
\end{itemize}
\item 
a norm function $\| \cdot \| :\Cs\to \RR$ such that: 
\begin{itemize}
\item 
completeness: $(\Cs_{AB},+,\cdot)$ are Banach spaces, $\forall A,B\in \Cs^0$, 
\item 
submultiplicativity: $\|x\circ y\|\leq\|x\|\cdot\|y\|$ whenever $x\circ y$ exists, 
\item 
C*-property: $\|x^*\circ x\|=\|x\|^2$ for all $x\in\Cs$, 
\item
positivity: for all $x\in \Cs$, the element $x^*\circ x$ is positive in the unital C*-algebra $\Cs_{s(x)s(x)}$, where $s(x)\xrightarrow{x}t(x)$.  
\end{itemize}
\end{itemize}
\end{definition}

\begin{remark}
The axiom of positivity, in the case of C*-algebras, is redundant. 

In the statement of this positivity property, we make use of the fact that $\Cs_{s(x)s(x)}$ is a unital C*-algebra, for all $x\in \Cs$, where $s(x)$ denotes the source partial identity of the element $x$. 

In fact it is immediately implied by the definition that, for all objects $A,B\in \Cs^0$, the diagonal home-sets 
$\Cs_{AA}$ are unital C*-algebras and the off-diagonal home-sets $\Cs_{AB}$ are unital Hilbert C*-bimodules, over the C*-algebra $\Cs_{BB}$ to the right, and over the C*-algebra $\Cs_{AA}$ to the left, with right and left inner products given respectively by ${}_\bullet\ip{x}{y}:=x\circ y^*$ and $\ip{x}{y}_\bullet:=x^*\circ y$ that satisfy the associative property ${}_\bullet\ip{x}{y}z=x\ip{y}{z}_\bullet$, for all $x,y,z\in \Cs$. 
\xqed{\lrcorner}
\end{remark} 

As we can expect from horizontal categorification, a C*-algebra is just a C*-category whose class of objects contains only one element.  
Basic examples of C*-categories are provided by the family $\Bs(\Hf)$ of linear bounded operators between Hilbert spaces belonging to a given class 
$\Hf$ (every C*-category can be seen as a norm-closed unital involutive sub-algebroid of $\Bs(\Hf)$ for a given family $\Hf$). 

\medskip 

A C*-category $\Cs$ can immediately be seen as a bundle, with Banach fibers $\Cs_{AB}$, over the pair groupoid 
$\Cs^0\times\Cs^0:=\{AB\ | \ A,B\in \Cs^0\}$ of its objects with the discrete topology. 
Allowing more than a single arrow connecting two objects $A,B$ of the base category and adding the possibility of a non-trivial topology, leads to the definition of a Fell bundle, that plays a fundamental role in spectral theory (in a way that further elaborates on the tradition of the celebrated Dauns-Hofmann theorem~\cite{DH}). 

\begin{definition}
A \emph{Banach bundle}\footnote{See, for example, J.Fell-R.Doran~\cite[Section~I.13]{FD} or N.Weaver~\cite[Chapter~9.1]{W} and the references therein.} is a bundle $(\Es,\pi,\Xs)$, i.e.~a continous open surjective map $\pi:\Es\to\Xs$, whose total space is equipped with:
\begin{itemize}
\item
a partially defined continuous binary operation of addition $+:\Es\times_\Xs\Es\to\Es$, with domain the set 
$\Es\times_\Xs\Es:=\{(x,y)\in \Es\times\Es \ | \ \pi(x)=\pi(y)\}$, 
\item 
a continuous operation of multiplication by scalars $\ \cdot:\KK\times\Es\to\Es$, 
\item 
a continuous ``norm'' $\|\cdot\|:\Es\to\RR$, such that:
\begin{itemize}
\item
for all $x\in \Xs$, the fiber $\Es_x:=\pi^{-1}(x)$ is a complex Banach space $(\Es_x,+,\cdot)$ with the norm $\|\cdot\|$, 
\item 
for all $x_o\in \Xs$, the family $U_{x_o}^{\Os,\epsilon}=\{e\in \Es \ | \ \|e\|<\epsilon,\ \pi(e)\in \Os\}$, where $\Os\subset \Xs$ is an open set containing $x_o\in\Xs$ and $\epsilon>0$, is a fundamental system of neighborhoods of 
$0\in E_{x_o}$.  
\end{itemize}
\end{itemize} 
A \emph{Hilbert bundle} is a Banach bundle whose norm is induced fiberwise by inner products. 

\medskip 

A \emph{Fell bundle}\footnote{For Fell bundles over topological groups see J.Fell~\cite[Section~II.16]{FD}; for Fell bundles over groupoids (originally introduced by S.Yamagami) see A.Kumjian~\cite{K}; for Fell bundles over inverse semigroups (defined by N.Seiben) see R.Exel~\cite[Section~2]{Ex}; Fell bundles over involutive inverse categories (involutive categories $\Xs$ such that $x\circ x^*\circ x=x$ for all $x\in \Xs$) appeared in~\cite{BCL4}.} 
over a topological involutive category $\Xs$, is a Banach bundle $(\Es,\pi,\Xs)$ that is also an involutive categorical bundle, i.e.~$\pi:\Es\to\Xs$ is a continuous $*$-functor between topological involutive categories $\Es,\Xs$, and such that: 
\begin{itemize}
\item
$\|x\circ y\|\leq \|x\|\cdot \|y\|$ for all composable $x,y\in \Es$, 
\item 
$\|x^*\circ x\|=\|x\|^2$ for all $x\in \Es$, 
\item 
$x^*\circ x$ is positive whenever $\pi(x^*\circ x)$ is an idempotent in $\Xs$.\footnote{The condition is meaningful because the fiber $\Es_{\pi(x^*\circ x)}\subset \Es$ is a C*-algebra if and only if $\pi(x^*\circ x)\in \Xs$ is an idempotent.} 
\end{itemize}
\end{definition}

\begin{remark}
The positivity condition in the previous definition requires some care: the axioms preceding it already imply that every fiber $\Es_p$ is a C*-algebra, whenever $p\in \Xs$ is an idempotent in the involutive category $\Xs$, hence it is perfectly possible to require the positivity of $x^*\circ x$ if it belongs to such a fiber (this is the usual condition in the case of Fell bundles over groupoids and C*-categories). 

It might seem suspicious that no additional positivity requirement is necessary for an arbitrary $x\in \Es$. Since $\Es_{\pi(x^*\circ x)}$ is generally only a Hilbert C*-bimodule, the only reasonable option would be to ask the positivity of $x^*\circ x$ as an element of a suitable \emph{convolution C*-algebra} ``generated'' by $\Es_{\pi(x^*\circ x)}$. The positivity axiom in the previous definition of Fell bundle is a necessary condition for the existence of such a C*-algebra; anyway, if such a C*-algebra exists, all the elements $x^*\circ x$ would always be already positive, making further requirements redundant.  

Although we will not enter here into this very interesting topic, using a variant of the construction of the C*-algebra of multipliers via double centralizers, it is actually possible to show that convolution C*-algebras for fibers of a Fell bundle (as defined here) always exist. 
\xqed{\lrcorner}
\end{remark}

As already anticipated, a C*-category $(\Cs,\circ,*,\|\cdot\|)$ is itself a special case of a Fell bundle over the pair groupoid $\Cs^0\times\Cs^0$ with the discrete topology and with fibers $\Cs_{AB}$, for $(A,B)\in \Cs^0\times\Cs^0$.  

Other elementary examples of Fell bundles (over groupoids) are given by the ``tautological'' bundles with base any strict groupoid of imprimitivity Hilbert C*-bimodules (in the category of strong Morita equivalences of complex unital C*-algebras), with fibers the Hilbert C*-bimodules themselves. Other notable examples of Fell bundles are given by \emph{spaceoids}: spectra of commutative full C*-categories defined, used and studied in~\cite{BCL2,BCL4}. 

\medskip 

The relevance of these structures for spectral theory can be fully appreciated considering the following theorem by A.Takahashi~\cite{T1,T2} (originally proved via the Dauns-Hofmann theorem), that simultaneously subsumes the Gel'fand-Na\u\i mark duality and (the Hermitian version of) the Serre-Swan equivalence,\footnote{Although A.Takahashi does not directly treat tensor products, the bicategorical version is immediate (see~\cite{BCL1,BCL4}).} and from the horizontal categorification of the Gel'fand-Na\u\i mark duality described in~\cite{BCL2}. 

\begin{theorem}[Takahashi~\cite{T2}]
There is a duality between the bicategories:\footnote{With operations given by composition and tensor product.} 
\begin{itemize}
\item
$\Cf$ of homomorphisms of Hilbert C*-modules over commutative unital C*-algebras,  
\item 
$\Sf$ of Takahashi morphisms of Hilbert bundles over compact Hausdorff spaces. 
\end{itemize} 
\end{theorem}
Morphisms in $\Cf$ are pairs $(\phi,\Phi):{}_\As\Ms\to {}_\Bs\Ns$ with $\Phi:\Ms\to\Ns$ adjointable map of Hilbert C*-modules and $\phi:\As\to\Bs$ a unital $*$-homorphism such that 
$\Phi(a\cdot x)=\phi(a)\cdot\Phi(x), \quad \forall a\in \As$, $x\in \Ms$. 

\medskip 

Morphisms in $\Sf$ are pairs $(f,F): (\Es,\pi,X)\to(\Fs,\rho,Y)$, where $f:X\to Y$ is a continuous map and $F:f^\bullet(\Fs)\to\Es$ is a continuous fiberwise-linear map of Hilbert bundles over $X$ defined on the total space $f^\bullet(\Fs)$ of the $f$-pull-back  of the Hilbert bundle $(\Fs,\rho, Y)$. 

\begin{theorem}[Bertozzini, Conti, Lewkeeratiyutkul~\cite{BCL2}] \label{th: bcl}
There is a duality, via horizontally categorified Gel'fand and evaluation natural transformations, between the categories: 
\begin{itemize}
\item
$\Cf$ of $*$-functors between full commutative C*-categories, 
\item 
$\Sf$ of Takahashi morphism of spectral spaceoids (that are Fell line-bundles over the Cartesian product of a pair groupoid and a compact Hausdorff topological space). 
\end{itemize}
\end{theorem}

The two functors in duality are the: 
\begin{itemize}
\item
section functor $\Sf\xrightarrow{\Gamma} \Cf$ that to a spaceoid associates its C*-category of continuous sections, 
\item 
spectrum functor $\Cf\xrightarrow{\Sigma}\Sf$ that to a commutative full C*-category associates its spectral spaceoid. 
\end{itemize}

\subsection{Monoidal C*-categories, Longo-Roberts 2-C*-categories} 

Towards a full vertical categorification of C*-algebras, in this subsection we start with a discussion of those few already available notions that are directly related to higher C*-categories. 

\medskip 

In S.Doplicher-J.E.Roberts~\cite{DR} a notion of monoidal (or tensor) C*-category has been developed. 
Since strict monoidal categories are strict 2-categories with only one object (the monoidal identity), such definition is the first available hint for the axioms of 2-C*-categories. 
\begin{definition}
A \emph{strict monoidal C*-category} is a C*-category $(\Cs,\circ,*,+,\cdot, \| \ \|)$ equipped with an additional binary operation 
$\otimes:\Cs\times\Cs\to\Cs$ such that:
\begin{itemize}
\item 
$(\Cs,\otimes)$ is a monoid (a category with only one object),  
\item
$\otimes:\Cs\times\Cs\to \Cs$ is a bifunctor,\footnote{Recall that a bifunctor from $(\Cs,\circ)$ to $(\hat{\Cs},\hatcirc)$ is a functor $\otimes: \Cs\times \Cs\to\hat{\Cs}$ defined on the  product category $\Cs\times\Cs$, with componentwise composition. This condition implies the exchange property.} 
\item
$\otimes$ is a bilinear map when restricted to pairs of composable 1-home-sets,
\item 
$*:(\Cs,\otimes)\to(\Cs,\otimes)$ is a covariant functor. 
\end{itemize}
\end{definition}

\begin{remark}\label{rem: conjugates 1}
The actual categories considered by S.Doplicher and J.E.Roberts for the theory of superselection sectors in algebraic quantum field theory are equipped with additional structures: they are symmetric monoidal C*-categories, closed under retracts, direct sums and (more important for us) with conjugates.\footnote{For some aspects of the theory it is also required the triviality condition $\Cs_{II}\simeq\CC$, where $I$ denotes the monoidal identity.} 
Following R.Haag~\cite[section IV.4]{H}, if necessary to avoid confusion, we reserve the name \emph{Doplicher-Roberts C*-categories} for such more specific cases. 
We will later return to a careful study of conjugates for strict monoidal C*-categories (and more generally for Longo-Roberts 
2-C*-categories defined here below) in section~\ref{rem: conjugates} and example~\ref{rem: conjugates2}.  
\xqed{\lrcorner}
\end{remark}

The notion of 2-C*-category was developed by R.Longo-J.E.Roberts~\cite[section~7]{LR} and further studied by P.Zito~\cite{Z}. It is a horizontal categorification of a monoidal C*-category defined as follows:\footnote{Notice here the presence of only one involution on 2-arrows over 1-arrows. Although in some important cases, conjugations (involutions of 2-arrows over objects) can be introduced (see section~\ref{rem: conjugates}), plenty of examples possess only one involution.} 
\begin{definition}\label{def: lr}
A \emph{Longo-Roberts 2-C*-category} is a strict 2-category $(\Cs,\circ,\otimes)$ such that
\begin{itemize}
\item
for all objects $A,B\in \Cs^0$ the home-set $\Cs_{AB}$ is a C*-category with composition $\circ$ and involution $*$, 
\item
the partial bifunctor $\otimes$ is bilinear when restricted to $\circ$-composable $0$-home-sets,   
\item
$*:(\Cs,\otimes)\to(\Cs,\otimes)$ is a covariant functor.\footnote{This property (that is true in all the examples) is actually missing in both~\cite{LR,Z}, but this is probably just a careless omission, otherwise the definition would not reproduce that of monoidal C*-categories when there is only one object.} 
\end{itemize}
\end{definition}

The easiest examples of monoidal C*-categories are given by bounded linear maps between a family of Hilbert spaces (with the usual composition and tensor product). 
Similarly, examples of Longo-Roberts 2-C*-categories are given by adjointable maps between a family of right (respectively left) Hilbert C*-correspondences (i.e.~unital bimodules ${}_\As\Ms_\Bs$ over complex unital C*-algebras with a $\Bs$-valued (respectively $\As$-valued) inner product).

\begin{remark}
Since algebraic tensor products of bimodules over rings (and similarly Rieffel internal tensor product of Hilbert C*-correspondences over C*-algebras) are only weakly associative and weakly unital, it would appear that the previous examples produce only 2-categories that are ``weak'' under $\otimes$ and hence do not precisely comply with definition~\ref{def: lr}. This problem is easily eliminated via the following useful \emph{strictification} procedure embedding all the given Hilbert C*-bimodules into their strictly associative \emph{tensor algebroid} of paths 
(this is a horizontal categorification of a well-known C.Chevalley's procedure~\cite{Ch} and essentially consists of constructing the required tensor products of  bimodules inside a strictly associative unital tensor ring: the free ring generated by the bimodules). 

\smallskip 

Consider a 1-quiver $\Mf\rightrightarrows \Af$ whose nodes are unital associative rings $\As,\Bs\in \Af$ and whose 1-arrows (for example with source $\Bs$ and target $\As$) are unital bimodules of the form ${}_\As\Ms_\Bs\in \Mf$. Denote by $[\Mf]\rightrightarrows \Af$ the \emph{fine graining} of the previous 1-quiver $\Mf$ consisting of the same nodes $\Af$ but with every element $x\in {}_\As\Ms_\Bs$ considered as a different 1-arrow from $\Bs$ to $\As$ (and including, for all $\As\in \Af$, all the elements $a\in \As$ as 1-loops based on $\As$). Proceed to the construction of the free 1-category of paths $\left<[\Mf]\right>$ generated by the fine grained 1-quiver $[\Mf]\rightrightarrows\Af$ and then to $\ZZ[\left<[\Mf]\right>]$, its \emph{category ringoid with coefficient in $\ZZ$}. This is a horizontal categorification of the usual monoid ring $\ZZ[\Xs]$ with integer coefficients over the monoid $\Xs$: its elements are finite formal linear combinations with integer coeffients of 1-arrows belonging only to a given home-set $\left<[\Mf]\right>_{\As\Bs}$ (hence each of these home-sets is an abelian group). Bilinearly extending the composition, $\ZZ[\left<[\Mf]\right>]$ turns out to be a ringoid with the set of objects $\Af$. Finally we obtain the \emph{tensor ringoid} $\Ts(\Mf)$ quotienting the ringoid $\ZZ[\left<[\Mf]\right>]$ by the categorical ideal $\Is$ generated (home-set by home-set) by the elements of the form 
\begin{gather*}
(\dots, x,b,y,\dots)-(\dots,x,by,\dots), 
\quad 
(\dots, x,b,y,\dots)-(\dots,xb,y,\dots), 
\\ 
(\dots,x_1+x_2,\dots)-(\dots,x_1,\dots)-(\dots,x_2,\dots),
\end{gather*} 
for all 1-arrows $x,x_1,x_2\in{}_\As\Ms_{\Bs}$, $y\in {}_\Bs\Ns_\Cs$ and all 1-loops $a\in \As$. Each one of the original bimodules ${}_\As\Ms_\Bs$ (and each one of the rings $\As\in \Af$) has an isomorphic copy inside $\Ts(\Mf)$ via the inclusion $x\mapsto [(x)]:=(x)+\Is_{\As\Bs}$ and the tensor product operation, defined by $[(x)]\otimes[(y)]:=[(x,y)]$, for all $x\in {}_\As\Ms_\Bs, y\in {}_\Bs\Ns_\Cs$, is now strictly associative and unital as required. 
\xqed{\lrcorner}
\end{remark}

\section{Strict Higher Categories and Non-commutative Exchange}  \label{sec: hc}

In this section we introduce, with some detail, the basic definitions in the theory of strict $n$-categories with their usual ``exchange property''. 
We then present the well-known Eckmann-Hilton collapse argument and, in order to avoid it, we propose a relaxed form of exchange property (quantum or non-commutative exchange) consisting in a request of $\circ_p$-functoriality for right/left $\circ_q$-multiplications by $p$-identities, whenever $q<p$. 
Finally, for later use, we also discuss examples (products of categories) whose $n$-cells admit compositions that do not fit with the usual globular or cubical situations. 

Here, the (admittedly questionable) inspiring ideology is to view the current developments in higher category theory as heavily motivated by ``classical homotopy theoretical'' arguments (for example the exchange property) that might not be suitable for a formalization of non-commutative operator algebraic structures that are otherwise perfectly natural and fitting into a higher categorical context. 

\medskip 

Although some of the most natural approaches to the definition of strict higher categorical environments are via 
``globular/cubical higher quivers''~\cite[definition~1.4.8]{L} and either via ``inductive enrichment of categories''~\cite[definition~1.4.1]{L} (for the case of globular shaped cells) or via ``inductive internal categories''~\cite[definition~1.4.13]{L} (for the case of cubical shaped cells), for our discussion here, in view of its extreme compactness, we will use the algebraic definition of globular strict $n$-categories via axioms for their 
``$n$-cells''. We will mainly consider the case of ``globular $n$-cells'' and a more careful study of strict $n$-tuple categories, based on similar algebraic axioms for ``cubical cells'', will be done elsewhere.\footnote{P.Bertozzini, R.Conti, R.Dawe-Martins, ``Involutive Double Categories'' (manuscript) and ``Double C*-categories and Double Fell Bundles'' (works in progress).}

\subsection{Strict Globular Higher Categories (via partial higher monoids)}

Among the equivalent definitions of strict 1-category, we choose a compact axiomatization formulated in terms of properties of 1-arrows under a partial binary operation of composition without any direct reference to objects, identities, source and target maps. 
The following, for example, is a variant of the definition provided in S.Mac Lane~\cite[section~1, page~9]{ML}. 
The resulting notion of partial monoid is a horizontal categorification of the usual definition of monoid, obtained by ``localization'' of identities. 

\begin{definition}
A \emph{1-category} $(\Cs,\circ)$ is a family $\Cs$ of 1-cells (arrows) equipped with a partially defined binary operation of composition $\circ$ that satisfies the following requirements:
\begin{itemize}
\item 
the composition is associative i.e.~whenever one of the two terms $f\circ(g\circ h)$ and $(f\circ g)\circ h$ exists, the other one exists as well and they coincide, 
\item 
for every arrow $f\in \Cs$ there exist a right composable (source) arrow $r\in \Cs$ and a left composable (target) arrow $l\in \Cs$ that are \emph{partial  identities} (objects) i.e.~for all arrows $h_1, h_2,k_1,k_2\in \Cs$: $h_1\circ r=h_1$, $r\circ h_2=h_2$ and $k_1\circ l=k_1$, $l\circ k_2=k_2$, whenever the compositions exist,  
\item 
if $f$ has a right identity that is also a left identity for $g$, the composition $f\circ g$ exists. 
\end{itemize}
A \emph{1-functor} $(\Cs_1,\circ_1)\xrightarrow{\phi} (\Cs_2,\circ_2)$ between two 1-categories is a map $\phi:\Cs_1\to\Cs_2$ such that 
\begin{itemize}
\item 
whenever $x\circ_1y$ exists, also $\phi(x)\circ_2\phi(y)$ exists and in this case $\phi(x\circ_1y)=\phi(x)\circ_2\phi(y)$, 
\item
if $e$ is a partial identity in $(\Cs_1,\circ_1)$, $\phi(e)$ is a partial identity in $(\Cs_2,\circ_2)$. 
\end{itemize} 
\end{definition}
The class of the partial identities (objects) will be denoted by $\Cs^0\subset \Cs$ with inclusion map $\iota:A\mapsto \iota_A$. 
From the axioms it follows immediately the every $x\in \Cs$ has a unique right partially identity (its source) and a unique left partial identity (its target) that we will denote respectively by $s(x)$ and $t(x)$. 
\\ 
The following graphical representations are self-explicative: 
\begin{gather*}
\text{$0$-cells (objects):} \quad \bullet, \quad \quad \quad  
\text{$1$-cells (arrows):} \quad \xymatrix{\bullet_1 \ar[r]& \bullet_2} 
\\ 
\text{sources / targets:} \quad \xymatrix{s(x)\ar[r]^x & t(x)}, \quad \text{identities:} \quad A \mapsto  \xymatrix{A\ar@(dr,ur)_{\iota_A}} 
\\ 
\text{composition:} \quad \xymatrix{A  \ar[r]^g & B \ar[r]^f & C} \ \ \mapsto \ \ \xymatrix{A \ar[rr]^{f\circ g}& & C.} 
\end{gather*}
Given $A,B\in \Cs^0$, we denote by $\Cs_{AB}:=\{x\in \Cs \ | \ s(x)=B, t(x)=A\}$ the home-set of 1-arrows with source $B$ and target $A$. 
The category $(\Cs,\circ)$ is said to be \emph{locally small} if every home-set $\Cs_{AB}$ is a set and \emph{small} if $\Cs$ is also a set.\footnote{For locally small categories this is equivalent to asking $\Cs^0$ to be a set.} 

\medskip 

Also for $n$-categories, we have an equivalent ``$n$-arrows based''-definition. The following is essentially the definition of J.E.Roberts as provided by J.E.Roberts-G.Ruzzi~\cite{RR} and already used, for the case $n=2$, in R.Longo-J.E.Roberts~\cite{LR} and P.Zito~\cite{Z}: 
\begin{definition}\label{def: ncat}
A \emph{globular strict $n$-category} $(\Cs,\circ_0,\dots , \circ_{n-1})$ is a set $\Cs$ equipped with a family of partially defined compositions $\circ_p$, for $p:=0,\dots, n-1$, that satisfy the following list of axioms: 
\begin{itemize}
\item
for all $p=0,\dots, n-1$, $(\Cs,\circ_p)$ is a 1-category, whose partial identities are denoted by $\Cs^p$,\footnote{We will of course use $\Cs^n$ to denote $\Cs$.} 
\item 
for all $q<p$, a $\circ_q$-identity is also a $\circ_p$-identity, i.e.~$\Cs^q\subset \Cs^p$, 
\item 
for all $p,q=0,\dots, n-1$, with $q<p$, the $\circ_q$-composition of $\circ_p$-identities, whenever exists, is a $\circ_p$-identity, i.e.~$\Cs^p\circ_q\Cs^p\subset \Cs^p$, 
\item 
the \emph{exchange property} holds for all $q<p$: whenever $(x\circ_p y)\circ_q (w\circ_p z)$ exists also 
$(x\circ_q w)\circ_p (y\circ_p z)$ exists and they coincide.\footnote{By symmetry, the exchange property automatically holds for all $q\neq p$.} 
\end{itemize} 
A \emph{covariant functor} $(\Cs_1,\circ_0,\dots,\circ_{n-1})\xrightarrow{\phi} (\Cs_2,\hatcirc_0,\dots,\hatcirc_{n-1})$ between two globular strict 
$n$-categories is a homomorphism for each of the partial 1-monoids involved, i.e.~a map $\phi:\Cs_1\to\Cs_2$ such that:  
\begin{itemize}
\item 
whenever $x\circ_q y$ exists, also $\phi(x)\hatcirc_q\phi(y)$ exists and in this case 
$\phi(x\circ_q y)=\phi(x)\hatcirc_q \phi(y)$, 
\item
if $e$ is a partial $\circ_q$-identity in $\Cs^q_1$, $\phi(e)$ is a $\hatcirc_q$-partial identity in $\Cs_2^q$. 
\end{itemize} 
More generally, a \emph{covariant relator} between $n$-categories is a relation $\Rg\subset\Cs_1\times\Cs_2$ such that for all $p$:
\begin{itemize} 
\item 
whenever $(x_1\circ_p x_2)$, $(y_1\hatcirc_p y_2)$ exist and $(x_1,y_1),(x_2,y_2)\in\Rg$ we have 
$((x_1\circ_p x_2),(y_1\hatcirc_p y_2))\in \Rg$, 
\item 
if $(x,y)\in \Rg$ and $e,f\in \Cs^p$, we have $(e,f)\in \Rg$ whenever $(x\circ_pe,y\hatcirc_p f)$, or 
$(e\circ_p x,f\hatcirc_p y)$, exists.\footnote{This is a vertical categorification of the (corrected) definition of relator for 1-categories that was presented in~\cite{BCL}, where this second unitality condition was mistakenly omitted.} 
\end{itemize}
\end{definition} 
The first and the third axioms, imply that, for all $0\leq q<p\leq n$, $(\Cs^p,\circ_q)$ is a 1-category. It follows immediately that, for all $p=1,\dots,n$, any $p$-cell $x\in \Cs^p$ has a unique $q$-source and a unique $q$-target $s^p_q(x),t^p_q(x)\in \Cs^q$. The second axiom allows us to define the  inclusion maps $\iota^p_q:\Cs^q\to\Cs^p$ such that, for all $x\in \Cs^q$, $s^p_q(\iota^p_q(x))=x=t^p_q(\iota^p_q(x))$. The third axiom also assures the functoriality of the maps $\iota^m_p: (\Cs^p,\circ_q)\to(\Cs^m,\circ_q)$. It is particularly crucial to notice that the globular shape of the 
\hbox{$m$-cells}, for all $1<m\leq n$, is actually implicitly determined by the specific form in which the exchange property is stated: for all 
$x\in \Cs^m$, for all $0\leq q < p < m\leq n$, $s^m_p(x),x,t^m_p(x)$ are $\circ_p$-composable and from the fact that both 
$s_q(s_p(x)),s_p(x),t_q(s_p(x))$ and $s_q(t_p(x)), t_p(x),t_q(t_p(x))$ are $\circ_q$-composable, from the exchange property we obtain the 
$\circ_q$-composability of both $s_q(s_p(x)),s_q(t_p(x))$ and $t_q(s_p(x)),t_q(t_p(x))$ that, being $\circ_q$-identities, implies the globular property 
$s_q(s_p(x))=s_q(t_p(x))$ and $t_q(s_p(x))=t_q(t_p(x))$. In a perfectly similar way, the exchange property, implies the functoriality of the maps $s^m_p,t^m_p:(\Cs^m,\circ_q)\to (\Cs^p,\circ_q)$, for all $0\leq q<p<m\leq n$. 

A graphical representation illustrates the combinatorial/geometrical meaning coded in the definition: 
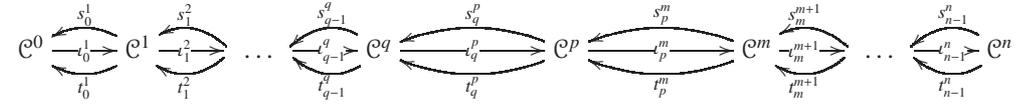
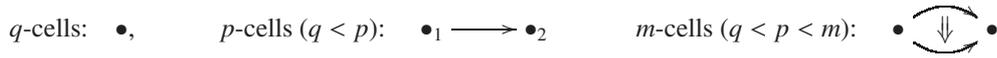
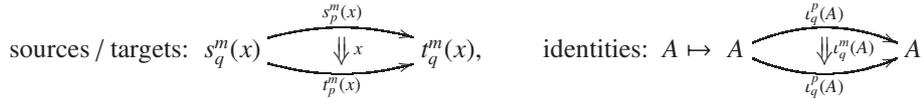
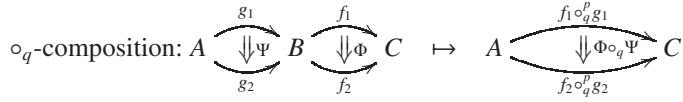
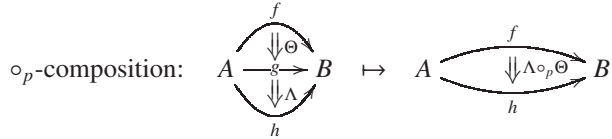
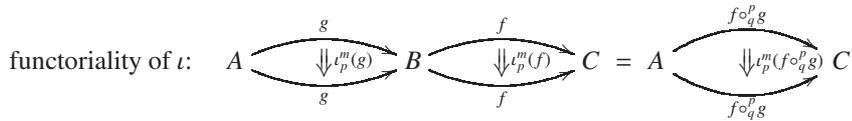
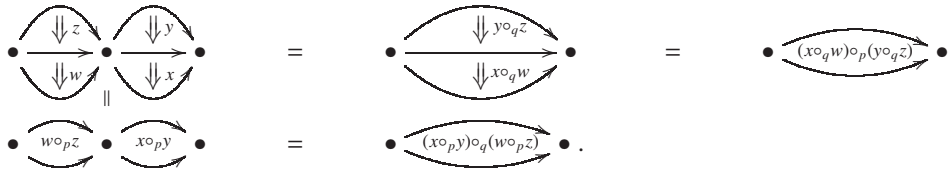
\begin{gather*}
\xymatrix{\Cs^0 \ar[r]|{\iota^1_0}& \Cs^1 \ltwocell_{s^1_0}^{t^1_0}{\omit}\ar[r]|{\iota^2_1}&  \ltwocell_{s^2_1}^{t^2_1}{\omit}}
\dots 
\xymatrix{
\ar[r]|{\iota^q_{q-1}}&  
\ltwocell_{s^q_{q-1}}^{t^q_{q-1}}{\omit}\Cs^q\ar[rr]|{\iota^p_q}& & \Cs^p  \lltwocell_{s^p_q}^{t^p_q}{\omit}\ar[rr]|{\iota^m_p}& & 
\Cs^m \lltwocell_{s^m_p}^{t^m_p}{\omit}\ar[r]|{\iota^{m+1}_m}& \ltwocell_{s^{m+1}_m}^{t^{m+1}_m}{\omit}}
\dots 
\xymatrix{\ar[r]|{\iota^n_{n-1}} & \ltwocell_{s^n_{n-1}}^{t^n_{n-1}}{\omit} \Cs^n 
}
\\ 
\text{$q$-cells:} \quad \bullet , \quad \quad \quad  
\text{$p$-cells ($q<p$):} \quad \xymatrix{\bullet_1 \ar[r]& \bullet_2} \quad \quad \quad 
\text{$m$-cells ($q<p<m$):} \quad \xymatrix{\bullet \rtwocell & \bullet}
\\ 
\text{sources / targets:} \ 
\xymatrix{s^m_q(x) \rrtwocell^{s^m_p(x)}_{t^m_p(x)} {x}& & t^m_q(x),} 
\quad  \quad 
\text{identities:} \ \ A \mapsto \xymatrix{A \rrtwocell^{\iota^p_q(A)}_{\iota^p_q(A)} {\quad \iota^m_q(A)}  & &  A}
\\
\text{$\circ_q$-composition:} \xymatrix{A \rtwocell^{g_1}_{g_2}{\Psi} & B \rtwocell^{f_1}_{f_2}{\Phi}& C}\ \ \mapsto \ \ 
\xymatrix{A \rrtwocell^{f_1\circ^p_q g_1}_{f_2\circ^p_q g_2}{\quad \ \Phi\circ_q\Psi}& & C} 
\\ 
\text{$\circ_p$-composition:} \quad \xymatrix{A \ruppertwocell^{f}{\Theta} \rlowertwocell_{h}{\Lambda} \ar[r]|{g} & B}\ \ \mapsto \ \ 
\xymatrix{A \rrtwocell^f_h{\quad \ \Lambda\circ_p \Theta}& &  B}
\\ 
\text{functoriality of $\iota$:} \quad 
\xymatrix{A \rrtwocell^{g}_{g}{\quad \iota^m_p(g)} & & B \rrtwocell^{f}_{f}{\quad \iota^m_p(f)}& & C} = 
\xymatrix{A \rrtwocell<5>^{f\circ^p_q g}_{f\circ^p_q g}{\quad \quad \iota^m_p(f\circ^p_q g)}& & \ C}
\end{gather*}

\begin{gather*}
\text{exchange property:} \\ 
\xymatrix{
\bullet \ruppertwocell{z} \rlowertwocell{w} \ar[r] 	& \bullet \ruppertwocell{y} \rlowertwocell{x} \ar[r] \ar@{}[d]|{\|} & \bullet 
 	& = 
& \bullet \rruppertwocell{\quad y\circ_q z} \rrlowertwocell{\quad x\circ_q w} \ar[rr] & & \bullet   
	& = 
& \bullet \rrtwocell{\omit (x\circ_q w)\circ_p(y\circ_q z)}&   & \bullet   
\\ 
\bullet \rtwocell{\omit w\circ_p z} &  \bullet \rtwocell{\omit x\circ_p y} & \bullet 
	&= 
& \bullet \rrtwocell{\omit (x\circ_p y)\circ_q(w\circ_p z)} & &  \bullet\ . 
& & }
\end{gather*}

We introduce, for all $q<p$ and $x,y\in \Cs^q$, the notation 
${}_q\Cs^p_{xy}:=\{z\in \Cs^p \ | \ s^p_q(z)=y, \ t^p_q(z)=x\}$ to denote the \emph{$q$-home-set of $p$-arrows} i.e.~the class of $p$-arrows whose $q$-source is $y$ and whose $q$-target is $x$. 

Clearly, ${}_q\Cs^p_{xz}\circ^p_q {}_q\Cs^p_{zy}\subset {}_q\Cs^p_{xy}$, for all $x,y,z\in \Cs^q$ and the family of pairs of $\circ_q$-composable $p$-cells $\Cs^p\times_{\Cs^q}\Cs^p:=\{(a,b)\in \Cs^p\times\Cs^p \ | \ s^p_q(b)=t^p_q(a)\}$ is given by $\Cs^p\times_{\Cs^q}\Cs^p=\bigcup_{x,y,z\in \Cs^q}{}_q\Cs^p_{zy}\times{}_q\Cs^p_{xz}$. 

For all $x,y\in \Cs^r$, whenever $r<q<p$, we have $\iota^p_q({}_r\Cs^{q}_{xy})\subset{}_r\Cs^p_{xy}$, furthermore 
${}_q\Cs^p_{xy}\subset {}_r\Cs^p_{t^q_r(x)s^q_r(y)}$ and 
${}_r\Cs^p_{xy}=\bigcup_{a,b\in {}_r\Cs^{q}_{xy}} \ {}_q\Cs^p_{ab}$, where the union is disjoint.

\medskip 

The usual notion of natural transformation between functors can be immediately reframed, in the setting of our ``$n$-arrows'' definition of strict $n$-categories, via ``intertwiners''.  
Furthermore, for the case of higher categories ($n>1$), following the terminology introduced by S.Crans~\cite{C} (see also the ``transfor'' page on the~\cite{n-Lab} and compare also with the works of C.Kachour~\cite{K1,K2} and G.V.Kondratiev~\cite{Ko}), we can similarly introduce $k$-transfors, for $0\leq k\leq n$, as vertically categorified analogs of natural transformations; in particular 0-transfors are functors, 1-transfors correspond to natural transformations, 2-transfors to modifications, 3-transfors to perturbations \dots (in the literature, for n-categories, with $n\geq 2$, there are already slightly different definitions in place, see the remarks in the \cite{n-Lab} page).  

\medskip 

The main idea behind the general definition of ``higher natural transformations'' between a pair of functors $\phi,\psi:\Cs\to\hat{\Cs}$ of $n$-categories, consists of introducing suitable ``homotopies'' between the different sources/targets of the $n$-cells $\phi(x)$ and $\psi(x)$, $x\in \Cs$, and proceeding iteratively, imposing ``intertwining conditions'' that at the level $n$ must consist of a usual commutative diagram of $n$-cells. 
The language of cubical $n$-categories~\cite{BH,BCM} is much more naturally adapted for the description of the 
\hbox{$(p+1)$-cells} generated by homotopies of $p$-cells and, whenever necessary, we will conveniently translate the ``intertwining conditions'' and ``compositions of transfors'' in such cubical setting.

\begin{definition}
Let the  \emph{0-transfors} be just the covariant functors between strict globular $n$-categories. 
\\ 
Given two 0-transfors $\xymatrix{\Cs\rtwocell^{\phi}_\psi{\omit} & \hat{\Cs}}$ of strict globular $n$-categories $(\Cs,\circ_0,\dots,\circ_{n-1})$ and $(\hat{\Cs},\hatcirc_0,\dots,\hatcirc_{n-1})$, with $n\geq 1$, 
a \emph{1-transfor} between $\phi$ and $\psi$ is a map $\Xi:\Cs^{0}\to \hat{\Cs}^1$, $x\mapsto\Xi(x)$ such that: 
\begin{itemize}
\item 
$\psi(x)\hatcirc_0\Xi(x_0)=\Xi({}_0x)\hatcirc_0\phi(x)$ for all $x\in \Cs$, 
\end{itemize}
where $x_p$ and ${}_p x$ denote respectively the unique source and target partial $\circ_p$-identities of $x\in\Cs$.\footnote{
We are using here a strict notion for 1-transfors (as in the treatment provided by F.Borceux~\cite[section~7.3]{Bo} and 
G.V.Kondratiev~\cite[definitions~1.6-1.7-1.8]{Ko}); more generally, one can consider lax natural transformations (see for example T.Leinster~\cite[definition~1.5.10]{L}) adn introduce classes of ``lax'' higher natural transformations (see~S.Crans~\cite{Cr} and C.Kachour~\cite{K1,K2}).}

\medskip 

By recursion, suppose that we already defined (globular) $(k-1)$-transfors between $(k-2)$-transfors. 
\\ 
A \emph{(globular) $k$-transfor} $\xymatrix{\Cs\rtwocell^{\Phi^{(k-1)}}_{\Psi^{(k-1)}}{\ \ \Xi^{(k)}} & \hat{\Cs}}$,  
between $(k-1)$-transfors  $\Phi^{(k-1)},\Psi^{(k-1)}$, \dots, between two functors 
$\Phi^{(0)},\Psi^{(0)}:\Cs\to \hat{\Cs}$, for two strict globular $n$-categories, 
$(\Cs,\circ_0,\dots,\circ_{n-1}), (\hat{\Cs},\hatcirc_0,\dots,\hatcirc_{n-1})$, with $n\geq k$, is a map $\Xi^{(k)}:\Cs^{0}\to \hat{\Cs}^k$, 
$x\mapsto\Xi(x)$ such that: 
\begin{itemize}
\item
$\Phi^{(k-1)}(A)\xrightarrow{\Xi^{(k)}(A)} \Psi^{(k-1)}(A)$ for all $A\in \Cs^0$, 
\item 
$\Psi^{(0)}(x)\hatcirc_0\Xi^{(k)}(x_0)=\Xi^{(k)}({}_0x)\hatcirc_0\Phi^{(0)}(x)$ for all $x\in \Cs$.
\end{itemize}
\end{definition} 

Functors ($k=0$) between small strict $n$-categories constitute a strict 1-category and natural trasformations ($k=1$) constitute a strict 2-category. 
Similarly, by induction, we have the following result (a sketch of the proof is presented in~\cite[proposition~1.4]{Ko} and, for the case $k=2$, in F.Borceux~\cite[section~7.3]{B}) that provides a nice class of examples of strict higher categories constructed inductively. 

\begin{theorem}\label{th: n-tr}
The family of (globular) $n$-transfors between small strict globular $n$-categories becomes a strict $(n+1)$-category.\footnote{More generally, in the same way, the family of (globular) $k$-transfors (for fixed $k\leq n$), of small strict globular $n$-categories, becomes a strict globular 
$(k+1)$-category.}
\end{theorem}

Variants of this result can be explored also for the case of ``lax'' transfors (C.Kachour~\cite{K2}). 

\subsection{Exchange Property and Eckmann-Hilton Collapse}

Whenever $o\in \Cs^q$ and $0\leq q\leq p< m\leq n$, we define the \emph{$q$-diagonal $p$-home-sets of $m$-arrows over $o$} 
\begin{equation*}
{}_p\Cs^m_{oo}:=\{x\in \Cs^m \ | \ s^m_k(x)= \iota^k_q(o) =t^m_k(x), \quad \forall k=q,\dots, p\}
\end{equation*}
as the family ${}_p\Cs^m_{oo}\subset {}_q\Cs^m_{oo}$ of $m$-arrows that share a common source and target 
$k$-arrow $\iota^k_q(o)$ for all $q\leq k\leq p$ and we note that, as a consequence, on such diagonal home-sets all the compositions $\circ_q,\dots,\circ_p$ are well defined global binary operations.    

\medskip 

The following proposition, that is fundamental for the discussion about non-commutativity in the context of $n$-categories, is just a higher-categorical version of the well-known Eckmann-Hilton argument~\cite{EH} (see for example T.Leinster~\cite[proposition~1.2.4]{L} or P.Zito~\cite{Z} in the case of 2-categories); it follows immediately from the exchange law and assures a strong trivialization of the categorical structure. 

\begin{proposition}\label{prop: eh}
If $0\leq q\leq p< m\leq n$ and $o\in \Cs^q$ is a $q$-arrow in an $n$-category $(\Cs,\circ_0,\dots,\circ_{n-1})$,  the $q$-diagonal $p$-home-set ${}_p\Cs^m_{oo}$ is a $(m-q)$-category $({}_p\Cs^m_{oo},\circ_q,\dots, \circ_{m-1})$ and a monoid for all the operations $\circ_q,\dots,\circ_p$. 
If $q<p$, all the operations $\circ_q,\dots,\circ_p$ coincide and they are commutative, hence ${}_p\Cs^m_{oo}$ actually collapses to a $(m-p)$-category that is a commutative monoid for $\circ_p$. 
\end{proposition}
\begin{proof}
If $o\in \Cs^q$, for all $q\leq k<p$, for all $x,y\in {}_q\Cs^m_{oo}$, if $x\circ_k y$ exists, it always belongs to ${}_q\Cs^p_{oo}$ and so we have a category $({}_q\Cs^p_{oo},\circ_q,\dots, \circ_{p-1})$. 
Since $\Cs^q\subset \Cs^k$, if $x,y\in {}_{k}\Cs^p_{oo}$, their compositions $x\circ_q y, \dots, x\circ_{k}y$ are always defined and so $({}_{k}\Cs^p_{oo},\circ_q,\dots,\circ_k)$ is a $(k-q+1)$-monoid. 

Suppose now that $q\leq i<j\leq p-1$, since $o\in\Cs^q\subset \Cs^i\subset\Cs^j$, we have that $o$ is an identity for both the compositions $\circ_i$ and $\circ_j$. 
From the exchange property it follows that, for all $x,y\in {}_k\Cs^p_{oo}$ 
\begin{gather*}
(x\circ_i y)=(x\circ_j o)\circ_i (o\circ_j y)=(x\circ_i o)\circ_j(o\circ_i y)=x\circ_j y, 
\\  
x\circ_i y=(o\circ_j x)\circ_i (y\circ_j o)=(o\circ_i y)\circ_j(x\circ_i o)=y\circ_j x  
\end{gather*}
that shows that $\circ_i=\circ_j$ when restricted to the set ${}_k\Cs^p_{oo}$. 
\end{proof}

If $q<n-1$ and $o\in \Cs^q$ is a $q$-arrow in an $n$-category $(\Cs,\circ_0,\dots,\circ_{n-1})$ the $q$-diagonal home-set ${}_q\Cs^p_{oo}$, with $q<p-1$, is a commutative $p$-category $({}_q\Cs^p_{oo},\circ_q,\dots,\circ_{p-1})$ and all the operations coincide and are commutative when restricted to the home-set ${}_{(p-1)}\Cs^{p}_{oo}$. 

\medskip 

Again a much better intuition comes from the following graphical explanation of the proof. 
If $q< p<n$ and $n$-arrows have a common $q$-source $q$-target $\bullet$: 
$\circ^n_p=\circ^n_q$ and they are commutative operations. 
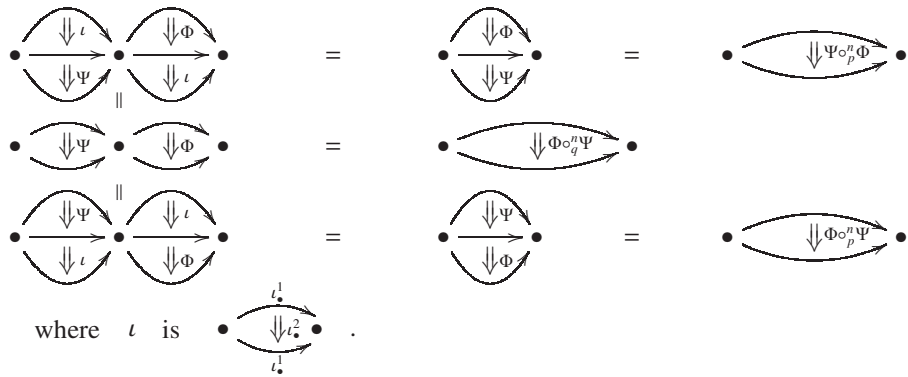
\begin{equation*}
\xymatrix{
\bullet \ruppertwocell{\iota} \rlowertwocell{\Psi} \ar[r] 	& \bullet \ruppertwocell{\Phi} \rlowertwocell{\iota} \ar[r] 	\ar@{}[d]|{\|} & \bullet 
 	& = 
& \bullet \ruppertwocell{\Phi} \rlowertwocell{\Psi} \ar[r] & \bullet   
	& = 
& \bullet \rrtwocell{\quad \ \Psi\circ^n_p\Phi}&   & \bullet   
\\ 
\bullet \rtwocell{\Psi} & \bullet \rtwocell{\Phi} \ar@{}[d]|{\|} & \bullet 
	&= 
& \bullet \rrtwocell{\quad \ \Phi\circ^n_q\Psi} & &  \bullet 
& & 
\\ 
\bullet \ruppertwocell{\Psi} \rlowertwocell{\iota} \ar[r] & \bullet \ruppertwocell{\iota} \rlowertwocell{\Phi} \ar[r] & \bullet 
	&= 
& \bullet \ruppertwocell{\Psi} \rlowertwocell{\Phi} \ar[r] & \bullet   
	&=  
&\bullet  \rrtwocell{\quad \ \Phi\circ^n_p\Psi}&  & \bullet 
\\ 
\ar@{}[r]|{\text{\normalsize where}} & \quad \iota  \ar@{}[r]|{\text{\normalsize is}} & \bullet \rtwocell^{\iota^1_\bullet}_{\iota^1_\bullet}{\iota^2_\bullet}& \bullet \quad .
& & & & & &  
}
\end{equation*}

\subsection{Non-commutative Exchange Property} 

As a possible solution in order to avoid the Eckmann-Hilton collapse of the algebraic structure of \hbox{$q$-diagonal} $p$-home-sets for $q<p$, we propose to relax the form of the exchange property for globular strict $n$-categories and we put forward the following definition.\footnote{We stress that the unique change from definition~\ref{def: ncat} consists in the modified exchange property (the fourth item below).} 
\begin{definition} 
A \emph{globular strict $n$-category with non-commutative exchange} 
$(\Cs,\circ_0,\dots,\circ_{n-1})$ 
is a class $\Cs$ equipped with a family $\circ_0,\dots,\circ_{n-1}$ of $n$ partially defined binary operations 
$\circ_p:\Cs\times\Cs\to\Cs$, for $p=0,\dots,n-1$, such that: 
\begin{itemize}
\item
$(\Cs,\circ_p)$ is a $1$-category for all $p=0,\dots,n-1$, 
\item 
$\Cs^q\subset \Cs^p$ for all $q<p$, i.e.~a $\circ_q$-identity is also a $\circ_p$-identity, 
\item 
$\Cs^p\circ_q\Cs^p\subset \Cs^p$ for all $p,q=0,\dots, n-1$ with $q<p$,  
i.e.~the $\circ_q$-composition of $\circ_p$-identities, whenever exists, is a $\circ_p$-identity, 
\item \emph{non-commutative (quantum) exchange}: 
for all $p$-identities $\iota$, for all $q<p$, the partially defined maps $\iota\circ_q-:(\Cs,\circ_p)\to(\Cs,\circ_p)$ and 
$-\circ_q \iota:(\Cs,\circ_p)\to(\Cs,\circ_p)$ are functorial (homomorphisms of partial 1-monoids). 

An $n$-category with (quantum) exchange will also be referred as a \emph{quantum $n$-category}.\footnote{The reason for the usage of the term ``quantum'' lies in the the fact that non-commutative operations (notoriously essential for a the mathematical formalization of quantum observables) can only survive in the presence of this relaxed exchange property.}  
\end{itemize}
\end{definition}
The graphical representation of the non-commutative exchange property (here $q<p$, $A,B,C\in \Cs^q$, $e,f,g,h\in \Cs^p$, $\Phi,\Psi\in \Cs$), makes immediately clear that this is just the original exchange axiom, required to hold only for the special situation when two of the $n$-cells involved coincide with a given $\circ_p$-identity: 

\begin{gather*}
\xymatrix{A \rtwocell^e_e{\ \iota(e)} & B  \ruppertwocell^f{\Psi} \rlowertwocell_h{\Phi} \ar[r]|{g} 	& C}
= 
\xymatrix{A \ruppertwocell^{e}{\iota(e)} \rlowertwocell_{e}{\iota(e)} \ar[r]|{e} 	& B \ruppertwocell^f{\Psi} \rlowertwocell_h{\Phi} \ar[r]|{g} & C} 
= 
\xymatrix{A \rruppertwocell^{f\circ e}{\quad \Psi\circ \iota(e)} \rrlowertwocell_{h\circ e}{\quad \Phi\circ \iota(e)} \ar[rr]|{g\circ e} &	& C} 
\end{gather*}

\begin{gather*}
\xymatrix{ B  \ruppertwocell^f{\Psi} \rlowertwocell_h{\Phi} \ar[r]|{g} & C \rtwocell^e_e{\ \iota(e)}& D }
=
\xymatrix{ B \ruppertwocell^f{\Psi} \rlowertwocell_h{\Phi} \ar[r]|{g} & C \ruppertwocell^{e}{\iota(e)} \rlowertwocell_{e}{\iota(e)} \ar[r]|{e} &D}
=
\xymatrix{B \rruppertwocell^{e\circ f}{\ \ \ \iota(e)\circ \Psi} \rrlowertwocell_{e\circ h}{\ \ \ \iota(e)\circ \Phi} \ar[rr]|{e\circ g} &	& D .}
\end{gather*} 
An acute reader will remember that the globularity of the $n$-cells was actually encoded in the specific form of the exchange property and might at this point question if our relaxed non-commutative exchange property still implies the globularity condition. Indeed this is true: the right-functoriality in the quantum exchange property assures that for 
$x\in \Cs$, the $\circ_q$-identity $s_q(x)$ is $\circ_p$-composable (on the right) with both the $\circ_q$-identities $s_q(s_p(x))$ and $s_q(t_p(x))$ and so the three of them must coincide. A similar argument using the $\circ_q$-identity $t_q(x)$ and the left-functoriality provides the second globular condition.  

\medskip 

The Eckmann-Hilton collapse (proposition~\ref{prop: eh}) is avoided: in the strict non-commutative $n$-category $(\Cs,\circ_0,\dots,\circ_{n-1})$, for all $q<n-1$, for all $o\in\Cs^q$, the $q$-diagonal $(n-1)$-home-sets of $n$-cells ${}_{n-1}\Cs_{oo}$ can all be non-commutative monoids with respect to the restriction of any one of the operations $\circ_{n-1},\dots,\circ_q$, and these restrictions are not forced anymore to coincide. This will be clear from the examples that are provided in the context of higher C*-categories. 

\medskip 

We are well aware that the proposed non-commutative exchange property is somehow going against the usual lines of development of the subject as  inspired by higher homotopy theory. The full justification for such a questionable, apparently arbitrary, modification of the usual notion of $n$-category, comes from the richness of natural examples of operator structures perfectly fitting with this framework as well as from quite elementary discussion of higher relational environments (higher categories of \hbox{$n$-quivers}) that will be presented further on.
We stress that some of our examples of strict $n$-categories with non-commutative exchange do not necessarily fall within the scope of weak $n$-categories, since they do not satisfy the usual exchange property even up to higher isomorphisms. 

\subsection{Product Categories as Full-depth Strict Higher Categories}\label{sec: full-depth} 

In this subsection we propose a generalization of the previous notion of strict higher category with non-commutative exchange that later on will turn out to be essential for a complete description of the operations between hypermatrices as ``higher convolutions''.   

\medskip 

It is well-known (see for example S.Mac Lane~\cite[section~II.3]{ML}) that the Cartesian product 
$\Xs_1\times\cdots\times \Xs_n$ of a family of $n$ different $1$-categories $(\Xs_k,\circ_k)$, for $k=1,\dots,n$, can be seen either as another $1$-category, with componentwise composition, or as an $n$-tuple category, with $n$ different directional compositions. 

\medskip 

In reality, the strict cubical $n$-category obtained from the Cartesian product $\Xs_1\times\cdots\times\Xs_n$ of 1-categories has a much richer structure and its cubical $n$-cells can generally be composed over $q$-arrows, for all $0\leq q\leq n$, via $2^n=\sum_{q=0}^n \binom{n}{q}$ compositions, $\circ_{j_1\cdots j_q}$, $\binom{n}{q}$ of which are at depth-$q$, one for each subset $\{j_1,\dots,j_q\}\subset \{1,\dots,n\}$, as specified by the following definition:   
\begin{definition}
Let $(\Xs_1,\circ_1), \dots, (\Xs_n,\circ_n)$ be strict 1-categories and $x_k,y_k\in \Xs_k$, for all $k=1,\dots,n$. 
\\ 
We say that $(x_1,\dots, x_n)$ and $(y_1,\dots,y_n)$ are $\circ_{j_1\dots j_q}$-composable if and only if for all $p\in \{j_1\dots,j_q\}$, 
$x_{j_k}=y_{j_k}$ and $(x_p,y_p)$ are composable in $\Xs_p$, for all $p\notin\{j_1,\dots,j_q\}\subset \{1,\dots,n\}$ and in this case
\begin{gather*}
(x_1,\dots, x_n)\circ_{j_1\cdots j_q}(y_1,\dots,y_n):=(t_1,\dots,t_n) \quad  \text{with} \quad 
t_{p}:=
\begin{cases}
x_{p}= y_{p}, \ \text{for all } \  p\in \{j_1,\dots j_q\},   
\\
x_p\circ y_p, \ \text{for all} \ p\notin\{j_1,\dots,j_q\}.
\end{cases}
\end{gather*}
\end{definition}
The usual componentwise composition making $\Xs_1\times\cdots\times\Xs_n$ into the 1-category Cartesian product of the family $(\Xs_k,\circ_k)_{k=1,\dots, n}$ corresponds to the unique case of operation $\circ_{12\ \cdots\ (n-1)n}$ of depth-$n$.  

\medskip 

A graphical representation of the $4=2^n$ composition for the case $n=2$ is here illustrated: 
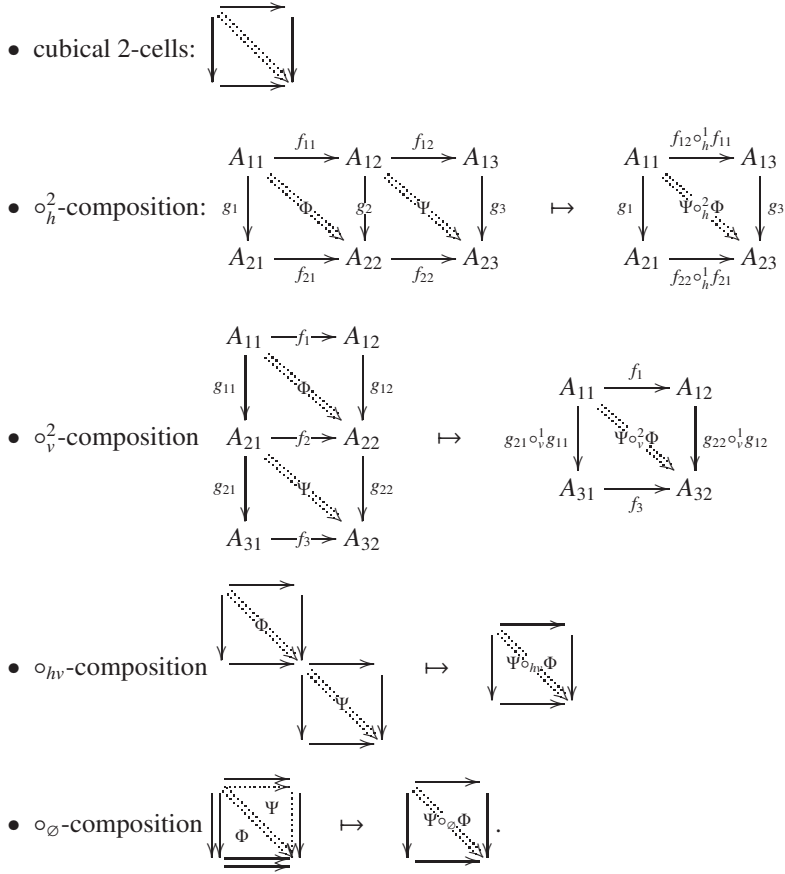
\begin{itemize}
\item 
cubical $2$-cells: $\vcenter{\xymatrix{\ar[r] \ar[d] \ar@{:>}[dr]&  \ar[d] \\  \ar[r]& }}$
\item 
$\circ^2_h$-composition: 
$\vcenter{\xymatrix{
A_{11} \ar[d]_{g_1} \ar[r]^{f_{11}} \ar@{:>}[dr]|{\Phi} & A_{12} \ar[d]|{g_2} \ar[r]^{f_{12}} \ar@{:>}[dr]|{\Psi}& A_{13} \ar[d]^{g_3} 
\\
A_{21} \ar[r]_{f_{21}} & A_{22} \ar[r]_{f_{22}} & A_{23}
}
}
\quad 
\mapsto 
\quad 
\vcenter{\xymatrix{
A_{11} \ar[d]_{g_1} \ar[r]^{f_{12}\circ^1_h f_{11}} \ar@{:>}[dr]|{\Psi\circ^2_h\Phi} & A_{13} \ar[d]^{g_3} 
\\
A_{21} \ar[r]_{f_{22}\circ^1_h f_{21}} & A_{23}
}
}$
\item  
$\circ^2_v$-composition 
$\vcenter{\xymatrix{
A_{11} \ar[d]_{g_{11}} \ar[r]|{f_1} \ar@{:>}[dr]|{\Phi} & A_{12} \ar[d]^{g_{12}} 
\\
A_{21} \ar[r]|{f_2} \ar[d]_{g_{21}}\ar@{:>}[dr]|{\Psi}& A_{22} \ar[d]^{g_{22}}
\\
A_{31} \ar[r]|{f_3} & A_{32}
}
}
\quad 
\mapsto 
\quad 
\vcenter{\xymatrix{
A_{11} \ar[d]_{g_{21}\circ^1_v g_{11}} \ar[r]^{f_1} \ar@{:>}[dr]|{\Psi\circ^2_v\Phi} & A_{12} \ar[d]^{g_{22}\circ^1_v g_{12}} 
\\
A_{31} \ar[r]_{f_3} & A_{32}
}
}
$
\item 
$\circ_{hv}$-composition 
$\vcenter{\xymatrix{\ar[r] \ar[d] \ar@{:>}[dr]|{\Phi}&  \ar[d] & 
\\  
\ar[r]&    \ar[r] \ar[d] \ar@{:>}[dr]|{\Psi}&  \ar[d] 
\\ 
&  \ar[r] &  
}}
\quad 
\mapsto 
\quad 
\vcenter{\xymatrix{\ar[r] \ar[d] \ar@{:>}[dr]|{\Psi\circ_{hv}\Phi}&  \ar[d] \\  \ar[r]& }}
$
\item 
$\circ_\varnothing$-composition 
$\vcenter{\xymatrix{\ar@<1pt>[r] \ar@<-2pt>@{.>}[r] \ar@<1pt>[d] \ar@<-2pt>[d] 
\ar@{:>}[dr]_{\Phi}^{\Psi}&  \ar@<1pt>[d] \ar@<-2pt>@{.>}[d] \\  \ar@<1pt>[r] \ar@<-2pt>[r] & }}
\quad 
\mapsto 
\quad 
\vcenter{\xymatrix{\ar[r] \ar[d] \ar@{:>}[dr]|{\Psi\circ_\varnothing\Phi} &  \ar[d] \\  \ar[r]& }}.
$
\end{itemize}
Motivated by this elementary Cartesian product example, we tentatively put forward an axiomatization of a higher strict categorical structure suitable for treating such situations. 
\begin{definition}
A \emph{full-depth strict cubical $n$-category} $(\Cs,\{\circ_\gamma\}_{\gamma\subset \{1.\dots,n\}})$ is a class $\Cs$ equipped with a family of $2^n$ partially defined compositions $\circ_\gamma$, one for each subset $\gamma\subset\{1.\dots,n\}$ satisfying:
\begin{itemize}
\item 
$(\Cs,\circ_\gamma)$ is a strict 1-category for all $\gamma\subset\{1,\dots,n\}$, 
\item 
$\Cs^\alpha\subset \Cs^\beta$, whenever $\beta\subset\alpha$ and $\Cs^\gamma$ denotes the partial identities of $(\Cs,\circ_\gamma)$, 
\item
$\Cs^\beta\circ_\alpha\Cs^\beta\subset \Cs^\beta$, for all $\beta\subset\alpha$, 
\item 
non-commutative (quantum) exchange: for all $\circ_\gamma$-identities $\iota\in \Cs^\gamma$ and for all $\gamma\subset\alpha$, the left/right 
\hbox{$\circ_\alpha$-compositions} $\iota\circ_\alpha -:(\Cs,\circ_\gamma)\to(\Cs,\circ_\gamma)$ and 
$-\circ_\alpha\iota:(\Cs,\circ_\gamma)\to(\Cs,\circ_\gamma)$ are functorial. 
\end{itemize}
\end{definition}
The above definition essentially works for any arbitrary ordered set of indices. Remarkably, it is the specific choice of the ordered set of indices that (via such axioms) determines the geometrical / combinatorial shape of the $n$-cells: when the set of indices is the power set of $\{1,\dots,n\}$ we get cubical cells; when the index set is the set of cardinals less than $n$, we get globular $n$-cells. Other choices different from these two will produce more exotic shapes \dots\ and this is another even stronger departure from the usual world of higher categories inspired by higher homotopy theory. 

\section{Strict Involutive Higher Categories} \label{sec: *}

In this section we discuss a possible notion of strict involution in the context of strict higher categories. 
From the case of 1-categories, we know that there are actually several different ways in which involutions and dualities have been introduced in the categorical context, either via involutive endo-functors, or via dualities implemented via adjointability, or dualizable objects.\footnote{For a more detailed discussion of other approaches to ``categorical involutivity/duality'' we suggest J.Baez-M.Stay~\cite{BS}, the \texttt{n-lab} page 
\hlink{http://ncatlab.org/nlab/show/category+with+duals}{Categories with Duals} and references therein.} A full comparison between these different notions deserves a separate treatment elsewhere; for our purpose here, involutions (dualities) will be defined as involutive functors with specific covariance and contravariance properties with respect to the several compositions. This point of view is directly inspired by the notion of $*$-category introduced (with additional linearity assumptions) by P.Ghez-R.Lima-J.E.Roberts and P.Mitchener~\cite{GLR,M} as a horizontal categorification of the involution of a $*$-algebra. Categories equipped with involutive (contravariant) endofunctors have been studied since the works of M.Burgin~\cite[section~2]{Bu} and J.Lambek~\cite{La}; with the denomination ``dagger categories'' they have been axiomatized by P.Selinger~\cite{S} and are now systematically used by S.Abramsky-B.Coecke~\cite{AC} and collaborators in their works on categorical quantum mechanics via compact closed categories (see also C.Heunen-B.Jacobs~\cite{HJ}).  
Weak forms of such involutive categories appear in our definition of \hbox{$*$-monoidal} category~\cite{BCL4} and similar structures were independently developed by J.Egger~\cite{Eg} and B.Jacobs~\cite{J}. 
Involutions in the context of monoidal C*-categories and (weak) 2-categories have actually been present since the initial works by R.Longo-J.E.Roberts~\cite{LR} (see details in the subsequent section~\ref{rem: conjugates}) and have been independently introduced also in J.Baez-L.Langford~\cite[definition~10]{BL}, T.Hayashi-S.Yamagami~\cite[definition~1.3]{HY},  A.Henriques-D.Penneys-\cite[definition~2.3]{HP}, (see also A.Henriques~\cite[definition~2.3]{He2} and D.Penneys-C.Jones~\cite[definition~2.4]{JP}). 
Weak higher involutions, for weak $\omega$-categories (in Penon's approach~\cite{P,K1}) are introduced in~\cite{BB}. 

\subsection{Strict Higher Involutions}

A graphical display of ``duals'' of $n$-cells helps to grasp the intuition behind the formal definitions:  
\begin{itemize}
\item 
dual cells for $1$-arrows (usual involution): $\xymatrix{A \ar[r]^f & B}\quad \mapsto \quad \xymatrix{A & B \ar[l]_{\ \ f^*}}$
\item
dual cells for globular $n$-arrows ($A,B\in \Cs^q$, $f,g\in\Cs^p$, $\Phi\in \Cs^m$, $0\leq q<p<m\leq n$): \\ 
$*_p: \xymatrix{A \rtwocell^f_g{\Phi}& B}\quad \mapsto \quad \xymatrix{A \rrtwocell^f_g{^{\quad \quad \quad \ \ {\Phi^{*_p}}}}& & B,}$
\quad  
$*_q: \xymatrix{A \rtwocell^f_g{\Phi}& B}\quad \mapsto \quad \xymatrix{A &  & \lltwocell_{f^{*_q}}^{g^{*_q}}{^{\quad \ \Phi^{*_q}}}B,}$
$*_{q p}: \xymatrix{A \rtwocell^f_g{\Phi}& B}\quad \mapsto \quad 
\xymatrix{A& & \lltwocell_{f^{*_q}}^{g^{*_q}}{\quad \quad \quad \ \ \ \Phi^{*_{q p}}}B,}$
\quad 
$*_{\varnothing}: \xymatrix{A \rtwocell^f_g{\Phi}& B}\quad \mapsto \quad 
\xymatrix{A \rrtwocell^f_g{\quad \Phi^{*_\varnothing}} & & B.}$
\end{itemize} 
For the general case of globular $n$-cells we have $2^n$ duals $*_\alpha$ (including the identity $*_\varnothing$) exchanging $q$-sources / $q$-targets for $q$ in an arbitrary set $\alpha\subset\{0,\dots,n-1\}$. In the previous diagrams, with some abuse of notation, we wrote $*_{p}$ for $*_{\{p\}}$ and $*_{q p}$ for $*_{\{q,p\}}$; this last duality can be realized as composition of $*_p$ and $*_q$. 

\medskip 

With the ideological point of view that an involution/duality in category theory should be considered, on the same level of the binary operations of composition, as a ``$1$-ary operation'' of the structure, we introduce the following definition of a strict involutive globular $n$-category via strict 
$n$-functors.

\begin{definition}
If $n\in \NN_0$, and $\alpha\subset\{0,\dots, n-1\}$, an \emph{$\alpha$-contravariant functor}\footnote{Notice that for $\alpha=\varnothing$ we recover the definition of covariant functor.} between two strict globular $n$-categories 
$(\Cs_1,\circ_0,\dots,\circ_{n-1})\xrightarrow{\phi}(\Cs_2,\hatcirc_0,\dots, \hatcirc_{n-1})$, is a map $\phi:\Cs_1\to\Cs_2$ 
such that:  
\begin{itemize}
\item 
for all $q\notin\alpha$,  
whenever $x\circ_q y$ exists, $\phi(x)\hatcirc_q\phi(y)$ also exists and in this case 
$\phi(x\circ_q y)=\phi(x)\hatcirc_q \phi(y)$, 
\item 
for all $q\in \alpha$, 
whenever $x\circ_q y$ exists, $\phi(y)\hatcirc_q\phi(x)$ also exists and in this case 
$\phi(x\circ_q y)=\phi(y)\hatcirc_q \phi(x)$, 
\item
if $e\in \Cs^q_1$ is a $\circ_q$-identity, $\phi(e)\in \Cs^q_2$ is a $\hatcirc_q$-identity. 
\end{itemize} 

An \emph{$\alpha$-involution} $*_\alpha$ on $(\Cs_,\circ_0,\dots,\circ_{n-1})$ is an $\alpha$-contravariant endofunctor such that $(x^{*_\alpha})^{*_\alpha}=x, \forall x\in \Cs$.  

If $\{*_\alpha\ | \ \alpha\in \Lambda\}$, with $\Lambda\subset\P(\{0,\dots,n-1\})$ (the power-set of $\{0,\dots,n-1\}$) is a family of commuting $\alpha$-involutions, the strict globular $n$-category is said to be \emph{$\Lambda$-involutive}.
\end{definition}

In practice, an $\alpha$-involution is an involution that is a unital homomorphism for all $\circ_q$-compositions with $q\notin \alpha$ and a unital anti-homomorphism for $\circ_q$-compositions with $q\in \alpha$.  

\medskip 

Whenever the family $\alpha\subset \{0,\dots,n-1\}$ consists of a singlet $\alpha=\{q\}$, we will simply use the notation $*_q:=*_{\{q\}}$ and in this particular case we will make use of the following terminology: 
\begin{definition}
Let $(\Cs_,\circ_0,\dots,\circ_{n-1})$ be a strict $n$-category and let $q\in\{0,\dots,n-1\}$.  
We say that $(\Cs_,\circ_0,\dots,\circ_{n-1})$ is equipped with an \emph{involution over $q$-arrows}, if there exists an \hbox{$\{q\}$-involution} i.e.~a map $*_q:\Cs\to\Cs$ such that 
\begin{itemize}
\item 
for all $p\neq q$ if $(x\circ_p y)^{*_q}$ exists, $x^{*_q}\circ_p y^{*_q}$ also exists and they coincide, 
\item 
for $p=q$, if $(x\circ_p y)^{*_q}$ exists, $y^{*_q}\circ_p x^{*_q}$ also exists and they coincide, 
\item 
for all $p$, 
if $x$ is a $p$-identity, $x^{*_q}$ is also a $\circ_p$-identity, 
\item
$(x^{*_q})^{*_q}=x$ for all $x\in \Cs$.  
\end{itemize}
The involution $*_q$ is said to be \emph{Hermitian} if:  
\begin{itemize}
\item
$x^{*_q}=x$ whenever $x$ is a $\circ_q$-identity. 
\end{itemize}
If the strict $n$-category is $\Lambda$-involutive and $\{q\},\{p\}\subset \Lambda$, we further impose the commutativity condition: 
\begin{itemize}
\item
$(x^{*_q})^{*_p}=(x^{*_p})^{*_q}$, for all $x\in \Cs$. 
\end{itemize}
A \emph{fully involutive strict $n$-category} is a strict $n$-category that is equipped with a $q$-involution for every $q=0,\dots,n-1$.  
A strict $n$-category is \emph{partially involutive} if it is equipped with only a proper subset of the family of involutions $*_q$, for $q=0,\dots,n-1$. 
\end{definition} 
It is immediate to check that, for all $x\in \Cs^m$, $s^m_q(x^{*_q})=t^m_q(x)$, $t^m_q(x^{*_q})=s^m_q(x)$ and $s^m_p(x^{*_q})=s^m_p(x)$, $t^m_p(x^{*_q})=t^m_p(x)$, when $p\neq q$. Similarly $\iota^m_p(x^{*_q})=(\iota^m_p(x))^{*_q}$. If $x\in \Cs^p$ and the $\{q\}$-involution $*_q$ is Hermitian, we also have $x^{*_q}=x$, for all $q\geq p$. 

\begin{remark}\label{rem: conj}
As a specific illustration of the previous definitions and also in view of a more direct comparison with the already existing literature on 2-C*-categories (in section~\ref{rem: conjugates}), we present here a detailed list of the properties required on a fully involutive 2-category $(\Cs,\otimes,\circ,\cj{\ \ },*)=(\Cs,\circ_0,\circ_1,*_0,*_1)$.
\\ 
This is a 2-category $(\Cs,\otimes,\circ )$ with two involutions, $\cj{\ \ }$ over objects, and $*$ over 1-arrows, such that: 
\begin{gather*}
(x^*)^*=x, 
\quad 
(x\otimes y)^*=x^*\otimes y^*,
\quad 
(x\circ y)^*=y^*\circ x^*, 
\\ 
e^*\in \Cs^1, \ \text{for all $\circ$-identities} \ e\in \Cs^1,   
\quad 
e^*\in \Cs^0, \ \text{for all $\circ$-identities} \ e\in \Cs^0,    
\\  
\cj{(\cj{x})}=x, 
\quad 
\cj{(x\circ y)}=\cj{x}\circ \cj{y}, 
\quad 
\cj{(x\otimes y)}=\cj{y}\otimes \cj{x},
\\  
\cj{e}\in\Cs^1 \ \text{for all $\otimes$-identities} \  e\in \Cs^1, 
\quad 
\cj{e}\in \Cs^0, \ \text{for all $\circ$-identities} \ e\in \Cs^0,    
\\ 
\cj{(x^*)}=(\cj{x})^*. 
\end{gather*}
The Hermitianity of $*$ means $e^*=e$, for all $e\in \Cs^1$; the Hermitianity of $\cj{\ \ }$ means $\cj{e}=e$, for all $e\in \Cs^0$.

The Hermitianity of $\cj{\ \ }$ is trivially satisfied when $\Cs^0$ consists of only one element, i.e.~in the case of a monoidal (tensorial) category; furthermore, in this case, for all $e\in \Cs^0$, $e^*=e$. 
\xqed{\lrcorner}
\end{remark}

\begin{remark}
If a strict $n$-category is $\Lambda$-involutive and $\alpha,\beta\in \Lambda$, then it is also $(\alpha\Delta\beta)$-involutive\footnote{Here $\Delta$ denotes the set-theoretic symmetric difference.} with involution 
$*_{\alpha\Delta\beta}:=*_\alpha\circ *_\beta$ and hence the strict $n$-category is actually $<\Lambda>$-involutive, where the symbol $<\Lambda>\ \subset\P(\{0,\dots,n-1\})$ denotes the family of sets generated by the symmetric difference of sets of $\Lambda$. 
This is actually an abelian group under set-difference that is isomorphic to the group of ``automorphisms'' generated by 
$\{*_\alpha \ | \ \alpha\in \Lambda\}$.   
The maximal abelian group obtainable in this way consists of $(\P(\{0,\dots,n-1\}),\Delta)$, it has cardinality $2^n$ and has a very convenient set of generators given by $\{\{q\} \ | \ q=0,\dots,n-1  \}$ corresponding to the involutions $\{*_q \ | \ q=0\dots, n-1\}$ described here above.    

A strict $n$-category is fully involutive if it is equipped with a family of involutions 
$\{*_\alpha\ | \ \alpha\in \Lambda\}$ that generates such a maximal abelian group with $n$-generators (that is always isomorphic to $\ZZ_2^n$). 
\xqed{\lrcorner}
\end{remark}

\begin{remark}
In principle it is perfectly possible for a strict $n$-category to be equipped with different (commuting) involutions $*_\alpha,\dag_\alpha$ with the same covariance $\alpha$. In this case we say that the involutive strict $n$-category has \emph{involutive multiplicity}. In our treatment here, we assume that $\alpha\mapsto *_\alpha$ is a map, since we are only interested in the internal self-duality of the strict $n$-category, rather than its ``dual-multiplicity''.  
\xqed{\lrcorner}
\end{remark}

\begin{remark}
For $\alpha\neq\varnothing$, a strict globular $n$-category $\Cs$ is $\alpha$-involutive with $\alpha$-involution $*_\alpha$ if and only if $(\Cs, *_\alpha)$ is an $\alpha$-dual of $\Cs$, i.e.~$\Cs$ is ``$\alpha$-self-dual''.\footnote{This means that the pair $(\Cs,*_\alpha)$ satisfies the following universal factorization property: for every $\alpha$-contravariant functor $\phi:\Cs\to\Ds$ into another strict globular $n$-category $\Ds$, there exists a unique covariant functor $\hat{\phi}:\Cs\to\Ds$ such that $\phi=\hat{\phi}\circ *_\alpha$.}   
It is in this sense that ($\alpha$-)involutions on a strict globular $n$-category provide a way to ``internalize'' the ($\alpha$-)dualities.  
\xqed{\lrcorner}
\end{remark}

The previous remark is fundamental to understanding our choice of formalization of the definition of ``fully involutive higher category'': we are requesting the self-dualizability of the category for all possible choices of $\alpha$-duals, selecting a minimal family of $\alpha$-involutions that are adequate for the purpose. 

\subsection{Examples of Strict Involutive Higher Categories}

The most elementary examples of strict involutive $n$-categories come from a \emph{strictification} of the usual (weak) $n$-categories of higher spans~\cite{Be,n-Lab2} (``bipartite quivers'' between pairs sets). 
\begin{example}\label{ex: spans}
The strict 1-category of relations between sets, with the operation of composition of relations, is involutive when we define, for every relation $f\subset A\times B$ from the set $A$ to the set $B$, its \emph{reciprocal relation} 
$f^*:=\{(b,a) \ | \ (a,b)\in f \}\subset B\times A$ from $B$ to $A$.\footnote{The 1-category of functions between sets is not involutive, since the reciprocal relation of a function, usually is not a function.} 

\medskip 

To generalize this basic example to arbitrary level-$n$, it is convenient to be able to possibly consider different ``links'' between the same pair of elements $a\in A$ and $b\in B$. For this purpose we consider a \emph{1-span} from $A$ to $B$ i.e.~a pair of maps $A\xleftarrow{s} R \xrightarrow{t} B$. 
Each element $r\in R$ is interpreted as an arrow connecting its source point $s(r)\in A$ to its target point $t(b)\in B$.

\medskip 

Construct now the free 1-category $[\Qs]$ generated by the 1-graph $\Qs$ consisting of a certain family of 1-spans between sets.\footnote{If $(R_\alpha,s_\alpha,t_\alpha)$, for $\alpha\in \Lambda$ is a family of bipartite 1-spans, $\Qs:=\uplus_{\alpha\in \Lambda}R_\alpha$ is the index set of edges of an oriented \hbox{1-multigraph}, possibly with loops, whose source and targets are the unique maps $s,t$ with restrictions on $R_\alpha$ coinciding with $s_\alpha,t_\alpha$, for all $\alpha\in \Lambda$ and vertex set $\cup_{\alpha\in \Lambda}(A_\alpha\cup B_\alpha)$.} 
The 1-arrows in $[\Qs]$ are finite sequences\footnote{We include here, for $k=0$, one ``empty'' sequence $(a,a)$, for all $a\in \Qs^0$.} $(a_1,r_1,\dots,r_k,a_k)$, such that for all $j=2,\dots,k\in\NN$, $s(r_{j-1})=t(r_j)$, with source $a_k:=s(r_k)$ and target $a_1:=t(r_1)$ and with composition given by the concatenation of finite sequences defined as follows: 
\begin{equation*}
(a_1,r_1,\dots,r_k,a_k)\circ(a_{k+1},r_{k+1},\dots,r_{k+l},a_{k+l}):=(a_1,r_1,\dots,r_{k+l},a_{k+l}), \quad \text{only when $a_k=a_{k+1}$.}
\end{equation*}  
The strict 1-category $[\Qs]$ contains a disjoint copy of each 1-span of the original family and can be used to ``strictify'' the usual composition of spans obtaining a strict 1-category of 1-spans. 

Such a strict \hbox{1-category} is not yet involutive. To obtain an involutive strict category, notice that every 1-span 
$A\xleftarrow{s} R \xrightarrow{t} B$ has a \emph{dual 1-span} $B\xleftarrow{\cj{s}} \cj{R} \xrightarrow{\cj{t}} A$, with  $\cj{R}:\{\cj{r}\ | \ r\in R\}$ a disjoint copy of $R$ and where $\cj{s}(\cj{r}):=t(r)$, $\cj{t}(\cj{r}):=s(r)$. 
The strict 1-category $[\Qs\cup\cj{\Qs}]$ generated by the union of the original family of 1-spans and their dual 1-spans is now naturally equipped with an involution given by: $(a_1,x_1,\dots,x_k,a_k)^*:=(a_k,{x^*}_k,\dots,{x^*}_1,a_1)$, where $x^*:=\cj{x}$ if $x\in \Qs$, and $x^*:=\cj{x}$ if $x\in\cj{Q}$. Since the strict involutive 1-category $[\Qs\cup\cj{\Qs}]$ contains disjoint copies of the original spans (and their duals), it can be used to define a ``strictified'' notion of involution in the category of spans obtaining a strict involutive 1-category of 1-spans. 

\medskip 

A \emph{$n$-span} is a sequence of 1-spans 
$R^{(n)}\rightrightarrows R^{(n-1)}\rightrightarrows \cdots\rightrightarrows R^{(0)}$, where at each level $q=0\dots,n-1$, $R^{q}:=A^{(q)}\cup B^{(q)}$ and $B^{(q)}\xleftarrow{ \ t^{q+1}_{q}}R^{(q+1)}\xrightarrow{s^{q+1}_{q}}A^{(q)}$ is a 1-span from $A^{(q)}$ to $B^{(q)}$. 
We restrict the attention to \emph{globular $n$-spans} that are those $n$-spans that satisfy the globularity condition 
$s^{q+1}_q\circ s_{q+1}^{q+2} =s_q^{q+1}\circ t_{q+1}^{q+2}$ and 
$t_q^{q+1}\circ t_{q+1}^{q+2}=t_q^{q+1}\circ s_{q+1}^{q+2}$, for $q=0,\dots,n-2$. 

Every $n$-span $R$ admits $2^n$ different \emph{$\alpha$-dual $n$-spans} 
$R\xrightarrow{\cj{\phantom{..}}^\alpha}\cj{R}^\alpha$, for $\alpha\subset\{1,\dots,n\}$, consisting of a disjoint copy $\cj{R}^{\alpha (m)}$ of the sets $R^{(m)}$, for $m=0,\dots,n$, with sources and targets maps given, whenever $q\in \alpha$ by $\cj{s}^m_q(\cj{r}^\alpha):=t^m_q(r)$, $\cj{t}^m_q(\cj{r}^\alpha):=s^m_q(r)$; and whenever $q\notin \alpha$ by 
$\cj{s}^m_q(\cj{r}^\alpha):=s^m_q(r)$, $\cj{t}^m_q(\cj{r}^\alpha):=t^m_q(r)$. 

Given again a family of globular $n$-spans, we consider the globular $n$-graph $\Qs\cup \cj{\Qs}$, whose globular $n$-arrows belong to at least one of the $n$-spans in the given family or to one of their duals, and construct the \emph{free globular $n$-category} $\Qs\cup\cj{\Qs}\xrightarrow{\theta}[\Qs\cup\cj{\Qs}]$ generated by the globular $n$-span  $\Qs\cup\cj{\Qs}$.\footnote{A construction of the free globular strict $n$-category of a globular $n$-graph can be found in T.Leinster~\cite[appendix~F]{L}; a left-adjoint functor to the forgetful functor from globular $\omega$-categories to globular $\omega$-graphs is also described in J.Penon~\cite[proposition~1]{P}. 
Quotient constructions of free (involutive) $\omega$-categories over a globular $\omega$-graph are presented in~\cite{BB}.} 

The strict globular $n$-category $[\Qs\cup\cj{\Qs}]$ is naturally equipped with an $\alpha$-involution, for all 
$\alpha\in \{0,\dots,n\}$, obtained by the universal factorization property for the free $n$-category $\Qs\cup\cj{\Qs}\xrightarrow{\theta}[\Qs\cup\cj{\Qs}]$, from the covariant morphism of 
$n$-spans $\Qs\cup\cj{\Qs}\xrightarrow{\gamma_\alpha}\widehat{[\Qs\cup\cj{\Qs}]}^\alpha$ into (the underlying $n$-span of) the \emph{abstract \hbox{$\alpha$-dual} category}\footnote{
The abstract $\alpha$-dual $n$-category of a strict $n$-category $\Cs$ is an $\alpha$-contravariant functor $\Cs\xrightarrow{\widehat{\phantom{..}}^\alpha} \widehat{\Cs}^\alpha$ into a strict $n$-category $\widehat{\Cs}^\alpha$ that uniquely factorizes, via covariant functors, any $\alpha$-contravariant functor into another $n$-category.} 
$[\Qs\cup\cj{\Qs}]\xrightarrow{\widehat{\phantom{..}}^\alpha}\widehat{[\Qs\cup\cj{\Qs}]}^\alpha$, where 
$\gamma_\alpha:=\widehat{\phantom{..}}^\alpha\circ\theta\circ \cj{\phantom{..}}^\alpha$ is the composition of the 
$\alpha$-duality morphism of $n$-spans $\Qs\cup\cj{\Qs}\xrightarrow{\cj{\phantom{..}}^\alpha}\Qs\cup\cj{\Qs}$, first with the covariant inclusion morphism $\theta$ into the free $n$-category, and then with the $\alpha$-contravariant isomorphism of $n$-categories $[\Qs\cup\cj{\Qs}]\xrightarrow{\widehat{\phantom{..}}^\alpha}\widehat{[\Qs\cup\cj{\Qs}]}^\alpha$. 

The strict globular fully involutive $n$-category $[\Qs\cup\cj{\Qs}]$ contains a (disjoint) copy of every globular  $n$-span of the original family (and of each of their $\alpha$-duals) and can now be used to obtain a strictification of the usual weak (involutive) $n$-category of globular spans of sets. 
For this purpose, one simply considers the \emph{coarse graining} of $[\Qs\cup\cj{\Qs}]$ i.e.~the strict globular $n$-category $\Pf([\Qs\cup\cj{\Qs}])$ whose elements are the subsets of $[\Qs\cup\cj{\Qs}]$ under term-by-term compositions (and term-by-term involutions) and finally selects inside $\Pf([\Qs\cup\cj{\Qs}])$ the strict globular (involutive) $n$-category generated  by the disjoint family of the original $n$-spans (and their $\alpha$-duals) embedded into $[\Qs\cup\cj{\Qs}]$. 

\medskip 

The fully involutive category $\Rs$ of globular $n$-relations is just a special case of the construction above, since every relation $R\subset A\times B$  canonically determines a 1-span from $A$ to $B$ via the restriction of the Cartesian projection to $R$. Unfortunately, such strict globular $n$-category of relations is degenerate above $k=1$ because $\Rs^k=\Rs^1$ for $1<k\leq n$:\footnote{In practice all the $n$-cells $x\in \Rs^n$ coincide with higher identities corresponding to 1-cells: $x=\iota^n_1(x)$.} in fact the globularity condition imposed on the $n$-cells in $[\Rs^n]$ implies that if $((a_1,b_1),(a_2,b_2))\in R\in \Rs^2$, necessarily $a_1=a_2$, $b_1=b_2$ and so 2-cells (and similarly higher cells) in $[\Rs^n]$ are identities. 
This justifies the need to consider globular $n$-spans. 

Eliminating the globularity condition is not going to solve the problem: for non-globular $n$-spans the (non-commutative) exchange property is not satisfied. 
\xqed{\lrcorner} 
\end{example}

\begin{example} 
Strict globular $n$-groupoids (see the papers by R.Brown-P.Higgins~\cite{BH} and the their book with R.Sivera~\cite{BHS}) are of course a special case of fully involutive strict globular \hbox{$n$-categories}, where the role of the involutions is taken by the inverse maps. 
\xqed{\lrcorner} 
\end{example}

\begin{example}
Consider the class of involutive monoids (or even more specifically unital algebras) and the family of unital, not necessarily involutive, homomorphisms between them, with the operation of functional composition. The composition of unital homomorphisms is a unital homomorphism, the composition is associative and every monoid is equipped with an identity map that is a unital homomorphism that satisfies the identity property and hence we have a 1-category. 

The involution on the monoids can be used to introduce a covariant involution $\phi\mapsto \phi^*$ of 1-arrows: given two unital involutive monoids $(A_1,\cdot_1,\dag_1)$, $(A_2,\cdot_2,\dag_2)$ and a unital homomorphism $\phi:A_1\to A_2$, define $\phi^{*}(x):=\phi(x^{\dag_1})^{\dag_2}$, for all $x\in A_1$, note that $\phi^*:A_1\to A_2$ is another unital homomorphism (it coincides with $\phi$ if and only if $\phi$ is a $*$-homomorphism) and that the map $\phi\mapsto \phi^*$ is a covariant involution. This is an example of $\varnothing$-involution that is Hermitian (since it does not move the objects). 
\xqed{\lrcorner}
\end{example} 

\begin{example}\label{ex: intertwiners}
Consider now the 2-category whose 1-arrows are the unital $*$-homorphisms $A_1\xrightarrow{\phi}A_2$ of unital involutive monoids $(A_j,\cdot_j,\dag_j), j=1,2$, and whose 2-arrows $\phi\xrightarrow{e}\psi$, for $\phi,\psi:A_1\to A_2$, are the intertwiners of pairs of unital homomorphisms of unital involutive monoids, i.e.~those elements $e\in A_2$ such that $e\cdot_2\phi(x)=\psi(x)\cdot_2 e$, for all $x\in A_1$. 
Given three 1-arrows $\phi,\psi,\eta:A_1\to A_2$ and two intertwiners $\phi\xrightarrow{e_2}\psi\xrightarrow{e_1}\eta$, their composition over 1-arrows is given by $e_1\circ^2_1 e_2:=e_1\cdot_2 e_2$ that is an intertwiner from $\phi$ to $\eta$. 
Given instead $\phi_2\xrightarrow{e_2}\psi_2$ and $\phi_1\xrightarrow{e_1}\psi_1$, where $\phi_1,\psi_1:A_2\to A_3$ and $\phi_2,\psi_2:A_1\to A_2$, the composition over objects is given by $e_1\circ^2_0 e_2:=e_1\cdot_3\phi_1(e_2)$ that is an intertwiner from $\phi_1\circ \phi_2$ to $\psi_1\circ\psi_2$. An involution of 2-arrows over 1-arrows is obtained as follows: if $\phi\xrightarrow{e}\psi$ is an intertwiner from $\phi$ to $\psi$, the element $e^{\dag_2}\in A_2$ is an intertwiner from $\psi$ to $\phi$. In general there is no involution of 2-arrows over objects and so this is an example of a partially involutive strict 2-category, whose only involution is $*_\alpha$ with $\alpha=\{1\}$. 

Restricting to the case of intertwiners between unital $*$-isomorphisms of $*$-monoids, it is not difficult to check that one obtains a fully involutive 2-category, where the additional involutions over objects $*_\alpha$, with $\alpha=\{0\}$, is given, for every $\xymatrix{{A_1}\rtwocell^{\phi}_{\psi}{e}& {A_2}}$ by $\xymatrix{{A_2}\rtwocell^{\phi^{-1}}_{\psi^{-1}}{e^{*_\alpha}} &  {A_1}}$, with $e^{*_\alpha}:=\phi^{-1}(e^{\dag_2})$.
\xqed{\lrcorner}
\end{example}

\begin{example}\label{ex: intertwiners2} 
There is a \emph{horizontal categorification} of example~\ref{ex: intertwiners}.
An involutive \hbox{1-category} $(\Cs,\circ_0,*_0)$ is a horizontal categorification (a ``many objects'' version) of an involutive monoid. 

A family $\Cf^{(2)}:=\{(\Cs_k,\cdot^k_0,\dag^k_0) \ | \ k\in \Lambda\}$ of involutive 1-categories becomes a strict 1-category with \hbox{1-arrows} consisting of the $*$-functors $\Cs_1\xrightarrow{\phi,\psi} \Cs_2$, and a strict globular 2-category with 2-arrows $\phi\xrightarrow{e}\psi$ consisting of natural transformations (the horizontal categorification of intertwiners) $(e_o): o\mapsto e_o$, with $o\in \Cs^0_1$ and 
$e_o\in \Cs_2^1$. 
The category $\Cf^{(2)}$ is partially involutive, with involution of 2-arrows over 1-arrows given by $(e_o)^{*_1}=(e^{\dag^2})_o$, for all $o\in \Cs^0_1$, where $\xymatrix{{\Cs_1}\rtwocell^{\phi}_{\psi}{e}& {\Cs_2}}$ is a natural transformation. 

As in example~\ref{ex: intertwiners}, restricting to the case of invertible $*$ functors, the category $\Cf^{(2)}$ becomes a fully involutive strict 
globular 2-category. 
\xqed{\lrcorner}
\end{example}

It is natural to ask whether it is possible to embed (as an involutive 2-subcategory) the fully involutive 2-category of intertwiners of unital 
$*$-isomorphisms of $*$-monoids, in the previous example~\ref{ex: intertwiners}, into a fully involutive \hbox{2-category} including (as a non-fully involutive 2-subcategory) the 2-category of intertwiners between unital $*$-homomorphisms of $*$-monoids. 
In order to introduce the missing \hbox{$\alpha$-involutions}, for $\alpha=\{0\}$, we will have to generalize the notion of a unital \hbox{$*$-homomorphism} between monoids, along the same lines leading from functions to relations and to spans in example~\ref{ex: spans}. 

\begin{example} 
Every $*$-homomorphism $\phi:A_1\to A_2$ between involutive monoids uniquely determines a \emph{congruence} (and hence a span) of involutive monoids i.e.~$\phi:=\{(a_1,a_2)\in A_1\times A_2\ | \ a_2=\phi(a_1)\}$ is a unital involutive submonoid of the product involutive monoid $A_1\times A_2$. 
Since the reciprocal relation $\phi^*\subset A_2\times A_1$ of any congruence $\phi\subset A_1\times A_2$ is another congruence, we immediately obtain an involutive 1-category of congruences of involutive monoids.

\medskip 

More generally one has an involutive 1-category of spans $A_1\xleftarrow{s} \phi \xrightarrow{t} A_2$ of involutive monoids 
(cf.~example~\ref{ex: spans}), where the source and target maps are involutive morphisms of involutive monoids.  
\xqed{\lrcorner}
\end{example}

One might wonder whether it is possible to identify in this setting a notion of ``intertwiner'' between ``morphisms'' of $*$-monoids that is ``involutive'' and naturally produces a fully involutive 2-category (and later use such notion to generalize $n$-transfors in a way suitable for fully involutive higher categories). Although it is relatively easy to obtain notions of intertwining between spans that admit $*_1$-involutions, for $*_0$-involutions 
a more radical approach via ``relational bivariant intertwiners'' is needed. 

\begin{remark}
Let $A_1 \xleftarrow{s_\phi} \phi \xrightarrow{t_\phi}A_2$ and $A_1 \xleftarrow{s_\psi} \psi \xrightarrow{t_\psi}A_2$ be spans between a pair of $*$-monoids $A_1,A_2$. 

A \emph{bivariant intertwiner} $\xymatrix{A_1 \rtwocell^\phi_\psi{\Xi} & A_2}$ between $\phi$ and $\psi$ is any family $\Xi$ consisting of some quadruples $(e,\psi,\phi,f)$ with $(e,f)\in A_2\times A_1$ such that the following intertwining conditions hold:\footnote{The usage of quadruples $(e,\psi,\phi,f)\in\Xi$ instead of just pairs $(e,f)$ in the definition of $\Xi$ is necessary to avoid the possibility that the same bivariant intertwiner $\Xi$ might have different spans as source and/or target.} 
\begin{gather}\label{eq: biv-int}
f \cdot_1 s_\phi(x)=s_\psi(y)\cdot_1 f, \quad e \cdot_2 t_\phi(x)=t_\psi(y)\cdot_2 e, \quad \forall x\in \phi, \ \forall y\in \psi.  
\end{gather}

Note that, given a collection $\Qs$ of spans of involutive monoids, the family of all quadruples 
$\xymatrix{f\rtwocell^x_y{\xi} & e}$, with $\xi:=(e,y,x,f)\in A_2\times \psi\times\phi\times A_1$, such that the following intertwining conditions hold: 
\begin{gather*}
f \cdot_1 s_\phi(x)=s_\psi(y)\cdot_1 f, \quad e \cdot_2 t_\phi(x)=t_\psi(y)\cdot_2 e,   
\end{gather*}
becomes a strict double category (cubical 2-category) $[\Qs]^2$ under the following compositions: 
\begin{gather*}
(e_1,y_1,x_1,f_1)\circ_1(e_2,y_2,x_2,f_2):=(e_1\cdot_2 e_2, y_1,x_2,f_1\cdot_1f_2), \quad \text{whenever $y_2=x_1$,} 
\\
(e_3,y_3,x_3,f_3)\circ_0(e_4,y_4,x_4,f_4):=(e_3, (y_3, y_4),(x_3,x_4),f_4), \quad \text{whenever $f_3=e_4$}, 
\end{gather*}
where $(y_3,y_4), (x_3,x_4)$ denote the concatenations of composable paths in the fine grained category $[\Qs]^1$. 
The category $[\Qs]^2$ is a strict fully involutive double category (see the manuscript~\cite{BCM} for definitions and a detailed treatment), with involutions given by: 
\begin{gather*}
(e,y,x,f)^{*_1}:=(e,x,y,f), \quad \quad (e,y,x,f)^{*_0}:=(f,y,x,e), \quad \forall (e,y,x,f)\in [\Qs]^2. 
\end{gather*}

Making the harmless identification between $\Xi$ and 
$\{(e,y,x,f)\in[\Qs]^2 \ | \ (e,\psi,\phi,f)\in \Xi, \ x\in \phi, \ y\in \psi\}$, the strict fully involutive 2-category $\Qs^{(2)}$ of bivariant intertwiners between globular spans in $\Qs$ is obtained by ``coarse graining'' the previous fully involutive double category $[\Qs]^2$ i.e.~considering, for all pairs $\phi,\psi\in\Qs$ in globular position, those subsets 
$\Xi\subset [\Qs]^2$ consisting of quadruples $(e,y,x,f)\in [\Qs]^2$, with $x\in\phi$ and $y\in \psi$, satisfying property~\eqref{eq: biv-int} and defining all the compositions and involutions ``elementwise'', whenever such compositions exist: 
\begin{gather*}
\Xi_1\circ_1\Xi_2:=\{\xi_1\circ_1\xi_2 \ | \ (\xi_1\in \Xi_1, \xi_2\in \Xi_2\}, 
\quad 
\Xi_3\circ_0\Xi_4:=\{\xi_3\circ_0\xi_4 \ | \ (\xi_3\in \Xi_3, \xi_4\in \Xi_4\}, 
\\
(\Xi)^{*_1}:=\{\xi^{*_1} \ | \ \xi\in \Xi\}, \quad  (\Xi)^{*_1}:=\{\xi^{*_1} \ | \ \xi\in \Xi\}. 
\end{gather*}

Of course such a notion of bivariant intertwiner between spans of involutive monoids admits a horizontal categorification in the context of example~\ref{ex: intertwiners2}. 

A \emph{relational 1-transfor} $\xymatrix{{\Cs_1}\rtwocell^{\phi}_{\psi}{\Xi}& {\Cs_2}}$ between a pair $\phi,\psi$ of spans of involutive 1-categories $\Cs_1,\Cs_2$ consists of a family $\Xi$ of quadruples $(e,\psi,\phi,f)$, with $(e,f)\in\Cs_2\times\Cs_1$ satisfying intertwining conditions of the form~\ref{eq: biv-int} in the respective 1-categories $(\Cs_1,\circ_1),(\Cs_2,\circ_2)$: 
\begin{gather*}
\forall x\in \phi \ \text{such that} \ {}_0(s_\phi(x))=f_0, \ {}_0(t_\phi(x))=e_0, \quad 
\forall y\in \psi \ \text{such that} \ (s_\psi(y))_0={}_0f, \ (t_\psi(y))_0={}_0e, 
\\ 
f \circ_1 s_\phi(x)=s_\psi(y)\circ_1 f, \quad e \circ_2 t_\phi(x)=t_\psi(y)\circ_2 e.  
\end{gather*}
The construction of a fully involutive 2-category of such relational 1-transfors follows the same lines indicated above for the one-object case. 

\medskip 

Relational 1-transfors are a quite vast generalization of natural transformations: every natural transformation provide a relational 1-transfor, but even if $\phi,\psi:\Cs_1\to\Cs_2$ are a pair of usual functors, a relational 1-transfor $\Xi:\phi\to\psi$ is not necessarily a natural transformation. 
Natural transformations are recovered if and only if, for all $A\in \Cs_1^0$ there exists one and only one $e_A\in\Cs_2^1$ such that $(e_A,\psi,\phi,\iota_A)\in \Xi$. For a general relational 1-transfor $\Xi$ between functors, it is neither assured that such an element $e_A$ exists, nor that it is unique. 

\medskip 

Recursively, in the case of strict globular $n$-categories, \emph{relational $n$-transfors} can be similarly defined and one can recover versions of theorem~\ref{th: iso-n-tr} for spans, without restricting to invertible $*$-functors. 
\xqed{\lrcorner}
\end{remark}

In the subsequent treatment (for simplicity) we will confine the discussion to the case of bijective spans, i.e.~$*$-isomorphisms of $*$-monoids, and further proceed to horizontal (natural transformations) and vertical categorification ($n$-transfors) in this particular case, see theorem~\ref{th: iso-n-tr}. 
In remark~\ref{rem: mor} we will further generalize the notion of intertwining via morphisms of bimodules (over involutive monoids) and we will mention how (at least for intertwiners between $*$-isomorphisms) the present 2-categories are embedded in categories of bimodules. 

The main motivation for the introduction of such generalized forms of intertwiners can be found in the attempt to prove an involutive analogue of theorem~\ref{th: n-tr} leading to a recursive construction of (fully) involutive higher categories via ``relational $n$-transfors'' (cf~\ref{prop: nlr} for a partially involutive case). 
For now, we examine the \emph{vertical categorification} of example~\ref{ex: intertwiners} restricting our considerations to the case of isomorphisms (further generalizations will be dealt with elsewhere). 

\begin{theorem}\label{th: iso-n-tr}
The family of small totally involutive $n$-categories with strict $n$-transfors, between invertible $*$-functors, constitutes a fully involutive $(n+1)$-category. 
\end{theorem}
\begin{proof}
We already know that taking (as objects) involutive 1-categories $(\Cs,\cdot,\dag),(\hat{\Cs},\hatcdot,\hat{\dag})$, with invertible $*$-functors 
$\Phi,\Psi:\Cs\to\hat{\Cs}$ (as 1-arrows), natural transformations (1-transfors) consist of intertwiners $\Xi:\Cs^0\to\hat{\Cs}^1$. In this way, the family of 1-transfors $\Cf^{(1)}$ constitutes a 2-category $(\Cf^{(1)},\circ_0,\circ_1,*_0,*_1)$ that is fully involutive with involutions given by: 
\begin{gather*}
\xymatrix{{\Cs} \rtwocell^{\Phi}_{\Psi}{\Xi}& {\hat{\Cs}}}\mapsto 
\xymatrix{{\Cs}\rtwocell^{\Phi}_{\Psi}{^\Xi^{*_1}\ }& {\hat{\Cs}}}, 
\quad \Xi^{*_1}(A):=\Xi(A)^{\hat{\dag}}, \quad \forall A\in \Cs^0, 
\\
\xymatrix{{\Cs}\rtwocell^{\Phi}_{\Psi}{\ \Xi}& {\hat{\Cs}}}\mapsto 
\xymatrix{\Cs & \ltwocell^{\Phi^{-1}}_{\Psi^{-1}}{^{\ \ \Xi^{*_0}}} {\hat{\Cs}}}, 
\quad \Xi^{*_0}(A):=\Phi^{-1}(\Xi(A)^{\hat{\dag}}), \quad \forall A\in \Cs^0. 
\end{gather*}
Suppose now, by induction, that we have a fully involutive $n$-category $(\Cf^{(n)},\circ_0,\dots,\circ_{n-1},*_0,\dots,*_{n-1})$ whose objects are small totally involutive \hbox{$n$-categories} $(\Cs,\cdot_0,\dots,\cdot_{n-1},\dag_0,\dots,\dag_{n-1})$, 
with 1-arrows the invertible $*$-functors (the $*$-isomorphisms of involutive $n$-categories) and $n$-arrows the strict $n$-transfors. 

Consider the globular $n$-cells $\xymatrix{{\Phi^{(0)}(A)} \rrtwocell^{\Phi^{(k)}(A)}_{\Psi^{(k)}(A)}{\quad \quad \Xi(A)}& & {\Psi^{(0)}(A)}}$ in $\hat{\Cs}$, for all $A\in \Cs^0$, where $\Xi:\Cs^0\to\hat{\Cs}^n$ is an $n$-transfor between $k$-transfors $\Phi^{(k)},\Psi^{(k)}$, for $k=0,\dots,n-1$, between invertible $*$-functors $\Phi^{(0)},\Psi^{(0)}$ from $\Cs$ to $\hat{\Cs}$. 
The $(n+1)$-cells in $\Cf^{(n)}$ are defined as $\xymatrix{\Cs \rrtwocell^{\Phi^{(k)}}_{\Psi^{(k)}}{\quad \Xi}& & \hat{\Cs}}$, for $k=0,\dots,n-1$. 

By theorem~\ref{th: n-tr}, we know that $\Cf^{(n)}$ with the $n$-transfors between (invertible) $*$-functors is already a strict globular $(n+1)$-category. Since, by induction we already have $n$-involutions of $n$-transfors $*_0,\dots,*_{n-1}$ in $\Cf^{(n)}$, to complete the proof, we only need to provide an involution over objects $*_0$, commuting with the previous involutions (and verify its covariance/contravariance properties). 

For this purpose, define $*_0: \Cf^{(n)}\to\Cf^{(n)}$ 
\begin{equation*}
*_0: \xymatrix{\Cs \rrtwocell^{\Phi^{(0)}}_{\Psi^{(0)}}{\quad \Xi}& & \hat{\Cs}}\mapsto 
\xymatrix{\Cs & & \lltwocell^{\Psi^{(0)}}_{\Phi^{(0)}}{^{\quad \Xi^{*_0}}} \hat{\Cs}},   
\end{equation*}
where $\Xi^{*_0}:\hat{\Cs}^0\to\Cs^n$ is given by  $\Xi^{*_0}(B):=(\Phi^{(0)})^{-1}(\Xi((\Psi^{(0)})^{-1}(B))^{\hat{\dag}_0})$, for all $B\in \hat{\Cs}^0$
(for $k=1,\dots,n-1$, $\Phi^{(k)}$ is the $k$-source and $\Psi^{(k)}$ is the $k$-target of $\Xi^{*_0}$). 
\end{proof}

\begin{example} 
Anticipating somehow the material developed later on in section~\ref{sec: hyper-conv}, whenever $(\As,\cdot,\dag)$ is a commutative $*$-monoid and $(\Xs,\circ_0,\dots,\circ_{n-1},*_0,\dots,*_{n-1})$ is a finite $n$-groupoid (or more generally a fully involutive $n$-category), the set $\Es:=\As\times\Xs$ becomes a fully involutive $n$-category with the following compositions and involutions, for $k=0,1$: 
\begin{gather*}
a_x \circ_k b_y:=(a\cdot b)_{x\circ_k y}, 
\quad 
(a_x)^{*_k}:=(a^\dag)_{x^{*_k}}, 
\end{gather*}
where we use the notation $a_x:=(a,x)\in\Es$. 
\xqed{\lrcorner}
\end{example} 

\subsection{Strict Involutive 2-Categories and Conjugations} \label{rem: conjugates} 

As promised in remarks~\ref{rem: conjugates 1} and~\ref{rem: conj}, we are going to discuss here in some more detail how fully involutive categories are related to the well-known notion of ``conjugation'' introduced in algebraic quantum field theory and constantly used in the theory of superselection sectors (see~\cite[section~III]{DHR}, \cite[section~3]{R2} and~\cite[section~2 and~7]{LR}). 
The study of fully involutive 2-C*-categories obtained in this way, will be completed later in 
example~\ref{rem: conjugates2}. 
Several of the properties of the conjugation maps (and of their associated involutions over objects) that are mentioned in this section, appeared already in~\cite{LR} and have also been used in previous works by C.Pinzari-J.E.Roberts~\cite{PR0}. 

\medskip 

For a more straightforward comparison with the formulas in the literature on superselection theory, we are using the ``reversed notation'' for the composition over objects in a 2-category $(\Cs,\otimes,\circ)$ and the usual notation for the composition over 1-arrows: $x\otimes y:=y\circ_0 x$ and $x\circ y:=x\circ_1 y$ for $x,y\in \Cs$. 

\medskip 

A generalized notion of right (left) conjugation for a pair $(x,y)$ of 1-arrows could actually be defined in the setting of (strict) 2-categories without involutions (or C*-structure) and consists in providing a specific adjunction  
$(-\otimes x) \dashv (-\otimes y)$ (respectively $(-\otimes x) \vdash (-\otimes y)$), between the partially defined functors $-\otimes x$ and $-\otimes y$, satisfying additional properties as described (for monoidal categories) in H.Lindner~\cite[proposition~5]{Lin}; this is further investigated in detail in the companion paper~\cite[section 3.2]{BCM}.

\medskip 

Here we will assume (as it is always the case in superselection theory) the existence of a strict involution $*$ over 1-arrows, hence $(\Phi\otimes \Psi)^*=\Phi^*\otimes \Psi^*$ and $(\Phi\circ\Psi)^*=\Psi^*\circ \Phi^*$ for all 
$\Phi,\Psi\in \Cs$, and adopt an ``intrinsic'' definition of conjugation as described below. 

In a strict 2-category (in particular in a strict monoidal category) equipped with a strict involution over 1-arrows $(\Cs,\otimes,\circ,*)$, a pair 
$x,\cj{x}\in\Cs^1$ of \hbox{1-arrows} $\xymatrix{ \rtwocell^{x}_{\cj{x}}{'} &}$, with $t(x)\xrightarrow{\cj{x}}s(x)$, are said to be \emph{conjugate} if there exist a pair of 2-arrows $R_x$, $\cj{R}_x$ that satisfy these 
\emph{conjugate equations}:   
\begin{gather*}
\xymatrix{\rtwocell^{\iota^1(t(x))}_{\cj{x}\otimes x}{\ R_x} &}, \
\xymatrix{\rtwocell^{\iota^1(s(x))}_{x\otimes \cj{x}}{\ \cj{R}_x}&}, \quad   
(\cj{R}^*_x\otimes \iota^2(x))\circ (\iota^2(x)\otimes R_x)=\iota^2(x), \quad  
(R^*_x\otimes \iota^2(\cj{x}))\circ (\iota^2(\cj{x})\otimes \cj{R}_x)=\iota^2(\cj{x}).
\end{gather*}

If, as we assume, the involution $*$ is Hermitian, the conjugation equations are equivalently rewritten as: 
\begin{gather*}
(\iota^2(x)\otimes R^*_x)
\circ 
(\cj{R}_x\otimes \iota^2(x))
=\iota^2(x), 
\quad  
(\iota^2(\cj{x})\otimes \cj{R}^*_x)
\circ 
(R_x\otimes \iota^2(\cj{x}))
=\iota^2(\cj{x}).
\end{gather*}
If $x,\cj{x}$ are conjugates, there are in general several pairs $(R_x,\cj{R}_x)$ of 2-arrows that satisfy the conjugate equations; on the other side, any pair $(R_x,\cj{R}_x)$ that satisfies the conjugate equations determines a unique pair $(x,\cj{x})$ of conjugate 1-arrows. 
The relation of conjugation is symmetric:\footnote{Notice that, for this statement, the existence of the involution $*$ over 1-arrows is crucial.}
if $(x,\cj{x})$ are conjugates, via $(R_x,\cj{R}_x)$, then also $(\cj{x},x)$ are conjugates via $(\cj{R}_x,R_x)$. 

\medskip 

If we assume all the 1-arrows in $\Cs$ to be conjugable (or alternatively we consider the full subcategory $\Cs_f$ of those 2-arrows in $\Cs$ with conjugable source and target), we can always choose a specific \emph{conjugation map} $x\mapsto (R_x,\cj{R}_x)$.  
Under this choice, we can define two \emph{folding maps} on 2-arrows $\Phi\in \Cs^2$:
\begin{align*}
&\xymatrix{A\rtwocell^x_y{\Phi}&B} \mapsto \xymatrix{B\rtwocell^{\cj{y}}_{\cj{x}}{\ \Phi_\bullet}&A}
&
\Phi_\bullet		&:=(\iota^2(\cj{x})\otimes \cj{R}^*_y)\circ (\iota^2(\cj{x})\otimes \Phi\otimes \iota^2(\cj{y}))\circ (R_x\otimes \iota^2(\cj{y})), 
\\
&\xymatrix{A\rtwocell^x_y{\Phi}&B} \mapsto \xymatrix{B\rtwocell^{\cj{y}}_{\cj{x}}{\ {}_\bullet\Phi}&A}
& 
{}_\bullet\Phi	&:=(R^*_y\otimes \iota^2(\cj{x}))\circ (\iota^2(\cj{y})\otimes \Phi\otimes \iota^2(\cj{x}))\circ (\iota^2(\cj{y})\otimes \cj{R}_x),  
\end{align*}
and hence two additional ``pseudo-involutions'' of 2-arrows over objects: 
$\xymatrix{B\rtwocell^{\cj{x}}_{\cj{y}}{\ \Phi^\dag}&A}$, 
$\xymatrix{B\rtwocell^{\cj{x}}_{\cj{y}}{\ \Phi^\ddag}&A}$, 
\begin{gather*}
\Phi^\dag:=(\Phi^*)_\bullet=
(\iota^2(\cj{y})\otimes\cj{R}_x^*)\circ (\iota^2(\cj{y})\otimes \Phi^*\otimes \iota^2(\cj{x}))\circ(R_y\otimes\iota^2(\cj{x})), 
\\
\Phi^\ddag:={}_\bullet(\Phi^*)=
(R^*_x\otimes\iota^2(\cj{y}))\circ(\iota^2(\cj{x})\otimes\Phi^*\otimes\iota^2(\cj{y}))\circ(\iota^2(\cj{x})\otimes \cj{R}_y). 
\end{gather*}
The map $\Phi\mapsto \Phi^\dag$ is the one actually considered by R.Longo-J.E.Roberts~\cite[lemma~2.3]{LR} and here we would like to further explore under which additional conditions $\dag$ (and similarly $\ddag$) can be taken as an involution over objects and hence provide a further example of fully involutive 
2-category $(\Cs,\otimes,\circ,*,\dag)$. 

\medskip 

The folding maps always satisfy the following $\circ$-contravariant properties: 
\begin{equation}\label{eq: circ-contra}
(\Phi\circ\Psi)_\bullet=\Psi_\bullet\circ\Phi_\bullet, \quad {}_\bullet(\Phi\circ\Psi)={}_\bullet\Psi\circ{}_\bullet\Phi, 
\quad 
\text{for} \quad  \xymatrix{A\ruppertwocell^x{\Psi}\rlowertwocell_z{\Phi} \ar[r]|y &B}.
\end{equation}
As a consequence: 
$(\Phi\circ\Psi)^\dag=((\Phi\circ\Psi)^*)_\bullet=
(\Psi^*\circ\Phi^*)_\bullet=(\Phi^*)_\bullet\circ(\Psi^*)_\bullet=\Phi^\dag\circ\Psi^\dag$ and, in a perfectly similar way, $(\Phi\circ\Psi)^\ddag=\Phi^\ddag\circ\Psi^\ddag$. Hence $\dag$ and $\ddag$ are both $\circ$-covariant maps.

Property~\ref{eq: circ-contra} is proved by usage of the exchange property\footnote{Notice that the full exchange property is not actually necessary: it is enough to require the validity of the exchange property whenever at least a pair of the 2-arrows involved belong to $\Cs^1$; such property clearly implies the non-commutative exchange, but it is a strictly stronger requirement.} 
and the conjugate equations:\footnote{With some abuse of notation, in all the subsequent calculations, we simply write $x:=\iota^2(x)$, for all $x\in \Cs^1$ and similarly $A:=\iota^2(A)$ for all $A\in \Cs^0$.}  
\begin{align*}
\Psi_\bullet 
\circ &\Phi_\bullet =
(\cj{x}\otimes \cj{R}^*_y)\circ(\cj{x}\otimes \Psi\otimes \cj{y})\circ(R_x\otimes \cj{y})
\circ
(\cj{y}\otimes\cj{R}^*_z)\circ (\cj{y}\otimes\Phi\otimes \cj{z})\circ(R_y\otimes \cj{z}) 
\\
&=
(\cj{x}\otimes \cj{R}^*_y)
\circ 
\Big([(\cj{x}\otimes \Psi)\circ R_x]\otimes \cj{y}\Big)
\circ
\Big(\cj{y}\otimes[\cj{R}^*_z\circ (\Phi\otimes \cj{z})]\Big)  
\circ 
(R_y\otimes \cj{z}) 
\\
&=
(\cj{x}\otimes \cj{R}^*_y)\circ \Big( [(\cj{x}\otimes \Psi)\circ R_x]\otimes \cj{y} \Big)
\circ
\Bigg(B \otimes \Big[\Big(\cj{y}\otimes [\cj{R}^*_z\circ (\Phi\otimes \cj{z})]\Big) \circ(R_y\otimes \cj{z}) \Big]\Bigg)
\\ 
&=
(\cj{x}\otimes \cj{R}^*_y)\circ \Bigg(
\Big([(\cj{x}\otimes \Psi)\circ R_x]\circ B\Big) \otimes \Big[\cj{y}\circ \Big( \cj{y}\otimes[\cj{R}^*_z\circ (\Phi\otimes \cj{z})]\Big) 
\circ (R_y\otimes \cj{z})
\Big] \Bigg) 
\\ 
&= 
(\cj{x}\otimes \cj{R}^*_y)\circ \Bigg(
[(\cj{x}\otimes \Psi)\circ R_x] \otimes \Big[\Big( \cj{y}\otimes[\cj{R}^*_z\circ (\Phi\otimes \cj{z})]\Big) \circ (R_y\otimes \cj{z})
\Big] \Bigg) 
\\
&=
(\cj{x}\otimes \cj{R}^*_y)\circ 
\Bigg( 
\Big( (\cj{x}\otimes y)\circ[(\cj{x}\otimes \Psi)\circ R_x]\Big) \otimes\Big( \Big[
( \cj{y}\otimes[\cj{R}^*_z\circ (\Phi\otimes \cj{z})]) \circ (R_y\otimes \cj{z})\Big]
\circ 
\cj{z}
\Big) \Bigg)  
\end{align*}

\begin{align*}
\phantom{\Psi_\bullet 
\circ} 
&= 
(\cj{x}\otimes \cj{R}^*_y)\circ \Big(
(\cj{x}\otimes y)\otimes
\Big[
(\cj{y}\otimes[\cj{R}^*_z\circ (\Phi\otimes \cj{z})]) \circ (R_y\otimes \cj{z})\Big]
\Big) \circ
( 
[(\cj{x}\otimes \Psi)\circ R_x]
\otimes 
\cj{z} 
) 
\\ 
&=
(\cj{x}\otimes \cj{R}^*_y\otimes A)\circ 
\Big(\cj{x}\otimes y\otimes \cj{y}\otimes[\cj{R}^*_z\circ (\Phi\otimes \cj{z})]\Big) \circ (\cj{x}\otimes y\otimes R_y\otimes \cj{z})
\circ
( 
[(\cj{x}\otimes \Psi)\circ R_x]
\otimes \cj{z} 
) 
\\ 
&=
\Bigg(
\Big((\cj{x}\otimes \cj{R}^*_y) \circ (\cj{x}\otimes y\otimes \cj{y})\Big) \otimes 
\Big(A \circ [\cj{R}^*_z\circ (\Phi\otimes \cj{z})]\Big) 
\Bigg)  
\circ
(\cj{x}\otimes y \otimes R_y\otimes \cj{z}) 
\circ
([(\cj{x}\otimes \Psi)\circ R_x]\otimes \cj{z})
\\ 
&=
\left(
(\cj{x}\otimes \cj{R}^*_y)\otimes 
[\cj{R}^*_z\circ (\Phi\otimes \cj{z})]
\right) 
\circ
(\cj{x}\otimes y \otimes R_y\otimes \cj{z}) 
\circ
([(\cj{x}\otimes \Psi)\circ R_x]\otimes \cj{z})
\\ 
&=
\Bigg( 
\Big(\cj{x}\circ (\cj{x}\otimes \cj{R}^*_y)\Big) \otimes 
\Big([\cj{R}^*_z\circ (\Phi\otimes \cj{z})]
\circ (y \otimes \cj{z})\Big)
\Bigg)  
\circ
(\cj{x}\otimes y \otimes R_y\otimes \cj{z}) 
\circ
([(\cj{x}\otimes \Psi)\circ R_x]\otimes \cj{z})
\\ 
&=
(\cj{x}\otimes [\cj{R}^*_z\circ (\Phi\otimes \cj{z})])
\circ
(\cj{x}\otimes \cj{R}^*_y\otimes y \otimes \cj{z})
\circ
(\cj{x}\otimes y \otimes R_y\otimes \cj{z}) 
\circ
([(\cj{x}\otimes \Psi)\circ R_x]\otimes \cj{z})
\\
\phantom{\Psi_\bullet \circ} 
&=
(\cj{x}\otimes [\cj{R}^*_z\circ (\Phi\otimes \cj{z})])
\circ
\Big(
\cj{x}\otimes
[(\cj{R}^*_y \otimes y) 
\circ
(y\otimes R_y)]  
\otimes \cj{z}
\Big)
\circ
([(\cj{x}\otimes \Psi)\circ R_x]\otimes \cj{z})
\\
&=
(\cj{x}\otimes [\cj{R}^*_z\circ (\Phi\otimes \cj{z})])
\circ
(\cj{x}\otimes y \otimes \cj{z})
\circ
([(\cj{x}\otimes \Psi)\circ R_x]\otimes \cj{z}) 
\\
&=
(\cj{x}\otimes \cj{R}^*_z)\circ (\cj{x}\otimes \Phi\otimes \cj{z})
\circ
(\cj{x}\otimes y \otimes \cj{z})
\circ
(\cj{x}\otimes \Psi\otimes \cj{z})\circ (R_x\otimes \cj{z}) 
\\
&=
(\cj{x}\otimes \cj{R}^*_z)\circ \Big(\cj{x}\otimes (\Phi \circ \Psi) \otimes \cj{z}\Big)\circ(R_x\otimes \cj{z})
=(\Phi \circ \Psi)_\bullet . 
\end{align*}

If $\Phi=\iota^2(A)\in \Cs^0$ is an object, since $\Phi^*=\Phi$, we get $\Phi^\dag=(\Phi^*)_\bullet=\Phi_\bullet$.  
If we assume now that the conjugation map $x\mapsto(R_x,\cj{R}_x)$ satisfies the additional \emph{unitality condition}  
\begin{equation}\label{eq: unitality} 
R_x=\iota^2(x)=\cj{R}_x, \quad \forall x\in \Cs^0, 
\end{equation} 
(such a choice is not restrictive and can always be done), in this case we necessarily have $\cj{x}=x$, furthermore 
$\Phi_\bullet=
(\iota^1(A)\otimes\iota^2(A)^*)\circ(\iota^1(A)\otimes\Phi\otimes\iota^1(A))\circ(\iota^2(A)\otimes\iota^1(A))=\Phi$ and hence $\Phi^\dag=\Phi$. Similarly, under condition~\eqref{eq: unitality}, $\Phi^\ddag=\Phi$, for $\Phi\in \Cs^0$, so that $\dag$ and $\ddag$ are covariant $(\Cs,\circ)$ endofunctors. 

\medskip 

The two foldings interchange under the $*$-involution i.e.~$(\Phi_\bullet)^*={}_\bullet(\Phi^*)$ and similarly $({}_\bullet\Phi)^*=(\Phi^*)_\bullet$: 
\begin{align*}
(\Phi_\bullet)^* &= 
[(\cj{x}\otimes \cj{R}^*_y)
\circ 
(\cj{x}\otimes \Phi\otimes \cj{y})
\circ 
(R_x \otimes \cj{y})]^*
=
(R_x \otimes \cj{y})^* 
\circ 
(\cj{x}\otimes \Phi\otimes \cj{y})^*
\circ  
(\cj{x}\otimes \cj{R}^*_y)^*
\\ 
&=
(R_x^* \otimes \cj{y}) 
\circ 
(\cj{x}\otimes \Phi^*\otimes \cj{y})
\circ  
(\cj{x}\otimes \cj{R}_y)={}_\bullet(\Phi^*). 
\end{align*} 
As a consequence, we immediately obtain that the maps $\dag$, $\ddag$ interchange under the involution $*$: 
\begin{gather*}
(\Phi^\dag)^*=((\Phi^*)_\bullet)^*={}_\bullet(\Phi^{**})=(\Phi^*)^\ddag, 
\quad \quad 
(\Phi^\ddag)^*=({}_\bullet(\Phi^*))^*=(\Phi^{**})_\bullet=(\Phi^*)^\dag. 
\end{gather*}

\medskip 

Without further requirements for conjugations, the maps $\dag$ and $\ddag$ are not usually involutive. 
For this purpose, let us assume that the conjugation map $x\mapsto (R_x,\cj{R}_x)$ satisfies the 
\emph{involutivity condition}:\footnote{Notice that the unitality~\eqref{eq: unitality} does not conflict with the involutivity~\eqref{eq: involutivity}, so that both conditions can be required together.} 
\begin{equation}\label{eq: involutivity} 
(R_x,\cj{R}_x)=(\cj{R}_{\cj{x}},R_{\cj{x}}), \quad \forall x\in \Cs^1. 
\end{equation}
Whenever a conjugation map satisfies~\eqref{eq: involutivity}, we have an induced involution $x\mapsto \cj{x}$ on 1-arrows $x\in \Cs^1$ and, when unitality~\eqref{eq: unitality} also holds, such involution is Hermitian on objects: $\cj{A}=A$, for $A\in \Cs^0$. 

\bigskip 

If the involutivity condition~\eqref{eq: involutivity} holds, the two previous folding maps are mutually inverses. 

Here below we show ${}_\bullet(\Phi_\bullet)=\Phi$, the proof of $({}_\bullet\Phi)_\bullet$ is similarly obtained in a ``specular way'':\footnote{Notice that, we have used the exchange properties involving at least a pair of arrows in $\Cs^1$, with the exception of the three passages leading to equations~\eqref{eq: 1}~\eqref{eq: 2}~\eqref{eq: 3}, where at least one object and at least one adjunction unit/counit was involved.} 

\begin{align}\notag 
{}_\bullet(\Phi_\bullet) &= 
(R^*_{\cj{x}}\otimes y)
\circ 
(x\otimes 
\Big[
(\cj{x}\otimes \cj{R}^*_y)
\circ 
(\cj{x}\otimes \Phi\otimes \cj{y})
\circ 
(R_x \otimes \cj{y}) 
\Big]
\otimes y)
\circ 
(x\otimes \cj{R}_{\cj{y}})
\\ \notag 
&= 
(\cj{R}^*_{x}\otimes y)
\circ 
(x\otimes 
\Big[
(\cj{x}\otimes \cj{R}^*_y)
\circ 
(\cj{x}\otimes \Phi\otimes \cj{y})
\circ 
(R_x \otimes \cj{y}) 
\Big]
\otimes y)
\circ 
(x\otimes R_{y})
\\ \notag 
&=
(\cj{R}^*_{x}\otimes y)
\circ 
(x\otimes \cj{x}\otimes \cj{R}^*_y\otimes y)
\circ 
(x\otimes \cj{x}\otimes \Phi\otimes \cj{y}\otimes y)
\circ 
(x\otimes R_x \otimes \cj{y}\otimes y) 
\circ 
(x\otimes R_{y})
\\ 
\notag 
&=
(\cj{R}^*_{x}\otimes y)
\circ 
(x\otimes \cj{x}\otimes \cj{R}^*_y\otimes y)
\circ 
\Bigg(
x\otimes 
\Big[ \Big(
[(\cj{x}\otimes \Phi)\circ R_x ]
\otimes (\cj{y}\otimes y)
\Big) 
\circ 
(B\otimes R_{y})
\Big] \Bigg)
\\ \notag 
&=
(\cj{R}^*_{x}\otimes y)
\circ 
(x\otimes \cj{x}\otimes \cj{R}^*_y\otimes y)
\circ 
\Bigg(
x\otimes 
\Big(
[(\cj{x}\otimes \Phi)\circ R_x ]
\circ B
\Big) 
\otimes 
((\cj{y}\otimes y)\circ R_{y})
\Bigg)
\\ \notag 
&=
(\cj{R}^*_{x}\otimes y)
\circ 
(x\otimes \cj{x}\otimes \cj{R}^*_y\otimes y)
\circ 
\Big(
x\otimes 
[(\cj{x}\otimes \Phi)\circ R_x ]
\otimes 
(R_{y}\circ B)
\Big) 
\\ \label{eq: 1}
&=
(\cj{R}^*_{x}\otimes y)
\circ 
(x\otimes \cj{x}\otimes \cj{R}^*_y\otimes y)
\circ 
\Big[
x\otimes 
\Big(
[(\cj{x}\otimes \Phi)\otimes R_y ]
\circ 
(R_{x}\otimes B)
\Big) 
\Big]  
\\ \notag 
&=
(\cj{R}^*_{x}\otimes y)
\circ 
\Big(
(x\otimes \cj{x})\otimes 
[(\cj{R}^*_y\otimes y)\circ (\Phi \otimes R_y)] 
\Big) 
\circ 
(x\otimes R_{x})
\\ \notag 
&=
\Big[
(\cj{R}^*_{x}\circ (x\otimes \cj{x}) )
\otimes 
\Big(
y \circ  
[(\cj{R}^*_y\otimes y)\circ (\Phi \otimes R_y)] 
\Big) \Big] 
\circ 
(x\otimes R_{x})
\\ 
\notag 
\phantom{{}_\bullet(\Phi_\bullet) } 
&=
\Big[
(A\circ \cj{R}^*_{x})  
\otimes 
\Big(
(\cj{R}^*_y\otimes y)\circ (\Phi \otimes R_y)  
\Big) 
\Big] 
\circ 
(x\otimes R_{x})
\\ \label{eq: 2}
&=
[ 
A \otimes 
(\cj{R}^*_y\otimes y)
]
\circ 
(\cj{R}^*_{x}\otimes \Phi \otimes R_y)  
\circ 
(x\otimes R_{x})
\\ \notag 
&= 
(\cj{R}^*_y\otimes y) 
\circ 
(\cj{R}^*_x\otimes \Phi\otimes R_y)
\circ 
(x\otimes R_x) 
\\ \notag 
&=  
(\cj{R}^*_y\otimes y) 
\circ 
\Big(\cj{R}^*_x\otimes [(y\circ\Phi) \otimes (R_y\circ B)]\Big)
\circ 
(x\otimes R_x) 
\\ \notag 
&=  
(\cj{R}^*_y\otimes y) 
\circ 
\Big(\cj{R}^*_x\otimes [(y\otimes R_y) \circ (\Phi\otimes B)]\Big)
\circ 
(x\otimes R_x) 
\\ \notag
&=  
(\cj{R}^*_y\otimes y) 
\circ 
\Big((A\circ \cj{R}^*_x)\otimes [(y\otimes R_y) \circ (\Phi\otimes B)]\Big)
\circ 
(x\otimes R_x) 
\\ \label{eq: 3}
&=  
(\cj{R}^*_y\otimes y) 
\circ 
(A\otimes (y\otimes R_y) )\circ [\cj{R}^*_x\otimes (\Phi\otimes B)] 
\circ 
(x\otimes R_x) 
\\ \notag 
&=  
(\cj{R}^*_y\otimes y) \circ (y\otimes R_y) \circ 
[(A\circ \cj{R}^*_x)\otimes ((\Phi\otimes B)\circ x)] 
\circ (x\otimes R_x) 
\\ \notag 
&= 
(\cj{R}^*_y\otimes y) 
\circ 
(y\otimes R_y)
\circ 
(A\otimes \Phi\otimes B)
\circ 
(\cj{R}^*_x\otimes x)
\circ 
(x\otimes R_x) 
\\ \notag 
&=
y\circ \Phi\circ x
=\Phi. 
\end{align}

The maps $\dag$, $\ddag$ are not necessarily involutive since, in general, ${}_\bullet(\Phi_\bullet)\neq\Phi$ and $({}_\bullet\Phi)_\bullet\neq\Phi$; but, when the conjugation map $x\mapsto(R_x,\cj{R}_x)$ satisfies the involutivity condition~\eqref{eq: involutivity}, $\dag$ and $\ddag$ are indeed  involutions: 
\begin{align*}
&(\Phi^\dag)^\dag=((\Phi^*)_\bullet)^\dag=(((\Phi^*)_\bullet)^*)_\bullet= ({}_\bullet(\Phi^{**}))_\bullet=({}_\bullet\Phi)_\bullet=\Phi, 
\\ 
&(\Phi^\ddag)^\ddag=({}_\bullet(\Phi^*))^\dag={}_\bullet(({}_\bullet(\Phi^*))^*)= {}_\bullet((\Phi^{**})_\bullet)={}_\bullet(\Phi_\bullet)=\Phi. 
\end{align*}

\medskip 

In general the maps $\dag$, $\ddag$ do not necessarily coincide: 
\begin{equation}\label{eq: tracial} 
\Phi^\dag=\Phi^\ddag \quad \iff \quad (\Phi^*)^\dag=(\Phi^\dag)^* \quad \iff \quad (\Phi^*)^\ddag=(\Phi^\ddag)^* \quad 
\iff \quad \Phi_\bullet={}_\bullet\Phi, \quad \forall \Phi\in \Cs. 
\end{equation}
This follows immediately from $(\Phi^*)^\dag=((\Phi^*)^*)_\bullet=\Phi_\bullet$ and from 
$(\Phi^\dag)^*=((\Phi^*)_\bullet)^*={}_\bullet((\Phi^*)^*)={}_\bullet\Phi$, (even in absence of unitality and involutivity conditions for the conjugation map). When the involutivity property~\eqref{eq: involutivity} is assumed, conditions~\eqref{eq: tracial} are further equivalent to the involutivity of the folding maps: $(\Phi_\bullet)_\bullet=\Phi$ and ${}_\bullet({}_\bullet\Phi)=\Phi$.  

The validity of any of the equivalent properties~\ref{eq: tracial} is (in the monoidal category case) implied by the ``traciability condition'' described by R.Longo-J.E.Roberts~\cite[lemma~2.3 c]{LR}, but is in general false. 

\medskip 

In general the equations $(\Phi\otimes\Psi)^\dag=\Psi^\dag\otimes\Phi^\dag$ and  $(\Phi\otimes\Psi)^\ddag=\Psi^\ddag\otimes\Phi^\ddag$ do not hold.
 
In those specific cases where it is possible to globally select a conjugation map $x\mapsto (R_x,\cj{R}_x)$ that satisfies the \emph{tensorial conditions}~\cite[proof of theorem~2.4]{LR}\footnote{In~\cite{BCM}, a strict 2-category with such property (and unitality) is said to be equipped with an ``internal adjunction structure''.} 
\begin{equation}\label{eq: tensorial} 
R_{x\otimes y}=(\iota^2(\cj{y})\otimes R_x\otimes \iota^2(y))\circ R_y, \quad 
\cj{R}_{x\otimes y}=(\iota^2(x)\otimes \cj{R}_y\otimes\iota^2(\cj{x}))\circ \cj{R}_x,  
\end{equation}
we observe (see~\cite[theorem~2.4]{LR} for the monoidal case) that conjugation becomes a $\otimes$-congruence relation on 1-arrows,  i.e.~whenever $(x,\cj{x})$ and $(y,\cj{y})$ are conjugate pairs, via $(R_x,\cj{R}_x)$ and $(R_y,\cj{R}_y)$, if $x\otimes y$ exists (so also 
$\cj{y}\otimes\cj{x}$ exists), $\cj{x\otimes y}=\cj{y}\otimes\cj{x}$, because $(x\otimes y,\cj{y}\otimes\cj{x})$ is a conjugate pair via 
$(R_{x\otimes y},\cj{R}_{x\otimes y})$: 
\begin{align*}
&
(R^*_{x\otimes y}\otimes x\otimes y)\circ (x\otimes y \otimes R_{x\otimes y}) 
=x\otimes y, 
\\ 
&
(R^*_{x\otimes y}\otimes\cj{y}\otimes \cj{x}) \circ (\cj{y}\otimes \cj{x}\otimes \cj{R}_{x\otimes y})
= 
\Big([R_y^*\circ (\cj{y}\otimes R_x^*\otimes y)]\otimes \cj{y}\otimes\cj{x}\Big) 
\circ 
\Big(
\cj{y}\otimes\cj{x}\otimes [(x\otimes \cj{R}_y\otimes \cj{x})\circ \cj{R}_x]
\Big)
\\
&=
(R^*_y\otimes \cj{y}\otimes \cj{x})\circ (\cj{y}\otimes B\otimes y\otimes \cj{y}\otimes\cj{x})
\circ (\cj{y}\otimes R^*_x\otimes \cj{R}_y\otimes \cj{x})\circ (\cj{y}\otimes \cj{x}\otimes x\otimes B\otimes \cj{x})
\circ (\cj{y}\otimes \cj{x}\otimes \cj{R}_x )
\\ 
&=
(R^*_y\otimes \cj{y}\otimes \cj{x})\circ  
(\cj{y}\otimes B\otimes\cj{R}_y\otimes \cj{x})
\circ 
(\cj{y}\otimes R^*_x\otimes  B\otimes \cj{x})
\circ (\cj{y}\otimes \cj{x}\otimes \cj{R}_x )
= (\cj{y}\otimes \cj{x})\circ  (\cj{y}\otimes \cj{x})= \cj{y}\otimes \cj{x} .
\end{align*}

Under the same tensorial conditions~\eqref{eq: tensorial}, we obtain also the $\otimes$-contravariance of the folding maps 
\begin{equation*}
(\Phi\otimes\Psi)_\bullet = \Psi_\bullet\otimes\Phi_\bullet, \quad 
{}_\bullet(\Phi\otimes\Psi) = {}_\bullet\Psi \otimes {}_\bullet\Phi, \quad 
\text{for} \quad \xymatrix{
A \rtwocell^{x_1}_{y_1}{\Phi} & B \rtwocell^{x_2}_{y_2}{\Psi} & C, 
}
\end{equation*}
via the following computation, using again the exchange property\footnote{Again, also here we only use the exchange property whenever at least two of the four 2-arrows involved are in $\Cs^1$.} and the conjugate equations:
\begin{align*}
\Psi_\bullet \otimes &\Phi_\bullet 
=(\cj{x}_2 \circ \Psi_\bullet)\otimes (\Phi_\bullet \circ \cj{y}_1) 
=(\cj{x}_2 \otimes \Phi_\bullet)\circ (\Psi_\bullet \otimes \cj{y}_1) 
\\
&
\begin{aligned}
=\Big(\cj{x}_2 \otimes [(\cj{x}_1\otimes \cj{R}^*_{y_1}) 
&
\circ(\cj{x}_1\otimes \Phi \otimes \cj{y}_1)\circ(R_{x_1}\otimes \cj{y}_1)] \Big)
\\ 
& \circ  
\Big([(\cj{x}_2\otimes \cj{R}^*_{y_2})\circ(\cj{x}_2\otimes \Psi \otimes \cj{y}_2)\circ(R_{x_2}\otimes \cj{y}_2)] \otimes \cj{y}_1\Big)  
\end{aligned}
\\
&
\begin{aligned}
=(\cj{x}_2\otimes\cj{x}_1\otimes\cj{R}^*_{y_1})\circ 
\Big([\cj{x}_2\otimes((\cj{x}_1\otimes\Phi)\circ R_{x_1})] \otimes \cj{y}_1\Big) 
& \circ 
\Big(\cj{x}_2 \otimes [(\cj{R}^*_{y_2}\circ(\Psi\otimes \cj{y}_2)\otimes\cj{y}_1] \Big)
\\ 
& \circ (R_{x_2}\otimes \cj{y}_2\otimes\cj{y}_1) 
\end{aligned}
\\ 
&
\begin{aligned}
= (\cj{x}_2\otimes\cj{x}_1\otimes\cj{R}^*_{y_1}) & \circ 
\Big[\Big([\cj{x}_2\otimes((\cj{x}_1\otimes\Phi)\circ R_{x_1})]\circ\cj{x}_2\Big)
\otimes 
\Big(\cj{y}_1\circ[(\cj{R}^*_{y_2}\circ(\Psi\otimes \cj{y}_2)\otimes\cj{y}_1]\Big)\Big] 
\\
& \circ (R_{x_2}\otimes \cj{y}_2\otimes\cj{y}_1)  
\end{aligned} 
\\
&
\begin{aligned}
= (\cj{x}_2\otimes\cj{x}_1\otimes\cj{R}^*_{y_1}) & \circ 
\Big[[\cj{x}_2\otimes((\cj{x}_1\otimes\Phi)\circ R_{x_1})]
\otimes 
[(\cj{R}^*_{y_2}\circ(\Psi\otimes \cj{y}_2)\otimes\cj{y}_1]\Big] 
\\ 
& \circ (R_{x_2}\otimes \cj{y}_2\otimes\cj{y}_1)  
\end{aligned}
\\
&
\begin{aligned}
= (\cj{x}_2\otimes\cj{x}_1\otimes\cj{R}^*_{y_1}) \circ 
\Big[
& 
\Big([\cj{x}_2\otimes\cj{x}_1\otimes y_1] \circ [\cj{x}_2 \otimes((\cj{x}_1\otimes\Phi)\circ R_{x_1})]\Big)
\\ 
& 
\otimes  
\Big([\cj{R}^*_{y_2}\circ(\Psi\otimes \cj{y}_2)\otimes\cj{y}_1] \circ [x_2\otimes \cj{y}_2\otimes\cj{y}_1]\Big)
\Big]
\circ (R_{x_2}\otimes \cj{y}_2\otimes\cj{y}_1) 
\end{aligned}
\\ 
& \begin{aligned}
= (\cj{x}_2\otimes\cj{x}_1\otimes\cj{R}^*_{y_1}) \circ 
\Big[&\Big([\cj{x}_2\otimes\cj{x}_1\otimes y_1] \otimes [\cj{R}^*_{y_2}\circ(\Psi\otimes \cj{y}_2)\otimes\cj{y}_1] \Big)
\\ 
& \circ \Big([\cj{x}_2\otimes((\cj{x}_1\otimes\Phi)\circ R_{x_1})]
\otimes [x_2\otimes \cj{y}_2\otimes\cj{y}_1]\Big)\Big] 
\circ (R_{x_2}\otimes \cj{y}_2\otimes\cj{y}_1) 
\end{aligned}
\\ 
& \begin{aligned}
=(\cj{x}_2\otimes \cj{x}_1 & \otimes \cj{R}^*_{y_1}) 
\circ 
(\cj{x}_2\otimes \cj{x}_1\otimes y_1 \otimes \cj{R}^*_{y_2} \otimes \cj{y}_1) 
\circ 
(\cj{x}_2\otimes\cj{x}_1\otimes y_1\otimes \Psi\otimes \cj{y}_2 \otimes \cj{y}_1)
\\ 
& \circ 
(\cj{x}_2\otimes\cj{x}_1\otimes\Phi \otimes x_2\otimes \cj{y}_2\otimes\cj{y}_1)
\circ 
(\cj{x}_2 \otimes R_{x_1}\otimes x_2\otimes \cj{y}_2\otimes \cj{y}_1)
\circ 
(R_{x_2}\otimes \cj{y}_2\otimes \cj{y}_1) 
\end{aligned}
\\ 
& \begin{aligned}
=(\cj{x}_2\otimes \cj{x}_1\otimes \cj{R}^*_{y_1}) 
\circ 
(\cj{x}_2\otimes \cj{x}_1\otimes y_1 \otimes \cj{R}^*_{y_2} \otimes \cj{y}_1) 
& \circ 
\Big(\cj{x}_2\otimes\cj{x}_1\otimes 
(
[y_1\otimes \Psi]
\circ 
[\Phi \otimes x_2]
)
\otimes \cj{y}_2\otimes\cj{y}_1\Big) 
\\ 
&
\circ 
(\cj{x}_2 \otimes R_{x_1}\otimes x_2\otimes \cj{y}_2\otimes \cj{y}_1) \circ (R_{x_2}\otimes \cj{y}_2\otimes \cj{y}_1)
\end{aligned}
\end{align*}
\begin{align*} 
& \begin{aligned}
=(\cj{x}_2\otimes \cj{x}_1\otimes \cj{R}^*_{y_1}) 
\circ 
(\cj{x}_2\otimes \cj{x}_1\otimes y_1 \otimes \cj{R}^*_{y_2} \otimes \cj{y}_1) 
& \circ (\cj{x}_2\otimes \cj{x}_1\otimes \Phi\otimes\Psi\otimes \cj{y}_2\otimes \cj{y}_1) 
\\ 
& \circ 
(\cj{x}_2 \otimes R_{x_1}\otimes x_2\otimes \cj{y}_2\otimes \cj{y}_1) \circ (R_{x_2}\otimes \cj{y}_2\otimes \cj{y}_1)
\end{aligned}
\\
& \begin{aligned}
=\Big(\cj{x}_2\otimes \cj{x}_1\otimes [\cj{R}^*_{y_1} \circ (y_1 \otimes \cj{R}^*_{y_2} \otimes \cj{y}_1)]\Big) 
& \circ (\cj{x}_2\otimes \cj{x}_1\otimes (\Phi\otimes\Psi)\otimes \cj{y}_2\otimes \cj{y}_1) 
\\ 
& \circ 
\Big([(\cj{x}_2 \otimes R_{x_1}\otimes x_2)\circ R_{x_2}] \otimes \cj{y}_2\otimes \cj{y}_1\Big) 
\end{aligned}
\\
&= (\cj{x_1\otimes x_2}\otimes \cj{R}^*_{y_1\otimes y_2})\circ (\cj{x_1\otimes x_2}\otimes (\Phi\otimes\Psi)\otimes \cj{y_1\otimes y_2})\circ 
(R_{x_1\otimes x_2}\otimes \cj{y_1\otimes y_2}) 
= (\Phi\otimes\Psi)_\bullet 
\end{align*}
The $\otimes$-contravariance for the second folding map is perfectly specular. 

As an immediate corollary we obtain:
$(\Phi\otimes\Psi)^\dag=\big((\Phi\otimes\Psi)^*\big)_\bullet=(\Phi^*\otimes\Psi^*)_\bullet=(\Psi^*)_\bullet\otimes(\Phi^*)_\bullet
=\Psi^\dag\otimes\Phi^\dag$ 
and, in the same way, also $(\Phi\otimes\Psi)^\ddag=\Psi^\ddag\otimes\Phi^\ddag$. 

\medskip 

As a consequence of the previous discussion, we see that a strict 2-category $(\Cs,\otimes, \circ,*)$ (with usual exchange), equipped with an involution over 1-arrows and a unital involutive tensorial conjugation map $x\mapsto (R_x,\cj{R}_x)$ that satisfies the traciability condition, is canonically endowed with an involution over objects $\dag$ and with such an involution $(\Cs,\otimes,\circ,*,\dag)$ is an example of fully involutive strict 2-category. 

\section{Strict Higher C*-categories} \label{sec: hC*} 

In this section we introduce home-set-wise additive and Banach structures on (fully involutive) strict globular $n$-categories and we thereby arrive at the main definitions of higher C*-categories (with and without non-commutative exchange), producing some basic examples. 

\subsection{Strict Higher Algebroids and Categorical Bundles} 

We now proceed to describe a vertical categorification of the $*$-algebroids that, for $n=1$, has been defined by 
P.Ghez-R.Lima-J.E.Roberts~\cite{GLR} and P.Mitchener~\cite{M}. 

\begin{definition}
An \emph{$n$-algebroid at level-$p$}, for $0\leq p<n$, is a strict $n$-category $(\Cs,\circ_0,\dots,\circ_{n-1})$ further equipped with a partial linear structure $(\Cs,+,\cdot)$ such that: 
\begin{itemize}
\item 
the $p$-home-sets $\Cs_{xy}:=\{w\in \Cs \ | \ \exists \ w\circ_p y, x\circ_p w\}$, $\forall x,y\in \Cs^p$, are disjoint union of linear spaces,\footnote{We choose here to introduce a (partial) linear structure only on the set of $n$-arrows (and similarly we will later introduce a norm only on $n$-arrows); in essentially all of the examples treated here, the spaces $\Cs_{xy}$, for $x,y\in\Cs^{n-1}$, will simply be vector spaces, but also in such case, whenever $x,y\in\Cs^{k}$, with $0\leq k<n-1$, the spaces $\Cs_{xy}$ are not vector spaces: they are union of the vector spaces $\Cs_{ab}$, for all $a,b\in\Cs^{n-1}$ with $p$-source $y$ and $p$-target $x$; the appearance of this ``superselection'' structure on lower level home-sets is inevitable.}  
\item 
the composition $\circ_p$ is bilinear when restricted to composable linear spaces.
\end{itemize}
Whenever an involution $*_\alpha$, for $\alpha\subset\{0,\dots,n\}$, is present on a given $n$-algebroid $(\Cs,\circ_0,\dots,\circ_{n-1},+,\cdot)$, we can require the involution to be linear or conjugate-linear when restricted to the linear spaces. In this case we say that we have a \emph{$*_\alpha$-$n$-algebroid}. In the case of fully involutive $n$-categories, we will simply use the term $*$-$n$-algebroid. 
\end{definition}

\begin{remark}
In  general the home-sets $\Cs_{xy}$, for $x,y\in \Cs^p$ might even be abelian groupoids $(\Cs_{xy},+)$. 

Here, for simplicity, we will usually assume that the home-sets $(\Cs_{xy},+,\cdot)$ are complex vector spaces. 

Furthermore, we will later restrict mostly to the case of depth-$(n-1)$, i.e.~$p=n-1$, so that no other independent linear structures are imposed for $0\leq p<n-1$. 
\xqed{\lrcorner}
\end{remark}

Every 1-category $(\Cs,\circ)$ can be seen as a bundle over the discrete pair groupoid $\Xs:=\Cs^0\times\Cs^0$, with projection functor $\pi:\Cs\to\Xs$ given by $\pi(x):=(t(x),s(x))$ and with fibers $\Cs_{AB}$, for all $(A,B)\in \Xs$. This process of ``bundlification'', trading categories for bundles and further generalizing to cases when the base is not simply a discrete pair groupoid, is quite useful and admits a vertical categorification.

\medskip 

We introduce here, in a slightly more general form than needed, those notions of bundles that are compatible with functorial projections and with the involutions and fiberwise linear structures present in algebroids. We will eventually make use of such notions for the specific case of vertical categorification of Fell bundles that will be treated in the subsequent section~\ref{sec: hfb}. 

\begin{definition}
A \emph{categorical $n$-bundle} $(\Es,\pi,\Xs)$ is a continuous open surjective $n$-functor $\pi:\Es\to\Xs$ between topological $n$-categories such that $\Es$ is equipped with a ``fiberwise uniform structure'' i.e.~a family of sets $U\subset \Es\times \Es$, with $U_e:=U\cap \Es_{\pi(e)}\subset\Es_{\pi(e)}$, for all $e\in\Es$, such that the sets $\bigcup_{x\in\O}U_{\sigma(x)}$, with $\O$ an open set in $\Xs$ and $\sigma:\Xs\to\Es$ a continuous section of 
$\Es$, form a base of neighborhoods of the topology of $\Es$.\footnote{In most cases of practical interest $\Xs$ will be locally compact.
\\ 
Completeness of the fibers can be imposed via the induced uniformity.}
\end{definition}

In a perfectly similar way we have bundlifications of the definitions of strict globular (involutive) \hbox{$n$-algebroid} at level-$p$. 

\begin{definition}
A \emph{level-$p$ algebroidal $n$-bundle} is a categorical $n$-bundle $\pi:\Es\to\Xs$ such that $(\Es,\circ_0,\dots,\circ_{n-1},+)$ is an $n$-algebroid at level-$p$ and such that $\circ_p$ is bilinear when restricted to composable home-sets.\footnote{Such bilinearity is equivalent to the $\circ_p$-bifunctoriality with respect to the additive structures.} 
Algebroidal bundles, with the fiberwise uniformities induced via a choice of norms on the fibers, are said to be \emph{normed} (and Banach when fiberwise complete). 
\end{definition}

Similarly, whenever $\Xs$ and $\Es$ are (partially/totally) involutive topological categories, we define \emph{involutive \hbox{$n$-algebroidal} bundles} as $n$-algebroidal bundles such that $\pi:\Es\to\Xs$ is a \hbox{$*$-functor} (for all the relevant involutions) and requiring that the involutions are linear or conjugate-linear when restricted to the linear spaces in every fiber. 

\medskip 

The description of a 1-category $\Cs$ as a bundle over the pair groupoid $\Cs^0\times\Cs^0$ admits a vertical categorification.

\begin{remark}\label{rem: cat-b}
Every $n$-category $(\Cs,\circ_0,\dots,\circ_{n-1})$ (with usual exchange or non-commutative exchange) can be equivalently described as an $n$-categorical bundle $\Es\xrightarrow{(t^{n-1},s^{n-1})}\Xs$ over the (discrete) $n$-category 
$\Xs:=(\Cs^{n-1}\times\Cs^{n-1},\cj{\circ}_0,\dots,\cj{\circ}_{n-1},\cj{\circ}_n)$ with operations defined as pairwise compositions, if $0\leq p<n$, and as ``concatenations'', if $p=n$:\footnote{The bundle defined here can have empty fibers (whenever the pair $(x,y)\in \Cs^{n-1}\times\Cs^{n-1}$ is not in globular position and there is no problem in restricting the base of the bundle to such ``globular'' pairs of $(n-1)$-arrows in $\Cs$.)} 
\begin{gather*}
(x_1,y_1)\cj{\circ}_p(x_2,y_2):=(x_1\circ_p x_2,y_1\circ_p y_2), \quad 0\leq p<n, \quad (x_1,x_2),(y_1,y_2)\in \Cs^{n-1}\times_{\Cs^p}\Cs^{n-1}
\\
(x,y)\cj{\circ}_n(y,z):=(x,z), \quad x,y,z\in \Cs^{n-1}. 
\xqedhere{8.3cm}{\lrcorner}
\end{gather*}
\end{remark}
 
\subsection{Higher C*-categories and Higher Fell Bundles} \label{sec: hfb}

After this quite long preparation on involutions and linear structures on strict globular $n$-categories, we are now ready to deal with the main subject of our investigation. We start with the following vertical categorification of the Longo-Roberts 2-C*-categories presented in definition~\ref{def: lr}. 
Here, in perfect continuity with the original definition~\ref{def: lr} (for $n=2$) we have $n$-categories equipped with only one top-level involution $*_{n-1}$; and hence C*-properties for them can be imposed only for the top-level composition/involution pair  $(\circ_{n-1},*_{n-1})$. Abundant classes of examples are naturally fitting such definition as exposed in the subsequent proposition~\ref{prop: nlr}. ``Higher conjugations'', in line with the treatment in~\ref{rem: conjugates} can be introduced, but will not be considered here. 
\begin{definition}
A \emph{strict $n$-C*-category of Longo-Roberts type} $(\Cs,\circ_{n-1},\dots,\circ_0,*_{n-1},+,\cdot, \|\cdot \|)$ is a strict globular $n$-category $(\Cs,\circ_{n-1},\dots,\circ_0)$ that satisfies the following additional properties:
\begin{itemize}
\item
$(\Cs,\circ_{n-1},*_{n-1},+,\cdot, \|\cdot \|)$ is a C*-category, with involution $*_{n-1}$, 
\item 
$*_{n-1}$ is a covariant functor on $(\Cs,\circ_k)$, for all $0\leq k<n-1$, 
\item
all the partial bifunctors $\circ_{k}$, for $0\leq k<n-1$, when restricted to $\circ_k$-composable $k$-home-sets, are bilinear and norm submultiplicative.
\end{itemize} 
A \emph{$*$-functor} $(\Cs,\circ_{n-1},\dots,\circ_0,*_{n-1})\xrightarrow{\phi}(\hat{\Cs},\circ_{n-1},\dots,\circ_0,*_{n-1})$ between 
$n$-C*-categories of Longo-Roberts type is just a functor between the underlying $n$-categories such that 
$\phi(x^{*_{n-1}})=\phi(x)^{\hat{*}_{n-1}}$, for all $x\in \Cs$. 

\medskip 

A \emph{natural transformation} between $*$-functors (and inductively a \emph{$k$-transfor $\Phi:\Cs^{0}\to\Cs^k$} between $(k-1)$-transfors), of $n$-C*-categories of Longo-Roberts type, is always assumed to be a natural transformation (respectively a $k$-transfor) that is \emph{bounded} i.e.~$\sup_{x\in\Cs^{n-k}}\|\Phi(x)\|<\infty$. 
\end{definition}

\begin{remark}
For $n=1$, since the second and third properties are ``vacuous'' (there are no compositions $\circ_k$ with $k<n-1=0$), the previous definition reproduces C*-categories and for $n=2$ we reobtain Longo-Roberts 2-C*-categories in 
definition~\ref{def: lr}. 
Notice that $(\Cs,\circ_0,\dots,\circ_{n-1},+,\cdot)$ is an $n$-algebroid at level-$p$, for all $p=0,\dots,n-1$, that is Banach with respect to the unique norm $\|\cdot\|$, and the category $\Cs$ is only partially involutive (with only one involution $*_{n-1}$). 

In the previous definition we have for now assumed that the strict globular $n$-category $(\Cs,\circ_0,\dots,\circ_{n-1})$ satisfies the usual exchange property, so the partially defined compositions $\circ_k$, $0\leq k<n$ are bifunctors on $(\Cs,\circ_q)$, for all $q\neq k$. 
This requirement can be relaxed as can be seen in the more general definition~\ref{def: nC*}.  

As natural transformations are $1$-transfors, their boundedness requirement is $\sup_{x\in \Cs^{n-1}}\|\Phi(x)\|<\infty$. 
\xqed{\lrcorner}
\end{remark}

Examples of $n$-C*-categories of Longo-Roberts type essentially reduce to a 1-C*-category living ``on the top'' of a commutative $(n-1)$-C*-category with only one involution; anyway, the following examples are interesting and naturally occur in the theory. 

\begin{example}\label{ex: lr} 
The family of (bounded) natural transformations of $*$-functors of small \hbox{1-C*-categories} is a Longo-Roberts \hbox{2-C*-category}, 
cf.~\cite[section~7]{LR}. 
Let $\Cs^{(2)}$ denote the family of such natural transformations.  
We construct a 2-C*-category $(\Cs^{(2)},\circ_0,\circ_1,*_1,+,\cdot,\| \cdot \|)$ as follows: 
\begin{itemize}
\item
the vertical composition $\xymatrix{\Cs  \ruppertwocell^\phi{\Psi} \rlowertwocell_\xi{\Phi} \ar[r]|{\psi} & \hat{\Cs}}$ of bounded natural transformations $\Psi,\Phi\in\Cs^{(2)}$ between the $*$-functors $\Cs\xrightarrow{\phi,\psi,\xi}\hat{\Cs}$ of the C*-categories $(\Cs,\circ, *,+,\cdot,\|\ \|_\Cs)$, and 
$(\hat{\Cs},\hatcirc,\hat{*},\hatplus,\hatcdot,\|\cdot \|_{\hat{\Cs}})$, is the natural transformation from 
$\phi$ to $\xi$ defined, for all $o\in \As^0$, by $(\Phi\circ_1\Psi)_o:=(\Phi_o)\hatcirc (\Psi_o)$; the boundedness follows from $\|(\Phi\circ_1\Psi)_o\|_{\hat{\Cs}}\leq\|\Phi_o\|_{\hat{\Cs}}\cdot\|\Psi_o\|_{\hat{\Cs}}$. 
\item 
the horizontal composition 
$\xymatrix{\Cs \rtwocell^{\phi_1}_{\psi_1}{\Theta}& \tilde{\Cs} \rtwocell^{\phi_2}_{\psi_2}{\Omega}& \hat{\Cs}}$ of bounded natural transformations $\Theta,\Omega\in \Cs^{(2)}$ is defined, for all $o\in \Cs^0$, by  
$\psi_2(\Theta_o)\hatcirc \Omega_{\phi_1(o)}=\Omega_{\psi_1(o)}\hatcirc \phi_2(\Theta_o)$. This expression yields indeed a natural transformation between the $*$-functors $\phi_2\circ\phi_1$ and $\psi_2\circ\psi_1$ whose boundedness readily follows from $\|\psi_2(\Theta_o)\circ \Omega_{\phi_1(o)}\|_{\hat{\Cs}}\leq 
\|\psi_2(\Theta_o)\|_{\hat{\Cs}}\cdot\| \Omega_{\phi_1(o)}\|_{\hat{\Cs}}\leq \|\Theta_o\|_{\widetilde{\Cs}} \cdot 
\| \Omega_{\phi_1(o)}\|_{\hat{\Cs}}$, since every $*$-functor between C*-categories is automatically bounded~\cite{GLR}.
\item 
Given a natural transformation $\xymatrix{\Cs\rtwocell^{\phi}_{\psi}{\Phi}& \hat{\Cs}}$ in $\Cs^{(2)}$, the map 
$\Phi^{*_1}: o\mapsto(\Phi_o)^{\hat{*}}$, for $o\in \Cs^0$, defines a natural transformation 
$\xymatrix{\Cs\rtwocell^{\phi}_{\psi}{^\Phi^{*_1}\ }& \hat{\Cs}}$ in $\Cs^{(2)}$: 
for all $x\in \Cs$, since $\Phi_{t(x)}\hatcirc\phi(x)=\psi(x)\hatcirc\Phi_{s(x)}$, we have  
$\Phi_{t(x)}^{\hat{*}}\hatcirc \psi(x)=(\psi(x^*))\hatcirc\Phi_{t(x)}^{\hat{*}}= 
(\Phi_{s(x)}\hatcirc\phi(x^*))^{\hat{*}}= 
\psi(x)\hatcirc \Phi_{s(x)}^{\hat{*}}$. Furthermore since the identities 
$(\Phi^{*_1})^{*_1}=\Phi$, $(\Phi\circ_1\Psi)^{*_1}=\Psi^{*_1}\circ_1\Phi^{*_1}$ and $(\Phi\circ_0\Psi)^{*_1}=\Phi^{*_1}\circ_0\Psi^{*_1}$ are satisfied, we have that $\Phi\mapsto \Phi^{*_1}$ provides an involution over 1-arrows for the 2-category $(\Cs^{(2)},\circ_0,\circ_1)$.  
\item 
Pointwise linear combinations of bounded natural transformations 
$\alpha\cdot\Phi+\Psi:o\mapsto \alpha\hatcdot \Phi_o\hatplus \Psi_o$, with $\alpha\in \CC$, are bounded natural transformations, since 
$\|\alpha\hatcdot \Phi_o \hatplus \Psi_o\|_{\hat{\Cs}}\leq|\alpha|\ \|\Phi_o\|_{\hat{\Cs}}+\|\Psi_o\|_{\hat{\Cs}}$, and hence the home-sets $\Cs^{(2)}_{\psi\phi}$ are vector spaces with such linear structure. 

Since $\Phi\circ_j(\alpha\cdot \Psi_1+\Psi_2)=\alpha\cdot (\phi\circ_j\Psi_1)+(\Phi\circ_j\Psi_2)$, for $j=0,1$, and similarly for the first argument, the previously defined compositions $\circ_0,\circ_1$ are home-set-wise bilinear maps. 

The involution $*_0$ is home-set-wise conjugate linear: $(\alpha\cdot\Phi+\Psi)^{*_1}=\cj{\alpha}\cdot\Phi^{*_1}+\Psi^{*_1}$. 
\item 
Every home-set $\Cs^{(2)}_{\psi\phi}$ becomes a normed space with the norm $\|\Phi\|:=\sup_{o\in \Cs^0}\|\Phi_o\|_{\hat{\Cs}}$. 
Since every Cauchy net $(\Phi^{(\lambda)})_{\lambda\in \Lambda}$ induces, for all objects $o\in \Cs^0$, a Cauchy net 
$(\Phi_o^{(\lambda)})_{\lambda\in \Lambda}$ in the Banach space $\hat{\Cs}^1_{\psi(o)\phi(o)}$, the pointwise limit 
$\Phi:o\mapsto \lim_{\lambda\in \Lambda}(\Phi^{(\lambda)}_o)$ exists. $\Phi$ is a natural transformation (as can be seen passing to the limit in the expression $\psi(x)\hatcirc \Phi^{(\lambda)}_{s(x)}= \Phi^{(\lambda)}_{t(x)}\hatcirc \phi(x)$, using the continuity of $\hatcirc$) and is bounded: 
$\|\Phi_o^{(\lambda)}\|_{\hat{\Cs}}\leq \|\Phi^{(\lambda)}_o-\Phi^{(\mu)}_o\|_{\hat{\Cs}}+\|\Phi^{(\mu)}_o\|_{\hat{\Cs}}\leq 
\|\Phi^{(\lambda)}-\Phi^{(\mu)}\|+\sup_{o\in \Cs^0}\|\Phi^{(\mu)}_o\|_{\hat{\Cs}}\leq \epsilon+\|\Phi^{(\mu)}\|$, 
eventually in $\lambda$, for all 
$o\in \Cs^0$. As a consequence $\Cs^{(2)}_{\psi\phi}$ is a Banach space.
\item 
The inequality $\|\Phi\circ_1\Psi\|\leq\|\Phi\|\cdot\|\Psi\|$ 
is obtained taking the supremum of the pointwise submultiplicativity of the norms in 
$\hat{\Cs}$ and similarly, for the C*-property,  \\ 
$\|\Phi^{*_1}\circ_1\Phi\|=\sup_{o\in \Cs^0}\|\Phi_o^{*_1}\hatcirc \Phi_o\|_{\hat{\Cs}}
=(\sup_{o\in \Cs^0}\|\Phi_o\|_{\hat{\Cs}})^2=\|\Phi\|^2$.  
\\  
Hence $(\Cs^{(2)},\circ_1,*_1,+,\cdot,\|\cdot \|)$ is a 1-C*-category and so $\Cs^{(2)}_{\phi\phi}$ is a C*-algebra for all $*$-functors $\phi$. 
Finally observe that, for each $\Cs\xrightarrow{\phi}\hat{\Cs}$, $\Cs^{(2)}_{\phi\phi}$ can be isometrically embedded into $\bigoplus_{o\in \Cs^0} \hat{\Cs}_{\phi(o)\phi(o)}$ in the obvious way, from which the positivity of every $\Phi^* \circ_1 \Phi$ follows at once. 
\xqed{\lrcorner} 
\end{itemize} 
\end{example}

\begin{proposition}\label{prop: nlr}
The category of small strict globular $n$-C*-categories of Longo-Roberts type with strict bounded $n$-transfors is an $(n+1)$-C*-category of 
Longo-Roberts type.
\end{proposition}
\begin{proof}
For $n=1$, the statement is described in the previous example~\ref{ex: lr}. 
Inductively, assuming the result for $n$, we prove it for $n+1$. 
Let $\Cf^{(n)}$ be a family of small globular $n$-C*-categories of Longo-Roberts type and, for all $\Cs,\hat{\Cs}\in \Cf^{(n)}$, consider the family of bounded $k$-transfors $\xymatrix{\Cs \rtwocell^{\Phi}_{\Psi}{\ \ \Xi^{(k)}}& \hat{\Cs}}$, for $k=1,\dots,n$, between $*$-functors $\Phi,\Psi:\Cs\to\hat{\Cs}$. 
By theorem~\ref{th: n-tr}, $\Cf^{(n)}$, with such bounded $n$-transfors, is already a strict globular $(n+1)$-category. 
For all $A\in \Cs^0$, the component $\Xi^{(n)}_A$, of a bounded $n$-transfor, between $k$-transfors $\Theta^{(k)},\Omega^{(k)}$, for $k=0,\dots,n-1$, is a globular $n$-cell 
$\xymatrix{{\Phi(A)} \rrtwocell^{\Omega^{(k)}_A}_{\Theta^{(k)}_A}{\quad \quad \Xi^{(n)}_A}& & {\Psi(A)}}$, in the \hbox{$n$-C*-category} 
$(\hat{\Cs},\hatcirc_0,\dots,\hatcirc_{n-1},\hat{\dag}_{n-1},\hatcdot,\hatplus,\| \cdot \|_{\hat{\Cs}})$ of Longo-Roberts type, hence $\Xi^{(n)}_A$ belongs to the Banach space $\hat{\Cs}_{\Theta^{(n-1)}_A\Omega^{(n-1)}_A}$.  
Addition and multiplication by scalars for $n$-transfors are defined ``componentwise'' by 
$(\Xi^{(n)}+\hat{\Xi}^{(n)})_A:=\Xi^{(n)}_A\, \hatplus \, \hat{\Xi}^{(n)}_A$ and $(\alpha\cdot\Xi^{(n)})_A:=\alpha\, \hatcdot \, \Xi^{(n)}_A$, and so the home-set $\Cf^{(n)}_{\Theta^{(n-1)}\Omega^{(n-1)}}$ of \hbox{$n$-transfors}, between $\Omega^{(n-1)}$ and $\Theta^{(n-1)}$, with the supremum norm $\|\Xi^{(n)}\|:=\sup_{A\in \As^0}\|\Xi^{(n)}_A\|_{\hat{\Cs}}$, is a Banach space. 

Compositions of $n$-transfors over objects of $\Cf^{(n)}$ are defined componentwise by the formula 
\\ 
$\xymatrix{\Cs \rtwocell^{\Phi_1}_{\Psi_1}{\ \ \Xi_1^{(n)}}& 
\hat{\Cs}\rtwocell^{\Phi_2}_{\Psi_2}{\ \ \Xi_2^{(n)}}& {\tilde{\Cs}}}$, \quad  
$(\Xi_2^{(n)}\circ_0 \Xi_1^{(n)})_A:=(\Xi_1^{(n)})_{\Psi_1(A)}\, {\tilde{\circ}}_0 \, \Phi_2((\Xi_1^{(n)})_A)$, \quad for $A\in \As^0$; 
\\
similarly, for $k=1,\dots,n$, the remaining compositions $\circ_k$ in $\Cf^{(n)}$ and the unique involution $*_n$ of $\Cf^{(n)}$ are also defined componentwise by 
$(\Xi^{(n)}\circ_k\hat{\Xi}^{(n)})_A:=\Xi^{(n)}_A \, \hatcirc_{k-1}\, \hat{\Xi}^{(n)}_A$, and 
$(\Xi^{(n)})^{*_{n}}_A:=(\Xi^{(n)}_A)^{\hat{\dag}_{n-1}}$. 

It follows that distributivity of compositions, the C*-property $\|(\Xi^{(n)})^{*_n}\circ_n\Xi^{(n)}\|=\|\Xi^{(n)}\|^2$ and the positivity of $(\Xi^{(n)})^{*_n}\circ_n\Xi^{(n)}$ in the C*-algebra ${}_{(n-1)}\Cf^{(n)}_{\Omega^{(n-1)}\Omega^{(n-1)}}$, are all derived by direct componentwise calculations, making use of the fact that $\Cs$ is a C*-category of Longo-Roberts type. 
\end{proof}

We now state our main definition of $n$-C*-category with non-commutative exchange. 

\begin{definition}\label{def: nC*} 
A \emph{fully involutive strict globular $n$-C*-category with non-commutative exchange} (also called \emph{quantum fully involutive strict globular $n$-C*-category}), is a fully involutive strict \hbox{$n$-cat}\-egory, with non-commutative exchange, denoted as $(\Cs,\circ_0,\dots,\circ_{n-1},*_0,\dots,*_{n-1},+,\cdot,\|\cdot\|)$, such that:~\footnote{\label{foo: norms} 
\emph{Important note}: our simple assumption here of a unique norm having the C*-property for all pairs $(\circ_k,*_k)$, will turn out to be too restrictive for non-trivial examples. A more appropriate axiomatic will allow, for all $k\in 0,\dots, n-1$, a different norm $\|\cdot\|_k$ having the C*-property only for the pair $(\circ_k,*_k)$. 
Although in the finite dimensional case all norms are equivalent and hence the Banach property of the fibers is satisfied at the same time for all $\|\cdot\|_k$ (as in the case of the examples in theorem~\ref{th: obs}), in this context the simultaneous requirement of Banach completeness for all norms $\|\cdot\|_k$ seems likely to be too restrictive. A possible solution will be to require the stronger property of completeness in the uniformity induced by the family of all norms. We will return elsewhere on these infinite dimensional technical details. 
}
\begin{itemize}
\item
$(\Cs,\circ_0,\dots,\circ_{n-1},+,\cdot)$ is an $n$-algebroid at every level $p=0,\dots,n-1$,
\item
for all $a,b\in \Cs^{n-1}$, the home-set $\Cs_{ab}$ is a Banach space with the norm $\|\cdot\|$,\footnote{Usually in examples $\Cs_{ab}$ is actually a Banach space, but it might also be a disjoint union of Banach spaces or possibly even a horizontal categorification of a $\CC$-vector space (where only the diagonal additive home-sets are Banach spaces).} 
\item 
for all $0\leq p<n$, $\|x\circ_p y\|\leq\|x\|\cdot\|y\|$, whenever $x\circ_p y$ exists, 
\item 
for all $0\leq p<n$, $\|x^{*_p}\circ_p x\|=\|x\|^2$, holds for all $x\in \Cs$,\footnote{Note that the conditions here assumed already imply that for all level-$p$ identities $e\in \Cs^p$ the home-set ${}_{(n-1)}\Cs^n_{ee}$ is a C*-algebra $({}_{(n-1)}\Cs_{ee},\circ_p,*_p,+,\cdot,\|\cdot \|)$, with composition $\circ_p$ and involution $*_p$.} \footnote{
Imposing the C*-property only whenever $x^{*_p}\circ_px\in{}_{(n-1)}\Cs_{ee}$, where $e$ is the $p$-source of $x$, is certainly possible, but it is a less restrictive condition that would result in a much more general structure (as soon as $n>1$).} 
\item 
for all $0\leq p<n$, $x^{*_p}\circ_p x$ is positive in ${}_{(n-1)}\Cs_{ee}$, where $e$ is the $p$-source of $x$. 
\end{itemize}
A \emph{partially involutive strict globular $n$-C*-category with non-commutative exchange} 
(also called \emph{quantum partially involutive strict globular $n$-C*-category}) will be equipped with only a subfamily of the previous involutions and will satisfy only those properties that can be formalized using the existing involutions. 
\end{definition}

\begin{remark}
Of course we can state the previous definition imposing the more restrictive exchange property (in which case we will omit the term ``non-commutative/quantum'', in the denomination), \hbox{$n$-C*-categories} of Longo-Roberts type are just special cases of partially involutive strict globular \hbox{$n$-C*-categories}. 
\xqed{\lrcorner}
\end{remark}

\medskip 

We now examine the most natural elementary examples of strict globular $n$-C*-categories. 

We start with a C*-categorical version of examples~\ref{ex: intertwiners} and~\ref{ex: intertwiners2}. 
\begin{example}\label{ex: 2c*}
Consider a family $\Cs^0$ of unital C*-algebras and let $\Cs^1$ be the groupoid of invertible $*$-homomorphisms between C*-algebras in $\Cs^0$. 
The family $\Cs^2$ of intertwiners $\xymatrix{\As \rtwocell^{\Phi}_{\Psi}{e}& \Bs}$, i.e.~the elements $e\in \Bs$ such that 
$e\cdot_\Bs\Phi(x)=\Psi(x)\cdot_\Bs e$, for all $x\in \As$, becomes a fully involutive 2-C*-category with the usual operations of composition and involution described in example~\ref{ex: intertwiners}. 

\medskip 

Again, there is a horizontal categorification of this result in parallel with example~\ref{ex: intertwiners2}. 
If $\Cs^0$ is a family of 1-C*-categories, and $\Cs^1$ is the family of the invertible $*$-isomorphisms between 
1-C*-categories in $\Cs^0$, the family $\Cs^2$ of bounded natural transformations (1-transfors) 
$\xymatrix{\As \rtwocell^{\Phi}_{\Psi}{\Xi}& \Bs}$ between invertible \hbox{$*$-functors} in $\Cs^1$, becomes a fully involutive 2-C*-category with the same compositions and involutions considered in example~\ref{ex: intertwiners2}. 
\xqed{\lrcorner}
\end{example}

We also have a general recursive construction of $n$-C*-categories in the spirit of theorems~\ref{th: n-tr} 
and~\ref{th: iso-n-tr}. 
\begin{theorem}
The family of small totally involutive $n$-C*-categories with bounded strict $n$-transfors, between invertible $*$-functors, constitutes a fully involutive $(n+1)$-C*-category. 
\end{theorem}
\begin{proof}
From example~\ref{ex: 2c*} we have a fully involutive 2-C*-category $\Cf^{(1)}$ of bounded 1-transfors between $*$-isomorphisms of a given family of 1-C*-categories. 

Let $\Cf^{(n)}$ be a family of fully involutive $n$-C*-categories $\As,\hat{\As},\dots$ equipped with bounded $k$-transfors, for $k=1,\dots,n$, $\xymatrix{\As \rtwocell^{\Phi}_{\Psi}{\ \ \Xi^{(k)}}& \hat{\As}}$, between invertible $*$-functors $\Phi,\Psi$. 

Since every $n$-C*-category is an $(n-1)$-C*-category (by truncating to its $(n-1)$-arrows), $\Cf^{(n)}$ can be seen as a family of fully involutive $(n-1)$-C*-categories. Together with the family of bounded \hbox{$k$-transfors}, for $k=0,\dots,n-1$, $\Cf^{(n)}$ becomes a fully involutive $n$-C*-category, by the inductive hypothesis.  

We need to show that $\Cf^{(n)}$, with the bounded $k$-transfors, for $k=0,\dots,n$, is also a fully involutive \hbox{$(n+1)$-C*-category}. 
By theorem~\ref{th: iso-n-tr}, we know that $\Cf^{(n)}$, with the family of bounded $k$-transfors, for $k=0,\dots,n$, is a fully involutive $(n+1)$-category. 

Recall that each component $\Xi^{(n)}_A$, $A\in \As^0$, of a strict bounded $n$-transfor $\Xi^{(n)}:\As^0\to\hat{\As}^n$, between $k$-transfors $\Theta^{(k)},\Omega^{(k)}$, for $k=0,\dots,n-1$, between invertible $*$-functors $\Phi,\Psi$ from $\As$ to $\hat{\As}$, is a globular $n$-cell 
$\xymatrix{{\Phi(A)} \rrtwocell^{\Omega^{(k)}_A}_{\Theta^{(k)}_A}{\quad \quad \Xi^{(n)}_A}& & {\Psi(A)}}$, in the 
$n$-C*-category $(\hat{\As},\hatcirc_0,\dots,\hatcirc_{n-1},\hat{\dag}_0,\dots,\hat{\dag}_{n-1},\hatcdot,\hatplus,
\| \cdot \|_{\hat{\As}})$.   
Since $\Xi^{(n)}_A$ is an element of the Banach space $\hat{\As}_{\Theta^{(n-1)}_A\Omega^{(n-1)}_A}$, we can immediately define ``componentwise'' the operations of addition and multiplication by scalars for $n$-transfors: 
$(\Xi^{(n)}+\hat{\Xi}^{(n)})_A:=\Xi^{(n)}_A\, \hatplus \, \hat{\Xi}^{(n)}_A$ and $(\alpha\cdot\Xi^{(n)})_A:=\alpha\, \hatcdot \, \Xi^{(n)}_A$; furthermore, with the norm $\|\Xi^{(n)}\|:=\sup_{A\in \As^0}\|\Xi^{(n)}_A\|_{\hat{\As}}$, the family $\Cf^{(n)}_{\Theta^{(n-1)}\Omega^{(n-1)}}$ of \hbox{$n$-transfors}, between $\Omega^{(n-1)}$ and $\Theta^{(n-1)}$ is a Banach space. 

Since compositions $\circ_k$ and involutions $*_k$ in the $(n+1)$-category $\Cf^{(n)}$, for $k=1,\dots,n$ are similarly componentwise defined by 
$(\Xi^{(n)}\circ_k\hat{\Xi}^{(n)})_A:=\Xi^{(n)}_A \, \hatcirc_{k-1}\, \hat{\Xi}^{(n)}_A$, 
$(\Xi^{(n)})^{*_k}_A:=(\Xi^{(n)}_A)^{\hat{\dag}_{k-1}}$, the distributivity with respect to addition of $\circ_1,\dots,\circ_n$ is valid and, for the same reasons, the positivity of $(\Xi^{(n)})^{*_k}\circ_k\Xi^{(n)}$ and the C*-properties $\|(\Xi^{(n)})^{*_k}\circ_k\Xi^{(n)}\|=\|\Xi^{(n)}\|^2$, for $k=1,\dots,n$, hold. 

We only need to show the distributivity of $\circ_0$, the positivity of $(\Xi^{(n)})^{*_0}\circ_0\Xi^{(n)}$ and the C*-property $\|(\Xi^{(n)})^{*_0}\circ_0\Xi^{(n)}\|=\|\Xi^{(n)}\|^2$. 

The distributivity follows immediately from the definition of $\circ_0$-composition of $n$-transfors: 
\\ 
$\xymatrix{\As \rtwocell^{\Phi^{(0)}_1}_{\Psi^{(0)}_1}{\ \ \Xi_1^{(n)}}& 
\hat{\As}\rtwocell^{\Phi^{(0)}_2}_{\Psi^{(0)}_2}{\ \ \Xi_2^{(n)}}& {\tilde{\As}}}$, \quad  
$(\Xi_2^{(n)}\circ_0 \Xi_1^{(n)})_A:=(\Xi_1^{(n)})_{\Psi_1^{(0)}(A)}\, {\tilde{\circ}}_0 \, \Phi_2^{(0)}((\Xi_1^{(n)})_A)$, \quad for $A\in \As^0$, 
\\
and from the distributivity of compositions over objects in the small $n$-C*-categories belonging to $\Cf^{(n)}$. 

For the C*-property and the positivity, we see that, for all $A\in \As^0$: 
\begin{align*} 
((\Xi^{(n)})^{*_0}\circ_0\Xi^{(n)})_A&=((\Xi^{(n)})^{*_0})_{\Psi(A)}\hatcirc_0\Phi^{-1}(\Xi^{(n)}_A)=
\Phi^{-1}(\Xi^{(n)}_A)^{\hat{\dag}_0}\hatcirc_0\Phi^{-1}(\Xi^{(n)}_A), \ \text{hence}
\\ 
\|((\Xi^{(n)})^{*_0}\circ_0\Xi^{(n)})\|&=\sup_{A\in \As^0}\|((\Xi^{(n)})^{*_0}\circ_0\Xi^{(n)})_A\|_{\As}
=\sup_{A\in \As^0}\|\Phi^{-1}(\Xi^{(n)}_A)^{\hat{\dag}_0}\hatcirc_0\Phi^{-1}(\Xi^{(n)}_A)\|_{\As}
\\ 
&=\sup_{A\in \As^0}\|\Phi^{-1}(\Xi^{(n)}_A)\|_{\As}^2=\|\Xi^{(n)}\|^2   
\end{align*}
and $\Phi^{-1}(\Xi^{(n)}_A)^{\hat{\dag}_0}\hatcirc_0\Phi^{-1}(\Xi^{(n)}_A)$ is positive, for all $A\in \As^0$, in the 
C*-algebra ${}_{(n-1)}\As_{AA}$. 
\end{proof}

The following examples utilize some elementary properties of Hilbert C*-modules~\cite{Pa,Ri1,Ri2} and Rieffel's Morita theory for imprimitivity Hilbert C*-bimodules~\cite{Ri3}; the reader is referred, for example, to~\cite[section III.7, definition III.7.6.1]{Bl} and also~\cite{BCL-} for such properties and further references. 

\begin{example}\label{ex: 2bimod}
A fully involutive 2-C*-category $(\Cf,\circ_0,\circ_{1},\cj{\phantom{x}},*,+,\cdot,\|\cdot\|)$,
of bi-adjointable morphisms of imprimitivity Hilbert C*-bimodules can be constructed as follows: 
\begin{itemize}
\item 
Consider as objects a given family $\Cf^0:=\{\As,\Bs,\dots\}$ of unital C*-algebras.
\item 
Take as 1-arrows the family $\Cf^1$ consisting of all the imprimitivity C*-bimodules between the unital C*-algebras in the family $\Cf^0$. Specifically, we denote by ${}_\As\Ms_\Bs$ an $\As$-$\Bs$-imprimitivity bimodule thought as a 1-arrow with source $\Bs$ and target $\As$, for any $\As,\Bs\in\Cf^0$. 
Whenever two C*-algebras $\As,\Bs\in\Cf^0$ are not Morita equivalent, we take $\Cf^1_{\As\Bs}=\varnothing$.
\item 
Define, for any ${}_\As\Ms_\Bs,{}_\As\Ns_\Bs\in\Cf^1$, the family $\Cf^2_{\Ns\Ms}$ of 2-arrows with source $\Ms$ and target $\Ns$ consisting of all the homomorphisms ${}_\As\Ms_\Bs\xrightarrow{\Phi}{}_\As\Ns_\Bs$
of such bimodules (i.e.~those additive maps such that $\Phi(a\cdot x\cdot b)=a\cdot \Phi(x)\cdot b$, for all $a\in \As$, $b\in\Bs$, $x\in \Ms$) that are \emph{bi-adjointable} in the sense that there exists a necessarily unique homomorphism of bimodules $\Phi^*:\Ns\to \Ms$ that is simultaneously right-adjoint (i.e.~with respect to the $\Bs$-valued inner product) and left-adjoint (i.e.~with respect to the $\As$-valued inner product) to $\Phi$.\footnote{For commutative C*-algebras, if a homomorphism of bimodules that is both right and left adjointable, the right and left adjoints must coincide.} 
\item 
We take $*:\Cf^1\to\Cf^1$ as the involution over 1-arrows and we note immediately that 
$(\Phi^*)^*=\Phi$. 
\item
As ``vertical'' composition of 2-arrows, we consider the usual composition of adjointable homomorphisms of imprimitivity bimodules (this is necessarily biadjointable since the composition of right/left adjointable maps between right/left-correspondences is right/left adjointable) and we see that $(\Phi\circ_1\Psi)^*=\Psi^*\circ_1\Phi^*$, for all $\Phi,\Psi\in{}_0\Cf^2_{\As\Bs}$, with $\Psi:{}_\As\Ms_\Bs\to{}_\As\Ns_\Bs$ and $\Phi:{}_\As\Ns_\Bs\to{}_\As\Ps_\Bs$. 
\item 
In order to define the ``horizontal'' composition of 2-arrows, construct first the 1-groupoid $\Xf$ of the imprimitivity Hilbert C*-bimodules in $\Cf^1$ under Rieffel tensor products and conjugations and consider the tautological Fell bundle over $\Xf$, with fibers $\Xs$ over $\Xs\in\Xf$, whose total space is the free involutive 1-category generated by all the elements $x\in \Ms$ with $\Ms\in\Cf^1$, where the fiberwise composition 
$(x,y)\mapsto x\otimes_\Bs y$, for $(x,y)\in{}_\As\Xs_\Bs\times{}_\Bs\Ys_\Cs$, strictly implements the Rieffel tensor product of the fibers $\Xs,\Ys\in\Xf$ and the fiberwise conjugation $x\mapsto \cj{x}$, for $x\in\Xs\in\Xf$, implements the Rieffel dual of the fibers $\Xs\in\Xf$. 

The ``horizontal'' composition of 2-arrows ${}_\Bs\Ms^2_{\Cs}\xrightarrow{\Psi}{}_\Bs\Ns^2_{\Cs}$, 
${}_\As\Ms^1_{\Bs}\xrightarrow{\Phi}{}_\As\Ns^1_{\Bs}$ is obtained, via the universal factorization property for Rieffel internal tensor products of imprimitivity bimodules, as the unique homomorphism of bimodules satisfying $(\Phi\circ_0\Psi)(x\otimes_\Bs y):=\Phi(x)\otimes_\Bs\Psi(y)$, for all 
$x\in \Ms^1$ and $y\in \Ms^2$. This is well-defined, since $\Phi\circ_0\Psi :{}_\As\Ms^1\otimes_\Bs \Ms^2_\Cs \to {}_\As\Ns^1\otimes_\Bs \Ns^2_\Cs$ is bi-adjointable when $\Phi,\Psi$ are bi-adjointable.  
Furthermore we also have that $(\Phi\circ_0\Psi)^*=\Phi^*\circ_0 \Psi^*$. 
\item 
The involution over objects of a 2-arrow ${}_\As\Ms_\Bs\xrightarrow{\Phi}{}_\As\Ns_\Bs$ is obtained as the bi-adjointable homomorphism of bimodules ${}_\Bs\cj{\Ns}_\As\xrightarrow{\cj{\Phi}}{}_\Bs\cj{\Ms}_\As$ (where ${}_\Bs\cj{\Ms}_\As$ denotes the Rieffel conjugate of ${}_\As\Ms_\Bs$), as fibers in the previous strictification Fell bundle, defined by $\cj{\Phi}(\cj{x}):=\cj{\Phi(x)}$, for all $\cj{x}\in \cj{\Ms}$. 
\item
The identities $\cj{(\Phi\circ_0\Psi)}=\cj{\Psi}\circ_0\cj{\Phi}, \quad \cj{(\Phi\circ_1\Psi)}=
\cj{\Phi}\circ_1 \cj{\Psi}, \quad \cj{\cj{\Phi}}=\Phi$, $(\cj{\Phi})^*=\cj{\Phi^*}$ and the usual exchange property $(\Phi\circ_0\Psi)\circ_1(\Theta\circ_0\Xi)=(\Phi\circ_1\Theta)\circ_0(\Psi\circ_1\Xi)$, hold true, as can be checked by a direct computation. 
\item
Each home-set $\Cf^2_{\Ns\Ms}$, with ${}_\As\Ms_\Bs,{}_\As\Ns_\Bs\in \Cf^1_{\As\Bs}$ becomes a normed space with linear structure defined by $(\Phi+\lambda\cdot\Psi)(x):=\Phi(x)+\lambda\cdot\Psi(x)$, for all $x\in \Ms$, $\lambda\in \CC$ and with norm given by 
$\|\Phi\|:=\sup_{x\in\Ms, \|x\|_\Ms\leq 1}\|\Phi(x)\|_\Ns$. The norm is well-defined since, by closed graph theorem, an adjointable operator between Hilbert C*-modules is bounded.

With respect to such linear spaces, all compositions are home-set-wise bilinear and all the involutions are conjugate linear. 
\item 
The $\circ_1$-submultiplicativity of the norm $\|\Phi\circ_1\Psi\|\leq\|\Phi\|\cdot\|\Psi\|$ holds for any functional composition of bounded linear maps. 
Since right (respectively left) adjointable maps of right (left) Hilbert C*-modules are a C*-category (and in our case the right and left norms of bi-adjointable homomorphisms coincide) the $\circ_1$-C*-property $\|\Phi^*\circ_1\Phi\|=\|\Phi\|^2$ holds. 
\item
In order to prove the $\circ_0$-submultiplicativity, given the pair of 2-arrows ${}_\As\Ms^1_\Bs\xrightarrow{\Phi}{}_\As\Ns^1_\Bs$ and ${}_\As\Ms^2_\Bs\xrightarrow{\Psi}{}_\As\Ns^2_\Bs$, denoting by $I_{\Ms^1}$, $I_{\Ms^2}$ the identities of the respective bimodules and by $\Cs_\Ms$ the C*-algebra $(\Cs_{\Ms\Ms},\circ_1,*)$, for $\Ms\in\Cs^1$, we notice that since the map $\Xi\mapsto\Xi\circ_0I_{\Ms^2}$, is a 
$*$-homomorphism between the C*-algebras $\Cs_{\Ms^1}$ and $\Cs_{\Ms^1\otimes_\Bs\Ms^2}$, and in a similar way the map 
$\Theta\mapsto I_{\Ms^1}\circ_0\Theta$ is a $*$-homomorphism between the C*-algebras $\Cs_{\Ms^2}$ and $\Cs_{\Ms^1\otimes_\Bs\Ms^2}$, they are both contractive; and hence we have:  
\begin{align*}
\|\Phi &\circ_0\Psi\|^2=\|(\Phi\circ_0\Psi)^*\circ_1(\Phi\circ_0\Psi)\|=
\|(\Phi^*\circ_1\Phi)\circ_0(\Psi^*\circ_1\Psi)\|
\\ 
&=\|[(\Phi^*\circ_1\Phi)\circ_0 I_{\Ms^2}]\circ_1[I_{\Ms^1}\circ_0(\Psi^*\circ_1\Psi)]\|
\\
&\leq \|(\Phi^*\circ_1\Phi)\circ_0 I_{\Ms^2}\|\cdot \|I_{\Ms^1}\circ_0(\Psi^*\circ_1\Psi)\|
\leq\|\Phi^*\circ_1\Phi\| \cdot \|\Psi^*\circ_1\Psi\| = (\|\Phi\|\cdot\|\Psi\|)^2. 
\end{align*}
\item 
To show the $\circ_0$-C*-property, first of all we have $\|\cj{\Phi}\circ_0\Phi\|\leq\|\cj{\Phi}\|\cdot\|\Phi\|\leq\|\Phi\|^2$, since $\|\cj{\Phi}\|=\|\Phi\|$. 

The second inequality is obtained, for $\Phi:{}_\As\Ms_\Bs\to{}_\As\Ns_\Bs$, as follows: 
\begin{align*}
&\|\cj{\Phi} \circ_0\Phi\|^2 \geq \sup_{x\in \Ms, \|x\|\leq 1}\|\cj{\Phi}(\cj{x})\otimes_\As\Phi(x)\|^2
=\sup_{x\in \Ms,\|x\|\leq 1} \|\ip{\cj{\Phi}(\cj{x})\otimes_\As\Phi(x)}{\cj{\Phi}(\cj{x})\otimes_\As\Phi(x)}_\Bs\| 
\\
&=\sup_{x\in \Ms,\|x\|\leq 1} \|\ip{\Phi(x)}{\ip{\cj{\Phi(x)}}{\cj{\Phi(x)}}_\As\Phi(x)}_\Bs\|
=\sup_{x\in \Ms,\|x\|\leq 1}\|\ip{\Phi(x)}{{}_\As\ip{\Phi(x)}{\Phi(x)}\Phi(x)}_\Bs\|
\\ 
&=\sup_{x\in \Ms,\|x\|\leq 1}\|\ip{\Phi(x)}{\Phi(x)\ip{\Phi(x)}{\Phi(x)}_\Bs}_\Bs\|
=\sup_{x\in \Ms,\|x\|\leq 1}\|\ip{\Phi(x)}{\Phi(x)}_\Bs\ip{\Phi(x)}{\Phi(x)}_\Bs\|
\\ 
&=\sup_{x\in \Ms,\|x\|\leq 1}\|\ip{\Phi(x)}{\Phi(x)}\|^2
=\|\Phi\|^4. 
\end{align*}
\item 
For the completeness of $\Cf^2_{\Ns\Ms}$, given a Cauchy net $\Phi_\mu\in \Cf^2_{\Ns\Ms}$, for all $x\in \Ms$, we see that the net $\Phi_\mu(x)$ is Cauchy in the Banach space $\Ms$ and converges to $\Phi(x)$. From the $\circ_1$-C*-property and the \hbox{$\circ_1$-submultiplicativity} of the norm, we obtain the isometry of the $*$-involution, and hence $\|\Phi_\mu^*\|=\|\Phi_\mu\|$, and so $\Phi^*_\mu(y)$ is a Cauchy net as well, for all $y\in \Ns$. 
Passing to the limit in the bi-adjointability conditions ${}_\As\ip{\Phi_\mu(x)}{y}={}_\As\ip{x}{\Phi^*_\mu(y)}$ and 
$\ip{\Phi_\mu(x)}{y}_\Bs=\ip{x}{\Phi^*_\mu(y)}_\Bs$, for $\Phi_\mu$, we immediately obtain that the map $x\mapsto \Phi(x)$ is bi-adjointable, hence linear and bounded, and the convergence in $\Cf^2_{\Ns\Ms}$ of the net $\Phi_\mu$. 
\item
By Eckmann-Hilton collapse, the C*-algebra of intertwiners of the Morita identity bimodule $\As$ is a commutative C*-algebra under the common product $\circ_0=\circ_1$ (see also P.Zito~\cite{Z}). 

Since two involutions that satisfy the C*-property, for a common product and the same norm,  necessarily coincide (see H.F.Bohnenblust-S.Karlin~\cite[theorem~9]{BK}) we see that $\Phi^*=\cj{\Phi}$, for all $\Phi\in {}_{1}\Cf^{2}_{\As\As}$. Hence, for all such intertwiners $\Phi$, we have $\cj{\Phi}\circ_0\Phi=\Phi^*\circ_1\Phi$ that is a positive element in ${}_{1}\Cf^{2}_{\As\As}$.
\xqed{\lrcorner}
\end{itemize}
\end{example}

\begin{example}
As a particular case of example~\ref{ex: 2bimod}, if $\Cs$ is a full 1-C*-category, the family of bi-adjointable endomorphisms $\xymatrix{B\rtwocell^{\Cs_{AB}}_{\Cs_{AB}}{\Phi}& A}$ of each one of the imprimitivity Hilbert C*-bimodules $\Cs_{AB}$, $A,B\in \Cs^0$, is a fully involutive 2-C*-category (where all the 2-arrows are loops over 1-arrows). 
\xqed{\lrcorner}
\end{example}

We describe here a horizontal categorification of example~\ref{ex: 2bimod}. 
For this purpose, we recall (see P.Mitchener~\cite[section~8]{M}) a preliminary definition of Hilbert C*-bimodule between 1-C*-categories: this is just a ``C*-operator algebraic'' version of the usual notion of ``categorical bimodule'' (the horizontal categorification of a bimodule over a monoid). 

\begin{example}\label{ex: hbim}
We have a fully involutive 2-C*-category of bi-adjointable maps between imprimitivity bimodules of full 1-C*-categories. 
\xqed{\lrcorner}
\end{example}

The following remark goes in the direction of a vertical categorification of C*-Morita theory. 
\begin{remark}\label{rem: mor} 
Although here we are not entering into further details, there is little doubt that it is possible to produce a vertical categorification of example~\ref{ex: 2bimod} providing a recursive construction of (fully involutive) higher \hbox{C*-categories} (with non-commutative exchange) and an ``operator categorical'' analog of theorem~\ref{th: iso-n-tr}; namely, given a family $\Cf^0$ of (fully involutive) n-C*-categories, the bi-adjointable morphims between pairs of imprimitivity bimodules ${}_{\Cs_1}\Ms_{\Cs_2}\xrightarrow{\Phi} {}_{\Cs_1}\Ns_{\Cs_2}\in \Cf^1$, between 
$n$-C*-categories $\Cs_1,\Cs_2\in \Cf^0$, are the $(n+1)$-arrows of a (fully involutive) $(n+1)$-C*-category $\Cf$.  
To deal with such a construction, one needs to introduce the notion of (imprimitivity) higher-C*-bimodules between \hbox{$n$-C*-categories} along the lines already mentioned in~\cite{BCL3}. 

\medskip 

There is a 2-categorical functor (injective on $2$-arrows and surjective on objects) from the strict fully involutive 2-C*-category of example~\ref{ex: 2c*} into the strictified fully involutive 2-C*-category of example~\ref{ex: 2bimod}, that to every isomorphism $\Phi:\As\to\Bs$ of C*-categories associates the C*-categorical imprimitivity Hilbert C*-bimodule ${}_\Phi\Bs$, obtained by left-twisting by $\Phi$, the identity bimodule ${}_\Bs\Bs_\Bs$ and associating to every intertwiner 
$\xymatrix{\As\rtwocell^{\Phi}_{\Psi}{\Xi}& \Bs}$ the bi-adjointable morphism of C*-categorical bimodules 
${}_\Phi\Bs\xrightarrow{M_\Xi}{}_\Psi\Bs$, $M_{\Xi_A}:({}_\Phi\Bs)_{\Phi(A)A}\to ({}_\Psi\Bs)_{\Psi(A)A}$ defined as 
$M_{\Xi_A}(b):=\Xi_A\cdot_\Bs b$, for all $b\in ({}_\Phi\Bs)_{\Phi(A)A}$, for all $A\in \As^0$. 
Such functor could be extended to a functor between strict fully involutive $n$-C*-categories. 

\medskip 

As a particular case, we mention a vertical categorification of example~\ref{ex: hbim}: if $\Cs$ is a fully involutive $n$-C*-category, the family of bi-adjointable endomorphisms of the home-sets $\Cs_{xy}$, with $x,y\in\Cs^{n-1}$, is a fully involutive $(n+1)$-C*-category.
\xqed{\lrcorner} 
\end{remark}

\begin{example}\label{rem: conjugates2} 
We continue here, examining the C*-categorical properties, the study of involutions induced by conjugations already started in example~\ref{rem: conjugates}. 

If $(\Cs,\otimes,\circ,*,\cdot,+,\| \ \|)$ is a Longo-Roberts 2-C*-category equipped with a unital involutive tensorial conjugation map that satisfies the traciability condition, the resulting fully involutive strict 2-category $(\Cs,\otimes,\circ,*,\dag,\cdot,+,\| \ \|)$ is an example of a fully involutive 2-category. 

Under our previous conditions the (unique) folding map is an involutive endofunctor of the C*-category 
$(\Cs,\circ, *,\cdot,+,\|\ \|)$ and hence it is a norm contractive map. Since it is involutive we obtain $\|\Phi_\bullet\|=\|\Phi\|$ and we immediately get the isometric property of the $\dag$-involution:  $\|\Phi^\dag\|=\|(\Phi^*)_\bullet\|=\|\Phi^*\|=\|\Phi\|$. 

Consider the following \emph{unitarity condition} for the conjugations maps $x\mapsto (R_x,\cj{R}_x)$: for all $x\in \Cs^1$ such that $\cj{x}\otimes x, x\otimes\cj{x}\in \Cs^0$, $R_x$ and $\cj{R}_x$ are unitary elements of the C*-category $(\Cs,\circ,*,+, \cdot, \| \cdot \|)$, i.e.~for $x\in \Cs^1_{BA}$, such that $\cj{x}\otimes x=\iota^2(B)$ and $x\otimes\cj{x}=\iota^2(A)$, $R_x$ is a unitary element of the C*-algebra ${}_1\Cs_{BB}$ and $\cj{R}_x$ is a unitary element in the C*-algebra ${}_1\Cs_{AA}$. 

Under such unitarity condition, a 2-C*-category of Longo-Roberts type, with unital involutive tensorial conjugations $(R_x,\cj{R}_x)$ that satisfy traciability, becomes a fully involutive 2-C*-category. 

In order to prove this statement, we recall that from section~\ref{rem: conjugates} $(\Cs,\circ,*,+, \cdot, \| \cdot \|)$ is naturally equipped with a structure of fully involutive 2-category. We only need to show the C*-property $\|\Phi^\dag\otimes\Phi\|=\|\Phi\|^2$ and the positivity of $\Phi^\dag\otimes \Phi$, whenever $\Phi^\dag\otimes\Phi$ belongs to the C*-algebra ${}_1\Cs_{AA}$, where $A=s_0(\Phi)\in\Cs^0$.  

Whenever $\xymatrix{A\rtwocell^x_y{\Phi}&B}$ is such that $\Phi^\dag\otimes\Phi\in {}_1\Cs_{AA}$, we always have 
$\cj{x}\otimes x=\iota^2(B), \cj{y}\otimes y=\iota^2(B)$ and hence $R_x, R_y\in {}_1\Cs_{BB}$ and $\cj{R}_x, \cj{R}_y\in\Cs_{AA}$ are unitary elements in the respective C*-algebras. 

By the fact that left/right tensorization with elements of $\Cs^1$ is a C*-functor and C*-functors are always norm contractive in a C*-category, we have $\|\Phi\| \leq \|\cj{x}\otimes (\Phi^*\circ\Phi)\|\leq \|x\otimes\cj{x}\otimes\Phi\|=\|\iota^2(A)\otimes\Phi\|=\|\Phi\|$.
Hence, since conjugation by unitary element is a norm preserving operation, we immediately obtain: 
\begin{equation}\label{eq: otimesc}
\|\ip{\Phi}{\Phi}_B\|=\|R_x\circ(\cj{x}\otimes (\Phi^*\circ\Phi))\circ R^*_x\|=\|\cj{x}\otimes (\Phi^*\circ\Phi)\|=\|\Phi^*\circ\Phi\|=\|\Phi\|^2.   
\end{equation}
A direct computation of $\Phi^\dag\otimes \Phi$ gives: 
\begin{align*}
\Phi^\dag\otimes\Phi&=(\Phi^*)_\bullet\otimes \Phi
= [(R_x^*\otimes \cj{y})\circ(\cj{x}\otimes \Phi^*\otimes \cj{y})\circ(\cj{x}\otimes \cj{R}_y)]\otimes \Phi
\\ 
&=  (R_x^*\otimes \cj{y}\otimes y)\circ 
[(\cj{x}\otimes \Phi^*\otimes \cj{y}\otimes y)\circ 
(\cj{x}\otimes y \otimes \cj{y}\otimes \Phi) ] \circ 
(\cj{x}\otimes \cj{R}_y\otimes x) 
\\ 
&= R_x^*\circ 
[\cj{x}\otimes (\Phi^*\circ\Phi)] \circ 
(\cj{x}\otimes \cj{R}_y\otimes x).  
\end{align*}
Making use of unitarity and tensoriality of the conjugations: 
\begin{align*}
\|\Phi^\dag\otimes\Phi\|&=\|R_x^*\circ [\cj{x}\otimes (\Phi^*\circ\Phi)] \circ (\cj{x}\otimes \cj{R}_y\otimes x)\|
=\|[\cj{x}\otimes (\Phi^*\circ\Phi)] \circ (\cj{x}\otimes \cj{R}_y\otimes x)\|
\\
&=\|[\cj{x}\otimes (\Phi^*\circ\Phi)] \circ (\cj{x}\otimes \cj{R}_y\otimes x)\circ R_x\| 
=\|[\cj{x}\otimes (\Phi^*\circ\Phi)] \circ \cj{R}_{\cj{x}\otimes y}\| 
=\|\cj{x}\otimes (\Phi^*\circ\Phi)\|  
\end{align*}
and the C*-property follows comparing to equation~\ref{eq: otimesc}. 

Notice that, by Eckmann-Hilton argument, the C*-algebras ${}_1\Cs_{BB}$ is commutative and for elements 
$\Phi,\Psi\in {}_1\Cs_{BB}$, $\Phi\circ\Psi=\Phi\otimes \Psi$. Furthermore, for $\Phi\in {}_1\Cs_{BB}$, by the $\otimes$-C*-property and~\cite{BK}, $\Phi^\dag=\Phi^*$. 

\medskip 

Regarding positivity, under the requirement of \emph{triviality} of conjugations $(R_x,\cj{R}_x)=(\iota^2(B),\iota^2(A))$, for all 1-arrows $x$ in the groupoid of invertible elements of the category $(\Cs^1,\otimes)$, we clearly have that 
\begin{equation*}
\Phi^\dag\otimes\Phi=\ip{\Phi}{\Phi}_B=\cj{x}\otimes(\Phi^*\circ\Phi) \quad 
\text{is positive for all $\Phi\in {}_1\Cs_{BB}$.}
\end{equation*} 

The seemingly strong requirement of triviality of conjugations is actually mild: from the already available assumptions of unitality, involutivity, tensoriality, it follows that conjugations are Hermitian $R_x=R_x^\dag=R_x^*$. The unitarity, for $x$ invertible in $(\Cs^1,\otimes)$, essentially says that $R_x$ is a continuous function with modulus one on the spectrum of the commutative C*-algebra ${}_1\Cs_{t(x)t(x)}$ and hence constant $\pm 1$ on each connected component. 
In such a context, it is likely that in many cases of interest the standard choice of triviality for conjugations (up to scalars) is forced from the other requirements. 
\xqed{\lrcorner}
\end{example}

\begin{remark}
For $n=2$, when $\Cs^0$ consists of only one object, our definition of partially involutive strict globular 2-C*-category with non-commutative exchange is compatible with the generalization of monoidal C*-categories recently described by R.Blute-M.Comeau~\cite{BC}.  
\xqed{\lrcorner}
\end{remark}

\begin{remark}
For $n=2$, again when $\Cs^0$ consists of only one object, our definition of a partially involutive strict globular 2-C*-category with non-commutative exchange is a special case of the semitensor C*-categories introduced by S.Doplicher-C.Pinzari-R.Zuccante~\cite[section~2]{DPZ}. 
\xqed{\lrcorner}
\end{remark}

The usual process of ``bundlification'' can be applied to our definition of strict quantum $n$-C*categories: 

\begin{definition}
An \emph{$n$-Fell bundle (with non-commutative exchange)} is given by a Banach bundle $(\Es,\pi,\Xs)$ where $\pi:\Es\to\Xs$ is an $n$-$*$-functor between fully involutive topological strict $n$-categories (with non-commutative exchange), such that: 
\begin{itemize}
\item 
the compositions $\circ_p$ are bilinear whenever defined,\footnote{This means that $(\Es,\pi,\Xs)$ is an $n$-algebroid at all levels.} 
\item 
the involutions $*_p$ are fiberwise conjugate-linear, 
\item
$\|x\circ_py\|\leq\|x\|\cdot\|y\|$, for all $\circ_p$-composable $x,y\in \Es$, 
\item
$\|x^{*_p}\circ_p x\|=\|x\|^2$ holds, for all $p$, for all $x\in \Es$,\footnote{As in the definition of higher C*-category, imposing the C*-property only whenever $\pi(x^{*_p}\circ_p x)$ is an idempotent (or a $p$-identity) in $\Xs$, is for sure possible, but it results in a much more general structure (even in the case of ordinary Fell bundles).}   
\item 
for all $x\in \Es$, $x^{*_p}\circ_p x$ is positive whenever $\pi(x^{*_p}\circ_p x)$ is an idempotent in $\Xs$.\footnote{This condition is meaningful since the previous axioms already assure that the fiber $\Es_{\pi(x^{*_p}\circ_p x)}$ is a C*-algebra.} 
\end{itemize}
\end{definition}

\subsection{Hypermatrices, Hyper-C*-algebras and Higher Convolutions} \label{sec: hyper-conv} 

In this subsection we finally start to provide the long awaited direct examples of strict (fully) involutive higher C*-categories with non-commutative exchange. We will discuss mostly discrete finite cases, that are already of great interest. 

The first step consists in reformulating the usual ``innocent'' definition of complex square matrix, making it apparently quite ``convoluted'', but ready for generalizations.

\begin{proposition}
A complex square matrix $[x^i_j]\in \MM_{N\times N}(\CC)$ of order $N\in \NN_0$ is a section of the Fell line-bundle $\Es:=\Xs\times\CC$ over the discrete finite pair groupoid $\Xs: \Xs^1\rightrightarrows \Xs^0$ of the set $\Xs^0:=\{1,\dots,N\}$.
\end{proposition}
\begin{proof}
To justify the statement, it is sufficient to consider the finite set $\Xs^0:=\{1,\dots,N\}$ together with the finite set of 1-arrows (ordered pairs) $(i,j)\in \Xs^1:=\Xs^0\times\Xs^0$, with source $j$ and target $i$ and note that $\Xs^1$ is naturally a groupoid (actually an equivalence relation with only one equivalence class) under the the usual composition $(i,j)\circ (j,k):=(i,k)$, for all $i,j,k\in \Xs^0$, with inverse given by $(i,j)^{-1}=(j,i)$, for all $i,j\in \Xs^0$ and partial identities $(j,j)$, for all $j\in \Xs^0$. The trivial Fell line-bundle $\Es:=\Xs\times\CC$ over the pair groupoid $\Xs$ is simply obtained by attaching a complex line $\Es_{(i,j)}:=\CC$ to each of the 1-arrows $(i,j)\in \Xs^1$. 
A section of such Fell line-bundle, being a function $x:\Xs^1\to \Es:=\bigcup_{(i,j)\in\Xs^1}\Es_{ij}$ such that $x^i_{j}:=x(i,j)\in \Es_{ij}$, for all $(i,j)\in \Xs^1$, is immediately seen to correspond to a complex square matrix $[x^i_j]$ with entries $x^i_j$, for all $i,j\in \Xs^0$. 
\end{proof}
An alternative way to construct the previous Fell line-bundle consists in considering the complex line $\CC$ as fiber over the space $\{(\bullet,\bullet)\}$, consisting of a unique loop $\xymatrix{\bullet \ar@(ur,dr)[]^{{}^{(\bullet,\bullet)}}}$ with source and target $\bullet$, and the $T$-pull-back $\Es:=T^\bullet(\CC)$ of such trivial one-point Fell bundle, via the unique functor $T:\Xs\to \{(\bullet,\bullet)\}$ that collapses every 1-arrow of the pair groupoid $\Xs$ to the unique loop $(\bullet,\bullet)$. Here is an intuitive picture of the Fell line-bundle $\Es$ (restricted to the base pair subgroupoid generated by the two points $1$ and $N$): 
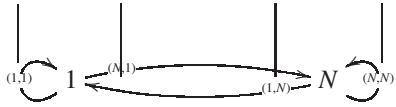
\begin{equation*}
\xymatrix{
& & & & 
\\
1 \ar@(dl,ul)|{^ {(1,1)}}
\ar@/^/[rrr]|(.2){^{(N,1)}} \ar@{-}@<+20pt>[u]&\ar@{-}@<+14pt>[u] & \ar@{-}@<-14pt>[u] 
& N \ar@(dr,ur)|{^{(N,N)}} \ar@/^/[lll]|(.2){^{(1,N)}} \ar@{-}@<-20pt>[u]}
\end{equation*}
Clearly for the family of continuous sections of $\Es:=T^\bullet(\CC)$ we have $\Gamma(\Xs;\Es)\simeq \MM_{N\times N}(\CC)$ and this construction can be applied in the same way, taking an arbitrary associative complex unital $*$-algebra $\As$ in place of $\CC$, obtaining the $*$-algebra 
$\Gamma(\Xs;T^\bullet(\As))\simeq \MM_{N\times N}(\As)\simeq \MM_{N\times N}(\CC)\otimes_\CC \As$ of $\As$-valued matrices. 

\medskip 

As a second step, we stress that there is no obstacle in generalizing the previous construction, starting with other finite groupoids, or even a finite involutive category, $\Xs$ in place of the previous pair groupoid of the set with $N$ points. 
\begin{proposition}
Given a finite involutive category $(\Xs,\circ,*)$ and a complex unital $*$-algebra $(\As,\cdot,-)$, the family $\Gamma(\Xs;T^\bullet(\As))$ of sections of the Fell bundle $T^\bullet(\As)$ over $\Xs$, obtained by $T$-pull-back of the fiber $\As$ via the terminal functor $T$ from $\Xs$ to the 1-loop space $\{(\bullet,\bullet)\}$, is a $*$-algebra with the operations: 
\begin{gather*}
(\sigma\circ\rho)_z:=\sum_{x\circ y=z}\sigma_x\cdot \rho_y, \quad \forall \sigma,\rho\in \Gamma(\Xs;T^\bullet(\As)), \quad \forall z\in \Xs, 
\\ 
(\sigma^*)_z:=\cj{\sigma_{z^*}}, \quad \forall \sigma\in \Gamma(\Xs;T^\bullet(\As)), \quad \forall z\in \Xs. 
\end{gather*}
\end{proposition}
The resulting $*$-algebra of sections $\Gamma(\Xs; \Es)$ is just the convolution algebra of the groupoid (respectively of the finite involutive category) $\Xs$ and it is usually denoted by $\CC[\Xs]$. In the case of the pair groupoid of a set of $N$ elements, the previous operations reduce exactly to the usual row-by-column multiplication and transpose conjugate involution of matrices in $\MM_{N\times N}(\As)$. Hence we just proved that: 

\textit{the $*$-algebra of matrices is just a special case of the convolution $*$-algebra of a finite $*$-category $\Xs$.}

\medskip 

Finally, as the last step, we examine what happens when, in place of a finite involutive 1-category, we allow a strict finite globular (fully involutive) $n$-category (with or without non-commutative exchange). 
\begin{theorem}\label{th: convo} 
Given a finite strict globular $n$-category $(\Xs,\circ_0,\dots,\circ_{n-1})$ (with usual exchange law or with non-commutative exchange) and an associative unital algebra $(\As,\cdot)$, the family $\Gamma(\Xs;T^\bullet(\As))$ of sections of the bundle $T^\bullet(\As)$ over $\Xs$, obtained by $T$-pull-back of the fiber $\As$ via the terminal functor $T$ from $\Xs$ to the strict globular $n$-category with only one $n$-arrow\footnote{This is the terminal $n$-category in which all the operations coincide and that satisfies the usual exchange law.} is a unital associative algebra with respect to each one of the following convolution operations $\hatcirc_p$, for $p=0,\dots,n-1$: 
\begin{gather*}
(\sigma\hatcirc_p\rho)_z:=\sum_{x\circ_p y=z}\sigma_x\cdot \rho_y, \quad \forall \sigma,\rho\in \Gamma(\Xs;T^\bullet(\As)), \quad \forall z\in \Xs. 
\end{gather*}
The bundle $T^\bullet(\As)=\As\times\Xs$ over $\Xs$ embeds into $\Gamma(T^\bullet(\As))$ via the fiberwise linear maps: 
\begin{equation*}
a_x\mapsto a \cdot (\delta^x)_y \quad \text{where $a\in \As$, $x\in \Xs$ and for all $x,y\in \Xs$} \quad  
(\delta^x)_y:=\begin{cases}
1_\As, \quad \text{if $x=y$}
\\
0_\As \quad \text{if $x\neq y$} 
\end{cases}
\end{equation*}
and becomes a strict globular $n$-category with the restriction of the convolution operations $\hatcirc_0,\dots,\hatcirc_{n-1}$. 
Whenever the algebra $\As$ fails to be commutative, the resulting $n$-category $(T^\bullet(\As),\hatcirc_0,\dots,\hatcirc_{n-1})$ does not satisfy the usual exchange law (even when $\Xs$ does), but satisfies the non-commutative exchange. 
\end{theorem} 
\begin{proof}
With the notations introduced above, the results amounts to a direct algebraic verification of associativity unitality and non-commutative exchange for the convolution operations $\hatcirc_p,\dots,\hatcirc_{n-1}$. 
The fact that non-commutative exchange is necessary whenever $\As$ is not abelian follows from the Eckmann-Hilton collapse and the fact that for $x\in \Xs^p$ and $p<n$, the fibers $(T^\bullet(\As)_x,\circ_p)$ are isomorphic to $(\As,\cdot)$ as unital associative algebras. 
\end{proof}

\medskip 

We would like to spend a few words to investigate here those algebraic properties making $\As$ eligible as a ``system of coefficients'' for a \emph{convolution $n$-category} $\Es:=T^\bullet(\As) \subset \MM_\Xs(\As):=\Gamma(\Xs;T^\bullet(\As))$ over an $n$-category $\Xs$ with usual, or with  non-commutative exchange. 

\medskip 

First of all we notice that for any convolution $n$-category $\Es\subset \MM_\Xs(\As)$, the fibers $\Es_\bullet$ over the $n$-identities of an object $\bullet\in \Xs^0$ are isomorphic to $\As$ and hence we can infer the necessary properties of $\As$ from the study of these fibers.
Secondly, for any $n$-categorical bundle $(\Es,\pi,\Xs)$, the fibers $\Es_\bullet$, for $\bullet\in \Xs^0$ are themselves $n$-categories with all the sets of $\circ_p$-identities of cardinality one, for all $p=0,\dots,n-1$. 

\begin{proposition}
Let $(\Es,\circ_0,\dots,\circ_{n-1})$ be a $n$-categorical bundle with non-commutative exchange over the $n$-category $\Xs$. 
For all $p=0,\dots,n-1$, for all $\bullet\in \Xs^0$, the fibers $(\Es_\bullet,\circ_p)$ are a family of (possibly non-commutative) monoids with a common identity (i.e.~such that $\Es_\bullet^0=\Es_\bullet^1=\cdots=\Es_\bullet^{n-1}$). 
\end{proposition}

By remark~\ref{rem: cat-b}, the previous proposition can be directly applied to the case of an $n$-category $\Cs$ yielding conditions on the $n$-diagonal home-sets ${}_{n-1}\Cs^n_{\bullet\bullet}$.  

\begin{remark}
Recall that when $\Es$ is an $n$-category with the usual exchange property, the Eckmann-Hilton collapse induce a strong trivialization, further imposing the coincidence of all the binary operations and their commutativity. 
As a consequence of the previous proposition, if $\As$ is a monoid with respect to $n$-operations, then $\As$ can be taken as a set of coefficients for a convolution $n$-category with non-commutative exchange if and only if, for all $p=0,\dots,n-1$, all the $\circ_p$-identities of the monoids coincide. Moreover, in that case, if $\As$ is a commutative monoid, then it can be taken as a set of coefficients for a convolution $n$-category. 
In particular this explains why we could immediately obtain examples of convolution $n$-categories $\Es$ over an $n$-category $\Xs$ with non-commutative exchange with coefficients in a single algebra (monoid) $\As$, since in this case all the operations in the monoid $\As$ coincide $\circ_0=\cdots=\circ_{n-1}$ and so do their identities.
\xqed{\lrcorner}
\end{remark}

\medskip 

We proceed now to examine what happens when one attempts to define involutions on the convolution $n$-category $\Es\subset\MM_\Xs(\As)$ over an involutive $n$-category $\Xs$ and which conditions must be imposed on the system of coefficients $\As$ in order to obtain such involutions on $\Es$. 

When the base category $\Xs$ has an involution that is contravariant with respect to all the compositions, we can immediately extend theorem~\ref{th: convo}, taking as a system of coefficients an involutive algebra $\As$. 

\begin{proposition} 
For a strict globular $n$-category $(\Xs,\circ_0,\dots,\circ_{n-1},*_\alpha)$ equipped with an \hbox{$\alpha$-involution}, with $\alpha=\{0,\dots,n-1\}$, and a complex unital associative $*$-algebra $(\As,\cdot,*_\As)$, the map 
\begin{gather} \label{eq: invo-cat}
(\sigma^{\hat{*}})_z:=(\sigma_{z^{*_\alpha}})^{*_\As}, \quad \forall \sigma\in \Gamma(T^\bullet(\As)), \quad 
\forall z\in \Xs  
\end{gather}
becomes an involution $\hat{*}$ for all the unital associative algebras $(\Gamma(T^\bullet(\As)),\hatcirc_p)$, for all $p=0,\dots, n-1$ and $(T^\bullet(\As),\hatcirc_0,\dots,\hatcirc_{n-1},\hat{*}_\alpha)$ is a partially involutive $n$-category with an $\alpha$-contravariant involution. 
\end{proposition} 

\begin{remark}\label{rem: obs} 
If the involutive unital associative algebra $\As$ is commutative, and the strict globular \hbox{$n$-category} $\Xs$ is $\Lambda$-involutive, formula~\ref{eq: invo-cat} can be used to define $\hat{*}_\alpha$ involutions on $\Gamma(T^\bullet(\As))$, for all $\alpha\in \Lambda$ and hence $T^\bullet(\As)\subset \Gamma(T^\bullet(\As))$ becomes a $\Lambda$-involutive category as well. 
Unfortunately, whenever $\As$ is not abelian, the antimultiplicativity of $*_\As$ conflicts with the covariance/contravariance properties required to define $\alpha$-involutions on $T^\bullet(\As)$ unless $\alpha=\{0,\dots,n-1\}$ (as already stated in the previous proposition). Hence, in order to construct examples of fully involutive strict globular $n$-categories with non-commutative exchange, as ``convolution algebroids'', we need a more elaborate choice of ``involutive algebra of coefficients'' $\As$. 
\xqed{\lrcorner}
\end{remark}

If $(\Es,\circ_0,\dots,\circ_{n-1},*_\alpha)$ is an $n$-category with non-commutative exchange that is $*_\alpha$-involutive, for $\alpha\subset \NN$, the \hbox{$n$-diagonal} home-set $\Es_\bullet$, corresponding to the object $\bullet\in \Es^0$ (that we already know to be a monoid with respect to each one of the operations 
$\circ_p$, $p=0,\dots,n-1$, sharing the same identity) is equipped with an involution $*_\alpha$ maintaining the same covariance/contravariance properties with respect to the monoidal compositions. This introduces further complications in the study of the class of ``systems of coefficients'' for a convolution bundle over a (partially) involutive $n$-category $\Xs$ with non-commutative exchange, as explained in the following result. 

\begin{proposition}\label{prop: obs}
Let $(\Xs,\circ_0,\dots,\circ_{n-1},\Lambda)$ be a (partially) involutive $n$-category, with non-commutative exchange, equipped with a family $\Lambda$ of $\alpha$-involutions $*_\alpha\in \Lambda$. 
Let $(\As,\cdot_0,\dots,\cdot_r,\dag_0,\dots,\dag_s)$ be such that, for all $k=0,\dots,r$, the $(\As,\cdot_k)$ are monoids with a common identity, and let it be equipped with a family of involutions $\dag_j$ for all $j=0,\dots,s$. 
The algebraic structure $\As$ can be a ``system of coefficients'' for a convolution (partially) involutive $n$-category 
$\Es$ over $\Xs$ if and only if it is possible to find a function 
$f:\{(\circ_p,*_\alpha) \ | \ p=0,\dots,n-1, *_\alpha\in \Lambda\}\to \{(\cdot_k,\dag_j)\ | \ k=0,\dots,r,j=0,\dots,s\}$ that is preserving the covariance properties of the pairs.  
\end{proposition}

As a consequence, we see immediately that commutative monoids do not pose any further problem as ``systems of coefficients'' and that, even when the non-commutative exchange is assumed, non-commutative involutive monoids $(\As,\cdot,\dag)$ can be ``systems of coefficients'' only when all the involutions in the base category $\Xs$ have (with all the compositions) the same covariance of the pair $(\cdot,\dag)$. 

In order to exploit convolution $n$-categories $\Es$ as a source of non-trivial examples of fully involutive \hbox{$n$-categories} with non-commutative $n$-diagonal home-sets $\Es_\bullet$ (and hence necessarily with non-commutative exchange), we must utilize a more ``sophisticated'' system of coefficients $\As$. 

\medskip 

Motivated from the previous discussion, we are naturally induced to propose the following notion: 
\begin{definition}
A \emph{hyper-C*-algebra}
\footnote{We warn the reader that there is a conflict of terminology with the usage of the term ``hyper-algebra'' in the area of universal algebra, where an ``hyperalgebra'' (also called multialgebra or polyalgebra) means an algebraic structure with set-valued operations. \label{foo: norms2}
\emph{Important note}: As already mentioned in footnote~\ref{foo: norms}, in the infinite dimensional case, the requirement of simultaneuous completeness for all the norms will be in general too restrictive and a more appropriate choice is to ask the completeness in the uniformity induced by the family of norms. We will return to these infinite dimensional technical issues elsewhere. 
}  
$(\As,\circ_0,\dots,\circ_{n-1},*_0,\dots,*_{n-1})$ is a complete topo-linear space $\As$ equipped with pairs of multiplication/involution $(\circ_k,*_k)$, for $k=0,\dots n-1$, each inducing on $\As$ a C*-algebra structure, via a necessarily unique C*-norm $\| \cdot\|_k$, compatible with the given fixed topology. 
\end{definition}

In the same vein, we might introduce the notions of \emph{hyper-monoid} and \emph{hyper-involutive-monoid} to describe the more general abstract algebraic structures naturally arising from (involutive) convolutions of $n$-categories and (partially) involutive $n$-categories (with non-commutative exchange), but we will not elaborate on this any further. 

\begin{proposition} 
Given a unital commutative C*-algebra $\As$ and a finite globular (cubical) higher (fully) involutive $n$-category $\Xs$, the $\Xs$-convolution $*$-algebra $\MM_\Xs(\As):=\Gamma(T^\bullet(\As))$ is a hyper C*-algebra with the operations of $\circ_q$-convolution and $*_q$-involutions, for $q=0,\dots, n-1$. 
\end{proposition}
\begin{proof}
To define the norm, take the direct sum $\bigoplus_{[\Xs]}\As$, where $[\Xs]$ denotes the set of all the globular $n$-cells of $\Xs$, with the usual norm; let $\MM_\Xs(\As)$ act on such direct sum in the usual way by each one of the convolutions and consider the different operator norms coming from each one of such compositions. 
\end{proof}

A class of extremely interesting examples of finite hyper C*-algebras, that are not naturally obtained as convolution hyper C*-algebras of strict globular (fully involutive) higher categories, is constituted by hypermatrices indexed by full-depth $n$-categories. 

\begin{definition}
A \emph{hypermatrix of depth-n} is a multimatrix $[x_{i_1\dots i_n}^{j_1\dots j_n}]\in \MM_{N^2_1\dots N^2_n}(\CC)$ having indices $i_k,j_k=1,\dots N_k$, for all $k=1,\dots,n$. 
\end{definition}

\begin{theorem}
The family $\MM_\Xs(\CC)$ of $\CC$-valued hypermatrices of depth-$n$ is a hyper C*-algebra. 
\end{theorem}
\begin{proof}

On $\MM_{N^2_1\dots N^2_n}(\CC)$ there are $2^n$ different multiplications acting at every level either as convolution or as Schur product: 
$[x^{i_1\dots i_k\dots i_n}_{j_1\dots j_k\dots j_n}]\bullet_{\gamma} 
[y^{i'_1\dots i'_k\dots i'_n}_{j'_1\dots j'_k\dots j'_n}]:=
[\sum_{k\in \gamma}\sum_{o_k=1}^{N_k} x^{i_1\dots i_k\dots i_n}_{j_1\dots o_k\dots j_n}\ 
y^{i_1\dots o_k\dots i_n}_{j_1\dots j_k\dots j_n}]$, 
where $\gamma\subset\{1,\dots,n\}$ is the set of contracting indices. 

There are $2^n$ involutions taking the conjugate of all the entries and, at every level, either the transpose or the identity: 
$[x^{i_1\dots i_k\dots i_n}_{j_1\dots j_k\dots j_n}]^{\star_\gamma}:=
[\cj{x}^{i_1\dots j_{k_1}\dots j_{k_m}\dots i_n}_{j_1\dots i_{k_1}\dots i_{k_m}\dots j_n}]$,  
for all $\gamma:=\{k_1,\dots,k_m\}\subset \{1,\dots,n\}$. 

There are $2^n$ C*-norms taking either the operator norm or the maximum norm at every level. Using the natural isomorphism 
$\MM_{N_1^2\dots N_n^2}(\CC)\simeq \MM_{N^2_1}(\CC)\otimes_\CC \cdots \otimes_\CC \MM_{N_n^2}(\CC)$, these norms can be defined as: 
$\|[x^{i_1}_{j_1}]\otimes \cdots \otimes [x^{i_n}_{j_n}]\|_\gamma:=
\prod_{k\in\gamma} \|[x^{i_k}_{j_k}]\|\cdot \prod_{k'\notin\gamma}\|[x^{i_{k'}}_{j_{k'}}]\|_\infty$, $\forall \gamma\subset\{1,\dots, n\}$, 
where $\|[x^{i_k}_{j_k}]\|$ is the C*-norm on $\MM_{N_k}(\CC)$ and $\|[x^{i_k}_{j_k}]\|_\infty:=\max_{i,j}|x^i_j|$. 

With such ingredients $(\MM_{N^2_1\dots N^2_n}(\CC),\bullet_\gamma,\star_\gamma, \|\ \|_\gamma, \gamma\subset\{1,\dots,n\})$ 
is a hyper C*-algebra. 
\end{proof}

\begin{remark}
If in place of the complex numbers $\CC$ we consider an arbitrary \emph{non-commutative} \hbox{C*-algebra} $\As$, the family of $\As$-valued hypermatrices $\MM_\Xs(\As)$ is still a hyper-C*-algebra. 
\xqed{\lrcorner}
\end{remark}

\textit{Can we see all the $2^n$ operations in the hypermatrices $\MM_{N_1^2,\dots,N^2_n}(\As)$ as convolutions of some $n$-category?} 

\medskip 

Hypermatrices $\MM_\Xs(\CC)$ obtained via convolution of globular $n$-categories $\Xs$ have only $n$ compositions. 
The same is actually true for convolutions of cubical $n$-categories (see~\cite{BCM}).  

The C*-algebra $\MM_{\Xs_1}(\CC)\otimes \MM_{\Xs_2}(\CC)$ coincides with the convolution C*-algebra 
$\CC[\Xs_1\times\Xs_2]=\MM_{\Xs_1\times\Xs_2}(\CC)$ of the Cartesian product $\Xs_1\times \Xs_2$ of the finite pair groupoids, but the product of $n$ finite pair groupoids $\Xs:=\{1,\dots,N_1\}^2\times\cdots\times\{1,\dots,N_n\}^2$ has a richer structure of ``full-depth  $n$-tuple category'' (via compositions on the ``oriented borders''), as we described in section~\ref{sec: full-depth}.   
Hence there are $2^n$ such possible compositions on $\Xs$ and we can recover $\MM_{N_1^2,\dots,N_n^2}(\CC)$ as a convolution hyper C*-algebra of the ``full depth $n$-tuple category'' $\Xs$. 

\begin{theorem}
Let $\As$ be a commutative C*-algebra.  
The hyper C*-algebra $\MM_\Xs(\As)$ of $\As$-valued hypermatrices, indexed by the Cartesian product $\Xs$ of $n$ finite pair groupoids $\Xs:=\Xs_1\times\dots\times\Xs_n$, is the convolution hyper C*-algebra of the fully involutive full-depth $n$-category $\Xs:=\Xs_1\times\dots\times\Xs_n$. 
\end{theorem}

\begin{remark}
When $\As$ is \textit{commutative}, the hyper C*-algebra $\MM_\Xs(\As)$ is the envelope of the fully involutive full-depth Fell-bundle $\Es:=T^\bullet(\As)$ with the usual exchange property in place. 

Unfortunately, when $\As$ is a non-commutative C*-algebra, the hyper C*-algebra $\MM_\Xs(\As)$ cannot be obtained as a convolution enveloping algebra of a Fell bundle, even if the non-commutative exchange property is assumed!\footnote{
The reason is again that the covariance conditions imposed by proposition~\ref{prop: obs} cannot in general be satisfied.} 
\xqed{\lrcorner}
\end{remark}

If we utilize hyper C*-algebras $\As$ as systems of coefficients, we can finally obtain explicit examples of fully involutive convolution globular $n$-categories $\Es$. 

\begin{theorem}\label{th: obs}
Let $\As$ be a hyper C*-algebra with respect to $n$ pairs of product/involution $(\cdot_k,\dagger_k)$, for $k=0,\dots,n-1$. 
Let $(\Xs,\circ_0,\dots,\circ_{n-1}, *_0,\dots,*_{n-1})$ be a fully involutive globular $n$-category (with commutative or non-commutative exchange). 
The convolution $n$-bundle $\Es\subset \MM_\Xs(\As)$ is now a fully involutive globular $n$-category necessarily with non-commutative exchange, as soon as one of the products in $\As$ is non-commutative.   
\end{theorem} 

\section{Outlook and Applications} \label{sec: app}

In this final section, we informally venture into uncharted territory, trying to suggest some intriguing connections between higher categories with non-commutative exchange and the study of ``morphisms'' of ``non-commutative spaces'' (and hence interactions of quantum systems~\cite{B}). 
We also provide a detailed list of several further interesting lines of development for the study of the categorical structures introduced in this paper. 

\subsection{Morphisms of Non-commutative Spaces}

Several people have already advocated the existence of an interplay between (higher) category theory and quantization (and hence non-commutative algebras) notably: J.E.Roberts, C.Isham, J.Baez, B.Coecke, N.Landsman, \dots, but the leading ideas for us here are mainly coming from: 
\begin{itemize}
\item 
L.Crane / R.Feynman~\cite{Cr}: in the suggestion to see quantization (non-commutativity) as a categorification effect (due to different paths between points), 
\item 
A.Connes / W.Heisenberg~\cite[chapter~1, section~1]{Co}: in their way to look at algebras of non-commutative spaces, such as matrix algebras, as convolution algebras of a category (groupoid).  
\end{itemize} 
These basic ideologies merge and are somehow strongly supported from our already mentioned results, theorem~\ref{th: bcl}, on the spectral structure of commutative full C*-categories in terms of spaceoids that  seem to indicate a direct route to a general spectral reconstruction of non-commutative C*-algebras as algebras of ``sections'' of complex line-bundles with a suitable categorical base space:\footnote{Significant work is ongoing on this topic: 

P.Bertozzini, R.Conti, N.Pitiwan ``Non-commutative Gel'fand-Na\u\i mark Duality'' (preprint in preparation); 

P.Bertozzini, R.Conti, N.Pitiwan ``Discrete Non-commutative Gel'fand-Na\u\i mark Duality'' (accepted for publication~\cite{BCP}); 

P.Bertozzini ``Non-commutative Gel'fand-Na\u\i mark Duality'' Mahidol International College, 28 March 2018 (seminar); 

P.Bertozzini ``Spectra of Non-commutative Unital C*-algebras'' Thammasat University, 14 August 2018 (seminar). 
}

\medskip 

\emph{Spectral Conjecture:} 
\textit{there is a spectral theory of non-commutative C*-algebras in terms of families of Fell complex line-bundles over involutive categories.} 
\begin{center}
\textit{Quantum space} \ 	
$\simeq$ \ \textit{spectrum of C*-algebra} \ 
$\simeq$ \ \textit{Fell line-bundle over an inverse involutive category}  
\xqed{\lrcorner}
\end{center} 
As it is stated above, without further details on the precise nature of the functors involved in such a non-commutative generalization of 
Gel'fand-Na\u\i mark duality, the conjecture is ``not even wrong'',\footnote{For example, even in finite dimensional situations, there are ``gauge redundancies'' that allow to express in different ways the same C*-algebra as convolution algebra of different spaceoids: an algebra of linear operators on a finite dimensional vector space is isomorphic in many different ways, one for every alternative choice of an orthonormal base, to an algebra of square matrices.} 
anyway this is not a serious issue for us here, because the conjecture surely holds for some sufficiently many interesting finite dimensional cases (such as matrix algebras) and our only goal for now is to make use of the spectral conjecture, in those ``safe cases'', as a motivation to propose an alternative way to look at the notion of morphism of non-commutative spaces. 

\medskip 

The usual Gel'fand-Na\u\i mark duality, when recasted in the language of theorem~\ref{th: bcl}, essentially says: 
\begin{align*}
\text{classical space} \ X 	&\simeq \text{spectrum of abelian C*-algebra}\ C(X;\CC) 
\\ 
									&\simeq \text{trivial line bundle $X\times \CC$ over space} \ X 
\\
									&\simeq \text{Fell line-bundle over the space $\Delta_X$ of ``loops'' of} \ X, 
\\
\\
\text{Abelian C*-algebra} \ C(X) 	&\simeq \text{algebra $\Gamma(X;X\times \CC)$ of sections  of} \ X\times \CC   
\\ 
									&\simeq \text{convolution algebra $\Gamma(\Delta_X;\Delta_X\times\CC)$.}  
\end{align*}

For the spectrum of a finite discrete space $X$ consisting of $N$ points, we have the following ``transitions'': 

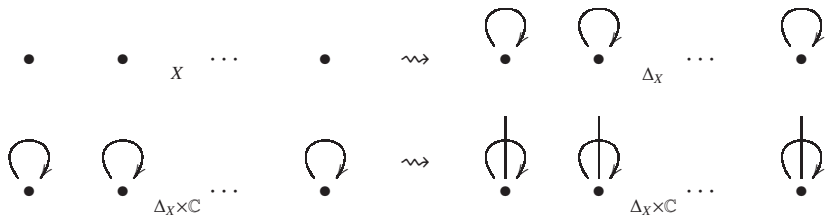
\medskip 

\begin{gather*}
\xymatrix{
\bullet \ar@{}[rrr]_{X}& \bullet & \cdots & \bullet 
}
\qquad \rightsquigarrow \qquad 
\vcenter{
\xymatrix{
 & 
 & & 
\\
\bullet  \ar@(ul,ur) \ar@{}[rrr]_{\Delta_X}& \bullet  \ar@(ul,ur)& \cdots & \bullet  \ar@(ul,ur)
}
}
\\
\vcenter{
\xymatrix{
 & 
 & & 
\\
\bullet  \ar@(ul,ur) \ar@{}[rrr]_{\Delta_X}& \bullet  \ar@(ul,ur)& \cdots & \bullet  \ar@(ul,ur)
}
}
\qquad \rightsquigarrow \qquad 
\vcenter{ 
\xymatrix{
\ar@{-}[d]& \ar@{-}[d]& & \ar@{-}[d]
\\
\bullet  \ar@(ul,ur) \ar@{}[rrr]_{\Delta_X\times \CC}& \bullet  \ar@(ul,ur)& \cdots & \bullet  \ar@(ul,ur) 
}
}
\end{gather*}
where in the first line, the discrete set $X$ corresponds to the discrete groupoid $\Delta_X$ of identity loops, 
and in the second line, the discrete groupoid $\Delta_X$ corresponds to the trivial Fell line-bundle $\Delta_X\times\CC$ over $\Delta_X$. 

\medskip 

For the case of morphisms between classical spaces, the first transition entails: 
\begin{align*}
\text{morphism of classical spaces} \ X, Y 		&\simeq \text{map / relation / 1-quiver} :X\to Y 
\\
																&\simeq \text{level-2 relation} :\Delta_X\to\Delta_Y, 
\end{align*}
\begin{equation*}
\xymatrix{
x \ar@{|->}[r]& y
}
\qquad \rightsquigarrow \qquad 
\xymatrix{
x\ar@(ru,rd)&\imp & y\ar@(ul,dl)
}
\qquad  x\in X, \ y\in Y.  
\end{equation*}
The transition from $\Delta_X,\Delta_Y$ to their associated Fell line-bundles $\Delta_X\times \CC, \Delta_Y\times \CC$, (attaching a complex fiber to each 1-loop) seems to further suggest that also each 1-arrow $x\mapsto y$ in the morphism from $X$ to $Y$ should have a complex fiber attached. 

\medskip 

Dually, for a relation $R\subset X\times Y$ (1-quiver) with reciprocal $R^*\subset Y\times X$, the ``convolution algebra'' $\As$ of the trivial Fell line-bundle with base $\Delta_X\cup R\cup R^*\cup \Delta_Y$ is given by a linking C*-algebra 
\begin{equation*}
\As=\begin{bmatrix}
C(X) 				& \Gamma(R^*\times\CC)
\\ 
\Gamma(R\times\CC)		& C(Y)
\end{bmatrix}
\end{equation*}
that contains on the diagonal the C*-algebras $C(X)$, $C(Y)$, and off-diagonal the bimodule $\Gamma(R,R\times\CC)$ and its contragredient $\Gamma(R^*; R^*\times\CC)$. Hence, in a quite familiar way, the morphisms from $X$ to $Y$ are dually given by (Hilbert C*) bimodules, over the commutative C*-algebras $C(Y)$ and $C(X)$. 

\bigskip 

When we pass to the study of (finite discrete) non-commutative spaces, we see that the appearance of level-2 relations and 2-cells, becomes  unavoidable and much more intriguing because, in light of the previous spectral conjecture, we have: 
\begin{align*}
\text{quantum space}  &\simeq \text{spectrum of non-commutative C*-algebra} 
\\
&\simeq \text{space of points with ``linearized relations''} 
\\
&\simeq \text{Fell line-bundle over a 1-quiver $Q^1$,} 
\\
\\ 
\text{algebra of functions on $Q^1$} &\simeq \text{``convolution'' algebra of $Q^1$.} 
\end{align*}
\medskip 

As a consequence, proceeding as before, we claim that: at the ``spectral level'' a morphism between two (finite discrete) quantum spaces $Q^1_X$ ad $Q^1_Y$ is a 2-quiver $Q^2$ with 2-cells like 
\begin{equation*}
\vcenter{\xymatrix{
x_1 \ar@/_0.5cm/[d]^f \ar@{.>}[r] \ar@{:>}[dr] & y_1 \ar@/^0.5cm/[d]_{g} 
\\
x_2 \ar@{.>}[r] & y_2
}} \qquad f\in Q^1_X, \quad g\in Q^1_Y, 
\end{equation*}
and so, at the ``dual level'', a morphism of quantum spaces is a ``level-2 bimodule'' inside the convolution depth-2 hyper C*-algebra $\Gamma(Q^2)$ of the involutive 2-category generated by the morphism 2-quiver $Q^2$.  

The possible relevance of higher C*-categories and hyper-C*-algebras to formally describe, at least at the topological level, these situations should be self-evident and we plan to address such issues in the future. 

\subsection{Other Related Topics} \label{sec: o}

Among the several lines of development directly related to the material introduced in this paper, we mention here only a few that are either already under study (and partially covered in other works) or that we deem particularly interesting or intriguing. 

\medskip 

\hspace{-0.30cm}{$\bullet$}
Involutive double categories (with usual exchange property) and their relationship with involutive \hbox{2-categories} are extensively studied in the preprint~\cite{BCM}. The study of involutions for general \hbox{$n$-tuple} cubical categories and versions of the non-commutative exchange for the cubical case should be the next immediate goal also in view of the inevitable appearance of cubical structures both in the study of hypermatrices and morphisms of non-commutative spaces. 
The possibility of even further ``exotic'' type of $n$-cells can be considered. In~\cite{BP} we had a first look at the case of ``hybrid'' globular 2-categories. 

\medskip 

\hspace{-0.30cm}{$\bullet$}
Strict (involutive) $\omega$-categories with quantum exchange and strict $\omega$-C*-categories\footnote{
A small technical obstacle in the definition of $\omega$-C*-categories, the triviality of the sets of $\omega$-arrows sharing the same $\omega$-cell, can be easily avoided (see for example P.Bejrakarbum's thesis~\cite{Bej}).} 
are immediately obtained omitting the finite bound on the number of binary operations of composition and the number of involutions involved in the definitions. 

Weak (involutive) higher categories with quantum exchange and weak higher C*-categories are of course a much more involved and complex area of investigation. 
We are currently formally developing such notions, starting from those more ``algebraic'' definitions of J.Penon, M.Batanin, T.Leinster (see~\cite{CL,Lsur,L} and the more recent works by C.Kachour~\cite{K1,K2}) that, being less motivated by classical homotopy theory, are more suitable for applications to operator theory.\footnote{P.Bejrakarbum, P.Bertozzini ``Weak Involutive Higher C*-Categories'' (work in progress).} 
Weak involutive categories in Penon's approach are studied in~\cite{BB}. 
Natural examples of weak higher C*-categories can be found in the study of higher categories of ``bimodules'' over strict higher \hbox{C*-categories} in the same way as the Morita weak \hbox{2-C*-category} originates from imprimitivity bimodules over C*-algebras.  

In this work we have only touched on some variants of higher categories, but we also have quite strong interest in the investigation of other possible higher involutive and C*-algebraic structures in the wider contexts of polycategories, multicategories (operads) and their vertically categorified counterparts.\footnote{P.Bertozzini ``C*-polycategories'' (work in progress).} 

\medskip 

\hspace{-0.30cm}{$\bullet$}
In the present paper, we opted for a compact treatment of $n$-C*-categories as ``involutive partial \hbox{$n$-mon}\-oids'' i.e.~via binary partial operations and involutions defined on $n$-cells. This can be too restrictive for the study of weak higher categories. A more immediate concern (also for the case of strict \hbox{C*-categories}) is that, in the same spirit, we defined linear structures and norms only at the level of $n$-cells. It is actually possible to provide a definition of ``iterated'' (quantum) \hbox{$n$-C*-categories} (and ``iterated'' $n$-Fell bundles), where different linear structures and different norms are introduced at each depth-level. This kind of approach has been already briefly presented (only for the case of usual exchange) in previous works~\cite{BCL3} and in the near future we plan to further comment on this point, clarifying the link between the two definitions.  

\medskip 

\hspace{-0.30cm}{$\bullet$}
Kre\u\i n versions of higher C*-categories, as a vertical categorification of the Kre\u\i n \hbox{C*-categories} already defined in~\cite{BRu}, can be produced and will be treated elsewhere. 

\bigskip 

\hspace{0.3cm} Among the many issues that remain to be explored, once a viable theory of higher C*-categories is in place, we mention: 

\medskip 

\hspace{-0.30cm}{$\bullet$}
Developing a representation theory of (quantum) higher C*-categories and higher \hbox{C*-algebras} (higher Gel'fand-Na\u\i mark representation theorems); Hilbert higher bimodules (higher Hilbert spaces) and higher Morita theory (higher $K$-theory); higher spectral theory via higher $n$-Fell line-bundles \dots\ higher Gel'fand-Na\u\i mark duality and more generally higher functional analysis. 

\medskip 

\hspace{-0.30cm}{$\bullet$}
In light of the already known connections between Fell bundles one one side and product systems~\cite{A} on the other, one would like to see if also higher categorifications of such notions retain similar connections. 

\medskip 

\hspace{-0.30cm}{$\bullet$}
Further extension of the investigation on the role of higher categories in the study of morphisms for non-commutative geometries and the study of ``higher non-commutative geometries'' as suitable ``spectral triples'' on higher C*-categories (as already suggested in~\cite{BCL3}). 

\medskip 

\hspace{-0.30cm}{$\bullet$}
In view of the current general interest in homotopy type theory and higher $\infty$-groupoids in the foundations of
mathematics~\cite{UFP} and in philosophy of mathematics~\cite{Ro} and the attempts to reconsider in this (categorical) light also the famous Hilbert sixth problem on the possible mathematical axiomatic foundations of physics (see U.Schreiber~\cite{Sc,Sc2} and A.Rodin~\cite{Ro2}), it is quite natural to speculate if the basic quantum nature of physics will give a more prominent role to quantum $\infty$-C*-categories for its foundations. 

\medskip 

\hspace{-0.30cm}{$\bullet$}
The usage of higher C*-categories for the formalization of relational Rovelli's quantum theory, and more generally ``quantum cybernetics'', has been already touched in~\cite{B} and it is one of the main motivations for the development of such techniques. 

\medskip 

\hspace{-0.30cm}{$\bullet$}
A possible definition of non-commutative homotopy theory.  

\medskip 

\hspace{-0.30cm}{$\bullet$}
The development of non-commutative ``higher measure theory'' and higher categorical modular theory (vertically categorifying the results in P.Ghez-R.Lima-J.E.Roberts~\cite[section~3]{GLR}) is one of our most immediate priorities also in view of the strong motivations coming from  proposals in ``modular algebraic quantum 
gravity''~\cite{BCL1, modular, B, Ra1, Ra2}. 

\medskip 

\hspace{-0.30cm}{$\bullet$}
The study of how non-commutative exchange will affect the usual notions of (higher) topoi, sites and Grothendieck categories (especially in situations where Cartesian closure is replaced by monoidal closure and suitable involutions/dualities are introduced). Possible links with the new notions of gleaves developed by F.Flori-F.Fritz~\cite{FF} are quite intriguing.  

\bigskip 

We are only taking the first steps into a vast landscape of vertically categorified functional analysis. 

\bigskip 

\emph{Notes and Acknowledgments:} 
We thank Ross Street and Ezio Vasselli for providing some additional references and some corrections. 

\smallskip

We also thank the anonymous referee for a careful reading of the original preprint and for suggesting improvements in the introduction (motivations), in section~\ref{rem: conjugates} and in the bibliography. 

\smallskip 

P.B.~thanks Starbucks Coffee at the $1^{\text{st}}$ floor of Emporium Tower and Jasmine Tower, in Sukhumvit, where he spent most of the time dedicated to this research project. 

He also would like to thank the following colleagues and friends for useful discussions and/or for the possibility to organize several seminar talks and attend conferences where preliminary versions of part of the material exposed in this paper have been presented: 
Ross Street, Michael Batanin and Tobias Fritz (``Category Theory 2013'' conference in Macquarie University - July 2013); Philipp H\"ohn, Matti Raasakka, Richard Kostecki and Andreas D\"oring (``Loop 13'' conference in Perimeter Institute - July 2013); David Evans (Bangkok meeting - 2013); Matilde Marcolli and Jan-Jitse Vanselaar (CalTech seminar - April 2014); Przemo Kranz, Luca Bombelli and Tommy Naugle (Olemiss seminars - April 2014); the late David Finkelstein (Atlanta meeting - May 2014); Jonathan Engle, Matthew Hogan and Kathie Broadhead (Florida Atlantic University seminar - May 2014); Rachel Martins, Pedro Resende and Roger Picken (Istituto Superior T\'ecnico Lisboa visit and seminar - May 2014); Thierry Masson, Carlo Rovelli, Igor Kanatchikov, Vaclav Zatloukal and Michel Wright (``FFP 14'' conference in Marseille - July 2014); 
Piotr Hajac and Tomasz Maszczyk (``Noncommutative Geometry and Index Theory'' conference at the Banach Center in Warsaw - November 2016); Thierry Masson and Roland Triay (CIMPA Vietnam School ``Noncommutative Geometry and Applications to Quantum Physics'' Quy Nhon - July 2017). 

\medskip 

This research is partially supported by Thammasat University research grant n.~2/15/2556: ``Categorical Non-commutative Geometry''. 
P.B.~also ackowledge the partial travel/accommodation support from the University of Rome II ``Tor Vergata'' (June 2014, May-June 2017) and from the Department of Mathematics and Statistics in Thammasat University.

\end{document}